\DeclareMathAlphabet{\mathpzc}{OT1}{pzc}{m}{it}

\documentclass[headsepline=true]{scrartcl}
\usepackage{amsmath}
\usepackage{amsthm, 	amssymb, amsfonts,bm}
\usepackage{mathtools}
\usepackage[top=1.2in,bottom=1.2in,left=1.2in,right=1.2in]{geometry}
\usepackage[table]{xcolor}
\usepackage{bbm} 

\definecolor{darkblue}{rgb}{0,0,0.7}

\usepackage[pdftex,hyperfootnotes=false,colorlinks=true,linkcolor=darkblue,
citecolor=purple,filecolor=magenta,urlcolor=darkgray,pdfborder={0 0 0}]{hyperref}
\usepackage{array}

\usepackage{abstract}
\usepackage{amscd}
\usepackage{tensor}
\usepackage{mathrsfs}
\usepackage{arydshln}
\usepackage{lscape}
\usepackage{enumerate}
\usepackage[capitalize]{cleveref}
\usepackage[utf8]{inputenc}
\usepackage{autonum}
\usepackage{footnote}

\allowdisplaybreaks[2]

\title{The Bloch--Okounkov theorem for congruence subgroups and Taylor coefficients of quasi-Jacobi forms}
\subtitle{}
\author{Jan-Willem M. van Ittersum%
\thanks{\emph{Email}: \href{mailto:ittersum@mpim-bonn.mpg.de}{ittersum@mpim-bonn.mpg.de}, \newline
Mathematisch Instituut, Universiteit Utrecht, Postbus 80.010, 3508 TA Utrecht, The Netherlands, \newline
Max-Planck-Institut f\"ur Mathematik, Vivatsgasse 7, 53111 Bonn, Germany. \newline
ORCID: 0000-0003-1541-0232
}}

\newcommand{\sltwoz}{\mathrm{SL}_2(\z)}

\newcommand{\LambdaN}{\Lambda^{\!(N)}}

\renewcommand{\mathbf}{\bm}
\renewcommand{\vec}{\bm}
\renewcommand{\phi}{\varphi}
\renewcommand{\Im}{\mathrm{Im}}
\renewcommand{\Re}{\mathrm{Re}}

\newcommand{\n}{\mathbb{N}}
\newcommand{\z}{\mathbb{Z}}
\newcommand{\q}{\mathbb{Q}}
\renewcommand{\r}{\mathbb{R}}
\renewcommand{\c}{\mathbb{C}}

\newcommand{\pdv}[2]{\frac{\partial #1}{\partial #2}}

\theoremstyle{plain} 
\newtheorem{thm}{Theorem}[section]
\newtheorem{lem}[thm]{Lemma} 
\newtheorem{cor}[thm]{Corollary} 
\newtheorem{prop}[thm]{Proposition}

\newenvironment{qedcor}    
  {\pushQED{\qed}\begin{cor}}
  {\popQED\end{cor}}
\newenvironment{qedprop}    
  {\pushQED{\qed}\begin{prop}}
  {\popQED\end{prop}}

\theoremstyle{definition} 
\newtheorem{defn}[thm]{Definition}

\newtheorem{convt}[thm]{Convention}

\theoremstyle{remark}
\newtheorem*{remarkx}{Remark}
\newenvironment{remark}
  {\pushQED{\qed}\remarkx}
  {\popQED\endremarkx}

\newcommand{\modular}{M}

\newcommand{\partitions}{\mathscr{P}}

\newcommand{\Faulhaber}{\mathpzc{F}}

\newcommand{\blok}{\odot}

\newcommand{\Hol}{\mathrm{Hol}}
\newcommand{\Mer}{\mathrm{Mer}}

\newcommand{\ind}{I}
\newcommand{\e}{\boldsymbol{e}}
\newcommand{\F}{\mathcal{F}}
\newcommand{\ii}{\mathrm{i}}

\newcommand{\dtau}{\delta_\tau}
\newcommand{\dd}{\delta}

\newcommand{\MM}{M_{n}(\q)}
\newcommand{\XX}{M_{n,2}(\q)}
\newcommand{\XXZ}{M_{n,2}(\z)}
\newcommand{\XXR}{M_{n,2}(\r)}
\newcommand{\lm}{(\vec{\lambda},\vec{\mu})} 
\newcommand{\lms}{(\lambda,\mu)} 
\newcommand{\lmgamma}{(\vec{\lambda}^\gamma,\vec{\mu}^\gamma)} 
\newcommand{\abcd}{\left(\begin{smallmatrix} a & b \\ c & d \end{smallmatrix}\right)}

\renewcommand{\=}{\: =\: }
\newcommand{\defis}{\: :=\: }
\newcommand{\+}{\,+\,}
\newcommand{\meno}{\,-\,}

\allowdisplaybreaks

\begin{document}
\maketitle

\begin{abstract}
There are many families of functions on partitions, such as the shifted symmetric functions, for which the corresponding~$q$-brackets are quasimodular forms. We extend these families so that the corresponding~$q$-brackets are quasimodular for a congruence subgroup. Moreover, we find subspaces of these families for which the~$q$-bracket is a modular form. These results follow from the properties of Taylor coefficients of strictly meromorphic quasi-Jacobi forms around rational lattice points. 
\end{abstract}

\section{Introduction}
Denote by~$\partitions$ the set of all partitions of integers. For a large class of functions $f\colon\partitions\to \c$ the~\emph{$q$-bracket}, which is defined as the formal power series
\begin{equation}\label{eq:qbrac}\langle f \rangle_q \,:=\, \frac{\sum_{\lambda \in \partitions} f(\lambda)\, q^{|\lambda|}}{\sum_{\lambda\in\partitions} q^{|\lambda|}} \in \c[\![q]\!],\end{equation}
is a quasimodular form for~$\sltwoz$ (here $\lambda$ is a partition of size~$|\lambda|$). It is, therefore, natural to raise the following two questions:
\begin{enumerate}[\upshape (I)]\itemsep2pt
\item\label{q:1} Given a congruence subgroup~$\Gamma\leq \sltwoz$, is there an (even larger) class of functions for which the~$q$-bracket is a \emph{quasimodular form for~$\Gamma$}?
\item\label{q:2} What is the class of functions for which the~$q$-bracket is a \emph{modular form} (and not just a quasimodular form)?
\end{enumerate}
In this paper we explain how to answer these questions by studying a different question of independent interest:
\begin{enumerate}[\upshape (I)]\itemsep0pt
\item[\upshape (III)] What is the modular or quasimodular behaviour of the Taylor coefficients of meromorphic quasi-Jacobi forms?
\end{enumerate}

To illustrate how these questions are related, consider the \emph{shifted symmetric functions} $Q_k\colon{\partitions\to\q}$, for~$k\geq 1$ given by
\begin{equation}\label{def:qk} Q_k(\lambda) \defis \beta_k \+ \frac{1}{(k-1)!}\sum_{i=1}^\infty \bigl((\lambda_i-i+\tfrac{1}{2})^{k-1} - (-i+\tfrac{1}{2})^{k-1}\bigr),\end{equation}
where $\sum_{k=0}^\infty \beta_{k}\, z^{k-1}:= \frac{e^{z/2}}{e^z-1}$.
By the celebrated Bloch--Okounkov theorem, for every homogeneous polynomial in these functions, the~$q$-bracket is a quasimodular form for~$\sltwoz$. 

More precisely, we have the following. Here, and throughout the paper, we let $\tau\in \mathfrak{h}$, the complex upper half plane, $z\in \c$ and write $\e(x):=e^{2\pi\ii x}$, $q=\e(\tau)$. Given $\lambda \in \partitions$, write
\begin{equation}\label{eq:BOgen} W_\lambda(z) \defis \sum_{i=1}^\infty \e((\lambda_i-i+\tfrac{1}{2})z) 
\end{equation}
for the generating series of the functions~$Q_k$ \cite{BO00, Zag16}, which converges for $\Im(z)<0$. Write $W(z_1)\cdots W(z_n)$ for the function $\lambda \mapsto W_\lambda(z_1)\cdots W_\lambda(z_n)$.
Then, by the Bloch--Okounkov theorem, $\langle W(z_1)\cdots W(z_n)\rangle_q$ is a meromorphic quasi-Jacobi form~$F_n\mspace{1mu}$, defined as follows.
\begin{defn}\label{defn:BOnpoint} For all $n\geq 1$, let~$\mathfrak{S}_n$ be the symmetric group on~$n$ letters and 
\[\label{eq:theta} \theta(\tau,z) \defis \sum_{\nu\in \mathbb{F}} (-1)^{\lfloor \nu \rfloor} \, \e(\nu z)\, q^{\nu^2/2} 
\qquad (\mathbb{F}=\mathbb{Z}+\tfrac{1}{2})\]
the \emph{Jacobi theta series}.
Define the \emph{Bloch--Okounkov~$n$-point functions}~$F_n$ by
\[F_n(\tau,z_1,\ldots,z_n) \defis \sum_{\sigma \in \mathfrak{S}_n} V_n(\tau,z_{\sigma(1)},\ldots,z_{\sigma(n)}),\]
where the functions~$V_n$ are defined recursively by~$V_0(\tau)=1$ and
\begin{equation}\label{eq:B-Ocor}\sum_{m=0}^n \frac{(-1)^{n-m}}{(n-m)!}\;\theta^{(n-m)}(\tau,z_1+\ldots+z_m)\;V_m(\tau,z_1,\ldots,z_m)\=0.\end{equation}
Here, $\theta^{(r)}(\tau,z)=\bigl(\frac{1}{2\pi\ii}\pdv{}{z}\bigr)^{\! r}\theta(\tau,z)$ for $r\geq 0$. 
\end{defn}

The study of the Taylor coefficients of~$F_n$ around rational values of~$z_1,\ldots,z_n$ yields an answer to Question~\ref{q:1}, whereas a detailed description of the (quasi)modular transformation of the Taylor coefficients of~$F_n$ around $z_i=0$ answers Question~\ref{q:2}. The Fourier coefficients of the functions~$F_n$ were studied in \cite{BM15}, and the Taylor coefficients of certain functions closely related to holomorphic quasi-Jacobi forms were studied in \cite{Bri18}. However, the \emph{Taylor coefficients} of~$F_n$ or of a more general meromorphic quasi-Jacobi form have not been studied before. 

We now give a short overview of the results in the literature as well as in the present paper. It should be noted that the answer to the third question can be read independently from the applications to functions on partitions.

\subsection{The Bloch--Okounkov theorem for congruence subgroups}
Given~$a\in\q$, define  $Q_k(\cdot,a)\colon\partitions\to\overline{\q}$ by $Q_0(\lambda,a)\defis \beta_{0}(a)$ and for $k\geq 1$ by
\begin{equation}\label{def:qka} Q_{k}(\lambda,a)\defis\beta_{k}(a)\+\frac{1}{(k-1)!}\sum_{i=1}^\infty \bigl( \e(a)^{\lambda_i-i}\,(\lambda_i-i+\tfrac{1}{2})^{k-1} -  \e(a)^{-i}\,(-i+\tfrac{1}{2})^{k-1}\bigr),\end{equation}
where~$\sum_{k\in \z} \beta_k(a)\,(2\pi\ii z)^{k-1} := \frac{\e(z/2)}{\e(z+a)-1}$. We recall $\e(x)=e^{2\pi\ii x}$. The main properties satisfied by these functions are a consequence of the fact that
\[\sum_{k\geq 0} Q_k(\lambda,a) \, z^{k-1} \= \e(-\tfrac{1}{2}a)\,W_\lambda(z+a),\]
and that $F_n(z_1,\ldots,z_n) = \langle W(z_1)\cdots W(z_n)\rangle_q$ is a quasi-Jacobi form. Observe that $Q_k(\lambda,0)=Q_k(\lambda)$. Up to a constant, these functions~$Q_k(\cdot,a)$ have been considered before in \cite{EO06} for $a=\tfrac{1}{2}$ and in \cite{Eng17} for all $a\in \q$. It was shown that a suitably adapted~$q$-bracket of any polynomial in these functions, excluding the function~$Q_1(\cdot,a)$, is quasimodular for~$\Gamma_1(N)$ for some~$N$.

In this work, we will not change the~$q$-bracket, nor exclude any of the functions~(\ref{def:qka}), and nevertheless prove the following result for the graded algebra~$\Lambda^*(N)$, contained in $\c^\partitions$, given by
\begin{align}\label{eq:LambdaN}
\Lambda^*(N) \defis \q\!\left[Q_k(\cdot,a) \mid k\geq 1, a\in \{0,\tfrac{1}{N},\ldots,\tfrac{N-1}{N}\}\right],
\end{align}
with the grading given by assigning weight~$k$ to~$Q_k(\cdot,a)$. For example, a basis for the elements of $\Lambda^*(2)$ of (homogeneous) weight~$3$ is given by 
\begin{align}
\{Q_3\mspace{1mu},\mspace{1mu} Q_3(\cdot,\tfrac{1}{2}),\mspace{1mu} Q_2\mspace{1mu}Q_1\mspace{1mu},\mspace{1mu} Q_2\mspace{1mu}Q_1(\cdot,\tfrac{1}{2}),\mspace{1mu} Q_2(\cdot,\tfrac{1}{2})\mspace{1mu}Q_1 \mspace{1mu},\mspace{1mu}Q_2(\cdot,\tfrac{1}{2})\mspace{1mu}Q_1(\cdot,\tfrac{1}{2}),&\\
\mspace{1mu}Q_1^3,\mspace{1mu}Q_1^2\mspace{1mu}Q_1(\cdot,\tfrac{1}{2}),\mspace{1mu}Q_1\mspace{1mu}Q_1(\cdot,\tfrac{1}{2})^2,\mspace{1mu}Q_1(\cdot,\tfrac{1}{2})^3&\}.
\end{align}
Observe that there are no elements of negative weight in $\Lambda^*(N)$. 

Given $N\geq 1$, write $(2,N)$ for $\gcd(2,N)$. For $\widehat{N}\in \z$, denote $m_{\widehat{N}}=\left(\begin{smallmatrix} \widehat{N} & 0 \\ 0 & 1 \end{smallmatrix}\right)$.
\begin{thm}\label{thm:higherlevel}
Let~$k\in \z, N\geq 1$ and $\widehat{N}=(2,N)N$. For $f\in \Lambda^*(N)$ of weight~$k$, the $q$-bracket~$\langle f \rangle_q$ is a quasimodular form of weight~$k$ for~$m_{\widehat{N}}^{-1}\Gamma(\widehat{N}) m_{\widehat{N}}\mspace{1mu}$. 
\end{thm}
\begin{remark}
The occurrence of the group~$m_{\widehat{N}}^{-1} \Gamma(\widehat{N}) m_{\widehat{N}}$ in the theorem can be interpreted in at least two ways. First of all, 
an equivalent formulation of the theorem is that for $f\in \Lambda^*(N)$ of weight~$k$, the series~$\langle f \rangle_{q_{\widehat{N}}}$ is a quasimodular form of weight~$k$ and level~${\widehat{N}}$, where $q_{\widehat{N}}:=q^{1/{\widehat{N}}}$.
Secondly, as $\Gamma_1({\widehat{N}}^2)\leq m_{{\widehat{N}}}^{-1} \Gamma({\widehat{N}}) m_{{\widehat{N}}}$, it follows that~$\langle f \rangle_q$ is a quasimodular form for~$\Gamma_1({\widehat{N}}^2)$. 

Moreover, the definition of~${\widehat{N}}$ indicates that the behaviour of~$Q_k(a)$ is different when the numerator of~$a$ is divisible by~$2$. We will see that this can also be explained by the $n$-point functions~$F_n\mspace{1mu}$: they are quasi-Jacobi forms for which the index is an element of~$M_n(\tfrac{1}{2}\z)$.
\end{remark}

The following theorem is a refinement of \cref{thm:higherlevel}, giving us~$q$-brackets (or quotients of~$q$-brackets) that are quasimodular forms on~$\Gamma_1(N)$ rather than only on the much smaller group~$m_{{\widehat{N}}}^{-1} \Gamma({\widehat{N}}) m_{{\widehat{N}}}\mspace{1mu}$.

\begin{thm}\label{thm:bohigherlevel}
Let $N\geq 1$. 
Given $k_i\in \z_{>0}\mspace{1mu}, a_i\in \frac{1}{N}\z$ for $i=1,\ldots,n$, denote $a=a_1+\ldots+a_n$ and $Q_{\vec{k}}(\cdot,\vec{a})=Q_{k_1}(\cdot,a_1)\cdots Q_{k_n}(\cdot,a_n)$.
Then,
\begin{itemize}
\item If $a\in \z$, then~$\langle Q_{\vec{k}}(\cdot,\vec{a})\rangle_q$ is a  quasimodular form for~$\Gamma_1(N);$
\item If $a\not\in \z$, then $\displaystyle\frac{\langle Q_{\vec{k}}(\cdot,\vec{a}) \rangle_q}{\langle Q_1(\cdot,a)\rangle_q}$ is a quasimodular form for~$\Gamma_1(N)$.
\end{itemize}
\end{thm}
\begin{remark}
Let $a\in \q$. The function~$\langle Q_1(\cdot,a)\rangle_q^{-1}$, equal to~$\Theta(a)$, is a so-called \emph{Klein form}; see, e.g., \cite{KL81}. 
Note  $\langle Q_1(\cdot,a)\rangle_q=0$ for $a\in \z$.
\end{remark}

\Cref{thm:higherlevel} should be compared with the results in \cite{GJT16}, where Griffin, Jameson and Trebat-Leder consider certain functions~$Q_k^{(p)}$ in the context of studying~$p$-adic analogues of the shifted symmetric functions~(\ref{def:qk}). Extending their definition to composite $m$, we let
\begin{align}\label{def:Qkp}
Q_k^{(m)}(\lambda) \defis Q_k(\lambda) \meno \frac{1}{m}\sum_{a=0}^{m-1} \e\Bigl(\frac{a}{m}\Bigr) \, Q_k\Bigl(\lambda,\frac{2a}{m}\Bigr) \qquad (k,m\geq 1).
 \end{align}
 Write~$\LambdaN$ for the graded $\q$-algebra generated by the functions~$Q_k^{(m)}$ for all $m\mid N$. 
 For primes $p$, these authors show that the~$q$-bracket of the functions $Q_k^{(p)}$ is quasimodular for~$\Gamma_0(p^2)$ and they suggest that it is likely that products of these functions also have quasimodular~$q$-brackets for the same group. Slightly more general, one can wonder whether for $f\in\LambdaN$ the $q$-bracket $\langle f\rangle_q$ is a quasimodular form for $\Gamma_0(N^2)$. Now, observe that the functions~$Q_k^{(m)}$ are elements of~$\Lambda^*(m)$. Hence, $\LambdaN\subset \Lambda^*(N)$. Therefore, for all odd~$N$ \cref{thm:higherlevel} implies that~$q$-brackets of elements of $\LambdaN$ are quasimodular for~$\Gamma_1(N^2)$. That these $q$-brackets are indeed quasimodular for the bigger group~$\Gamma_0(N^2)$ is the content of the next theorem. 

\begin{thm}\label{thm:bohigherlevel2}
Let $N\geq 1$ and $k\in \z$. 
For all homogeneous $f\in \LambdaN$ of weight~$k$ the function~$\langle f \rangle_q$ is a quasimodular form of weight~$k$ for~$\Gamma_0(N^2)$. 
\end{thm}

It should be noted that the above theorems are true in greater generality. For example, the so-called hook-length moments introduced in \cite{CMZ16}, and studied in the context of harmonic Maass forms for a congruence subgroup in \cite{BOW20}, also have natural generalisations obtained by studying their corresponding~$n$-point functions. Similarly, the moment functions and their generalisations in \cite{vI20} can equally well be generalised to congruence subgroups. Therefore, we will state and prove the above results in \cref{sec:set-up} in a more general setting that allows application to the hook-length moments and moment functions.

\subsection{When is the~\texorpdfstring{$q$}{q}-bracket modular?}\label{sec:introwhen}
To illustrate the main ideas, consider again the Bloch--Okounkov algebra $\Lambda^*=\Lambda^*(1)$. For all $k\geq 0$, let $h_k\in \Lambda^*$ be given by
\[h_k(\lambda)\defis \sum_{r=0}^{\lfloor \frac{k}{2}\rfloor} \frac{Q_2(\lambda)^r \,Q_{k-2r}(\lambda)}{2^{r}\,(k-r-\frac{3}{2})_r\,r!} \qquad\qquad \left((x)_n:=x(x+1)\cdots(x+n-1)\right).\]
We show that the function~$h_k$ satisfies the following three properties:
\begin{enumerate}[{\upshape (i)}]\itemsep2pt
\item The difference~$h_k-Q_k$ is divisible by~$Q_2\, $;
\item The~$q$-bracket~$\langle h_k\rangle_q$ is a modular form (and not just a quasimodular form);
\item For $f\in \Lambda^*$ with~$\langle f \rangle_q$ modular and~${f-Q_{k}}$ divisible by~$Q_2\, $, we have $\langle f\rangle_q = \langle h_{k}\rangle_q\, .$ 
\end{enumerate}

Hence, one can think of the difference~${h_k-Q_k}$ as a correction term for~$Q_k$ with respect to the property of being a modular form under the~$q$-bracket. By the third property, this correction term is unique up to elements in the kernel of the~$q$-bracket. More generally, one has the following result. Here, $\modular$ denotes the algebra of modular forms for $\sltwoz$ \emph{with rational Fourier coefficients}.
\begin{thm}\label{thm:when}
For any algebra~$\mathcal{F}$ of functions on partitions satisfying the conditions in \cref{sec:whenconstruction}, there exists a computable subspace $\mathcal{M}=\mathcal{M}(\mathcal{F})\subseteq \mathcal{F}$ such that
\begin{enumerate}[{\upshape (i)}]\itemsep2pt
\item \label{mainit:splitting} $\F = \mathcal{M} \oplus Q_2\F;$
\item $\langle \mathcal{M}\rangle_q \subseteq \modular;$
\item $\langle Q_2\F\rangle_q \cap \modular = \{0\}.$
\end{enumerate}
\end{thm}

The algebra $\mathcal{F}=\Lambda^*$ is an example to which the above result applies. In particular, by using~(\ref{mainit:splitting}) inductively, each $f\in \Lambda^*$ can uniquely be written as 
a polynomial in~$Q_2$ with coefficients $g_i\in\mathcal{M}$, i.e.,
\[f\= \sum_{i\geq 0} g_i\, Q_2^i\,.\]
As $\langle g \rangle_q=0$ if and only if $\langle Q_2\,g\rangle_q=0$ for $g\in \Lambda^*$, the following are equivalent:
\begin{enumerate}[\upshape (a)]\itemsep2pt
\item $\langle f \rangle_q$ is modular;
\item $\langle f \rangle_q = \langle g_0 \rangle_q\,$;
\item $\langle g_i \rangle_q = 0$ for all $i>0$.
\end{enumerate}

The functions~$h_\lambda\in \Lambda^*$ defined in \cite{vI18} form a basis for the space~$\mathcal{M}(\Lambda^*)$. The method of proof in this work (i.e., using quasi-Jacobi forms) allowed us to state \cref{thm:when} for many algebras~$\mathcal{F}$, whereas the results in \cite{vI18} could not easily be generalised to other algebras than~$\Lambda^*$. In \cref{sec:when}, we prove this theorem and apply this result to the algebra~$\Lambda^*$ of shifted symmetric functions, to its extensions to higher levels~$\Lambda^*(N)$ and to the aforementioned algebra of moment functions.

\subsection{Quasi-Jacobi forms and their Taylor coefficients}
Recall $\tau\in \mathfrak{h}$, the complex upper half plane, $z\in \c$ and $\e(x)=e^{2\pi\ii x}, q=\e(\tau)$. 
The \emph{Kronecker--Eisenstein series}~$E_k$ are given by
\[\label{eq:kronecker-eisenstein} E_k(\tau,z) \defis \sum_{m,n\in\z}\hspace{-6pt}{}^{\raisebox{3pt}{\footnotesize $e$}}\hspace{5pt} \frac{1}{(z+m\tau+n)^k}\qquad (k\geq 1) .\]
Here, the letter `e' indicates we perform the \emph{Eisenstein summation} procedure, given by $\sum_{m,n\in\z}\hspace{-28pt}{}^{\raisebox{1pt}{\footnotesize $e$}}\hspace{23pt}:=\lim_{M\to\infty} \sum_{m=-M}^M\bigl(\lim_{N\to\infty} \sum_{n=-N}^N\bigr)$. Note this summation procedure is only important for $k=1$ and $k=2$, as for $k\geq 3$ the series converges absolutely.

Do these series $E_k\mspace{1mu}$, the Jacobi theta function~$\theta$ and its derivatives, and the Bloch--Okounkov $n$-points functions $F_n$ have a common property? 

To answer this question, observe that the first two functions 
can be interpreted as generating functions of the \emph{Eisenstein series}~$e_k\mspace{1mu}$,  given by\footnote{Here, we follow the notation in \cite{W99}, which goes back to Eisenstein. Often, $e_k$ is denoted by $G_k \mspace{1mu}.$}
\[\label{def:eis}\hspace{40pt} e_k(\tau) \defis \! \sum_{\substack{(m,n)\in \z^2\\ (m,n)\neq (0,0)}}\hspace{-15pt}{}^{'}\hspace{10pt} \frac{1}{(m\tau+n)^k} \qquad (k\geq 2).
\]
Namely,
\begin{align}\label{eq:theta=zexp}
\Theta(\tau,z)\defis &\frac{\theta(\tau,z)}{\theta'(\tau,0)} \=  2\pi\ii z \exp\Bigl(-\sum\nolimits_{m\geq 1} \frac{e_{2m}(\tau)}{2m}z^{2m}\Bigr),\\[5pt]
E_k(\tau,z) \= & \frac{1}{z^k} \+ (-1)^k\sum_{m\geq k/2 }\binom{2m-1}{k-1}\, e_{2m}(\tau)\, z^{2m-k}. \end{align}
In particular, the Taylor (or Laurent) coefficients around $z=0$ of these functions are polynomials in Eisenstein series, hence quasimodular forms---a property which is shared by the stronger notion of a Jacobi form. Just as~$e_2$ transforms as a \emph{quasi}modular form, so do~$\theta^{(r)}$ and~$F_n$ transform as \emph{quasi}-Jacobi forms. Our answer to the question is that all these functions are quasi-Jacobi forms, introduced in \cref{sec:quasiJacforms}.

Quasi-Jacobi forms transform comparable to quasimodular forms, as we explain now. There is a slash action (see \cref{defn:slash}) on all functions ${\phi\colon\mathfrak{h}\times\c^n\to\c}$ for all $\gamma \in \sltwoz$ and for all $X\in \XX$ (so actually for the action of their semidirect product $\sltwoz \ltimes \XX$). In case~$\phi$ is a quasi-Jacobi form of weight~$k$ and index~$M$, there exist quasi-Jacobi forms~$\phi_{i,\vec{j}}\mspace{1mu}$, indexed by a finite subset of $\z_{\geq 0}\times \z_{\geq 0}^n\mspace{1mu}$, such that for all $\gamma=\abcd$ and ${X=\lm\in \XXZ}$ one has
\begin{align}\label{eq:action-quasimod}
(\phi|_{k,M}\mspace{2mu}\gamma)(\tau,z_1,\ldots,z_n) &\= \sum_{\substack{i, \vec{j}}} \phi_{i,\vec{j}}(\tau,z_1,\ldots,z_n) \Bigl(\frac{c}{c\tau+d}\Bigr)^{i+|\vec{j}|} \, \frac{z_1^{j_1}\cdots z_n^{j_n}}{(2\pi\ii)^i},\qquad\phantom{.}\\
\label{eq:action-quasiell} (\phi|_M\mspace{2mu}X)(\tau,z_1,\ldots,z_n) &\= \sum_{\vec{j}}\phi_{0,\vec{j}}(\tau,z_1,\ldots,z_n)\,(-\lambda_1)^{j_1}\cdots(-\lambda_n)^{j_n},
\end{align}
together with similar formulas for each $\phi_{i,\vec{j}}|_{k,M}\mspace{2mu}\gamma$ and~$\phi_{i,\vec{j}}|_M\mspace{2mu} X$.
Jacobi forms are quasi-Jacobi forms for which the right-hand side in both equations above equals~$\phi$. 

The quasi-Jacobi forms~$\phi$ we study are \emph{strictly meromorphic}, i.e., meromorphic such that if $\vec{z}\in \r^n\tau+\r^n$ is a pole of $\phi(\tau,\cdot)$ for some $\tau\in \mathfrak{h}$, it is a pole for almost all $\tau\in \mathfrak{h}$. This is a new notion we introduce in this work; see \cref{sec:sm}. The Weierstrass $\wp$-function is an example of a strictly meromorphic Jacobi form, but its inverse is not. Also, $E_k,\Theta^{-1}$ and $F_n$ are strictly meromorphic. If the number of elliptic variables~$n$ satisfies $n=1$, this condition is equivalent to the statement that all poles of~$\phi$ are torsion points $z\in \q\tau+\q$. This is crucial to obtain mock modular forms as Fourier coefficients of meromorphic Jacobi forms (see \cite{Zwe02,DMZ14}). For $n>1$ the Jacobi transformation properties of~$\phi$ imply a more complicated restriction on the positions of the poles. By studying the orbits of the action of~$\sltwoz$ on~$(\r^2)^n$ we prove the following result. 
\begin{thm}\label{thm:poleshyperplane}
Let~$\phi$ be a strictly meromorphic quasi-Jacobi form and $\tau\in \mathfrak{h}$. Then all poles~$\vec{z}$ of~$\phi(\tau,\cdot)$ lie in finite union of rational hyperplanes
\[ s_1z_1+\ldots +s_nz_n \in u\tau+v\]
with $s_1,\ldots,s_n\in \z$ and $u,v\in \q/\z$. 
\end{thm}

In this work, we determine conditions on the Taylor coefficients (or rather Laurent coefficients in case we are expanding around a pole) of a meromorphic function $\phi\colon\mathfrak{h}\times \c^n\to \c$ for it to be a (quasi-)Jacobi form. This is the technical result we need to answer Question~\ref{q:1} and~\ref{q:2}. In one direction, by the work of Eichler and Zagier \cite{EZ85}, it is known that for a Jacobi form the Taylor coefficients around rational lattice points are quasimodular, or equivalently that certain linear combinations~$\xi_\ell^X$ of derivatives of these Taylor coefficients are modular. For example,~$\Theta$ is a weak Jacobi form, hence it satisfies
\[(\Theta|X|\gamma)(\tau,z) \propto (\Theta|X\gamma)(\tau,z)\]
for all $X=(\lambda,\mu)\in M_{1,2}(\q)$ and~$\gamma \in \sltwoz$, where 
the implicit multiplicative constant is a root of unity depending on~$X$ and~$\gamma$. Hence, it follows that the Taylor coefficients of~$\Theta$ around $z=\lambda\tau+\mu$ (after multiplying with a certain power of~$q$) are quasimodular for some subgroup~$\Gamma_X$ of~$\sltwoz$ consisting of~$\gamma$ for which $X\gamma-X\in M_{1,2}(\z)$ (see~(\ref{eq:gammaX})). In contrast, the weak quasi-Jacobi form~$\Theta'(\tau,z)=\frac{\theta'(\tau,z)}{\theta'(\tau,0)}$ transforms as
\begin{align}\label{eq:transftheta'}(\Theta'|X|\gamma)&(\tau,z) \,\propto \\
&(\Theta'|X\gamma)(\tau,z) \+ \frac{cz}{c\tau+d}(\Theta|X\gamma)(\tau,z)\+\lambda \,(\Theta|X\gamma)(\tau,z) \meno \frac{\lambda}{c\tau+d}(\Theta|X\gamma)(\tau,z).\end{align}
Therefore, the Taylor coefficients of~$\Theta'+\lambda\mspace{1mu} \Theta$ around~$z=\lambda\tau+\mu$, rather than of~$\Theta'$, give rise to quasimodular forms for~$\Gamma_X\mspace{1mu}$. We write $g_\ell^X(\Theta')$ for the $\ell$th Taylor coefficient of~$\Theta'+\lambda\mspace{1mu} \Theta$. In \cref{defn:Taylor} we extend this notation by writing $g_\ell^X(\phi)$ for the `correct' Taylor coefficient of a quasi-Jacobi form~$\phi$.  Moreover, in~\eqref{eq:xi} we define $\xi^{X}_{\vec{m}}(\phi)$ to be a certain combination of the derivatives of~$g_{\vec{\ell}}^X(\phi)$. 

 The main result on Taylor coefficients of quasi-Jacobi forms is given by \cref{thm:main} and summarized in the following result. For simplicity, we assume that $\vec{s}=(s_1,\ldots,s_n)$ in \cref{thm:poleshyperplane} is always a standard basis vector, e.g., we allow $\wp(z_1)\wp(z_2+\tfrac{1}{2})$, but we do not allow $\wp(z_1-z_2)$. 

\begin{thm}\label{thm:1}
Let~$\phi$ be a strictly meromorphic quasi-Jacobi form of weight~$k$ and index~$M$ whose poles~$\vec{z}$ lie on a finite collection of hyperplanes of the form
$ z_j \in u\tau+v$
with $j\in \{1,\ldots,n\}$ and $u,v\in \q/\z$. Then
\begin{enumerate}[\upshape (i)]
\item\label{mainit:mt1} for all~$X\in \XX$ and~$\vec{\ell}\in \z^n$ the `Taylor coefficients'~$g_{\vec{\ell}}^X(\phi)$ 
are quasimodular forms of weight~$k+\ell_1+\ldots+\ell_n$ for the group~$\Gamma_X$ and satisfy the functional equations~\eqref{eq:ellcoeff} and~\eqref{eq:glX}. 
\item\label{mainit:mt2} for all~$X\in \XX$ and~$\vec{m}\in \z^n$ the functions~$\xi^{X}_{\vec{m}}(\phi)$
are modular forms of weight~$k+m_1+\ldots+m_n$ for~$\Gamma_X$ and satisfy the functional equation~\eqref{eq:ellcoeff}.
\end{enumerate}
\end{thm} 
\Cref{defn:Taylor} and Equation~\eqref{eq:xi} (defining $g_{\vec{\ell}}^X(\phi)$ and $\xi^{X}_{\vec{m}}(\phi)$) use the functions~$\phi_{i,\vec{j}}$ in~\eqref{eq:action-quasimod} and~\eqref{eq:action-quasiell}, which are uniquely determined by the quasi-Jacobi form~$\phi$. The results of \cref{ch:Taylor} also show that the four collections of functions $\{\phi\}$, $\{\phi_{i,\vec{j}}\}$, $\{g_{\vec{\ell}}^X\}$ and $\{\xi^{X}_{\vec{m}}\}$ determine each other in a computable way, and give explicit conditions on the collections $\{\phi_{i,\vec{j}}\}$, $\{g_{\vec{\ell}}^X\}$ and $\{\xi^{X}_{\vec{m}}\}$ that imply that they arise from a quasi-Jacobi form. 

In \cref{sec:qjf}, we introduce strictly meromorphic quasi-Jacobi forms and prove \cref{thm:poleshyperplane} and \cref{thm:1}. The quasimodular transformation of the Taylor coefficients 
 of~$F_n$ will then imply the results on the Bloch--Okounkov theorem for congruence subgroups in \cref{sec:cs}. Moreover, in \cref{sec:when}, we see that ``pulling back'' the functions~$\xi^X_{\vec{m}}$ under the~$q$-bracket leads to the construction of the functions~$h_{\vec{m}}$ for which the~$q$-bracket is modular.

\subsection*{Acknowledgement}
I am very grateful for the encouragement and feedback received from both my supervisors Gunther Cornelissen and Don Zagier, especially taking into consideration the difficulties to meet each other in person. Also, I want to thank Georg Oberdieck for introducing me to quasi-Jacobi forms, and Frits Beukers and Oleg German for some interesting remarks on the Diophantine approximation problem for $\sltwoz$. 


\section{Quasi-Jacobi forms}\label{sec:qjf}
The definition of a Jacobi form in \cite{EZ85} has been generalised in many ways. We provide a generalization that incorporates several elliptic variables, characters, weakly holomorphic and meromorphic functions, and quasi-Jacobi forms \cite{Hat15,Lib11, OP19}. In particular, the normalized Jacobi theta function~$\Theta$, the Kronecker--Eisen\-stein series~$E_k$ and the Bloch--Okounkov~$n$-point functions~$F_n$ from the introduction will be examples throughout this paper. 

For $\tau\in \mathfrak{h}$, the complex upper half plane, we write $L_\tau=\z\tau+\z$. As is customary, we often omit the dependence on the modular variable~$\tau$ in any type of Jacobi form, e.g.\ we write~$\Theta(z)$ for~$\Theta(\tau,z)$. We write $\vec{a}=(a_1,\ldots,a_n)$ for a vector of elements and we write~$|\vec{a}|$ for~$a_1+\ldots+a_n\mspace{1mu}.$  \label{page: notation} For vectors~$\vec{a}$ and~$\vec{b}$, we write~$\vec{a}^{\vec{b}}$ to denote~$\prod_{i} a_i^{b_i}$.  Also, given $X\in \XXR$, we denote the rows of $X$ by $\vec{\lambda}$ and $\vec{\mu}$, i.e., $X=\lm$. Moreover, for $\gamma \in \sltwoz$, we write $\gamma=\abcd$ and we write $\vec{\lambda}^\gamma$ and $\vec{\mu}^\gamma$ for the rows of $X\gamma$, i.e., $X\gamma = \lmgamma.$ 


\subsection{Strictly meromorphic Jacobi forms}\label{sec:sm}
The definition of a strictly meromorphic Jacobi form is subtle, excluding many meromorphic functions transforming as Jacobi forms. For example, the~$j$-invariant, the reciprocal of an Eisenstein series~$e_k\mspace{1mu}$, or $\wp/\Delta$, where $\wp$ is the Weierstrass~$\wp$-function and~$\Delta$ the modular discriminant, are all \emph{non-examples} of strictly meromorphic Jacobi forms. Namely, although a strictly meromorphic Jacobi form is meromorphic we want its Taylor coefficients in the elliptic variables to be holomorphic (rather than weakly holomorphic or meromorphic) quasimodular forms. But with this not everything has been said, the definition is even stricter: we require the poles of~$\phi$ to be ``constant'' in the modular variable~$\tau$. Consider, for example, the Weierstrass~$\wp$-function, which is an example of a strictly meromorphic Jacobi form. 
For every fixed $z\in \c$ (e.g., $z=\ii$) the function~$\wp(\tau,z)$ is a meromorphic function of~$\tau$, as~$\wp(\tau,z)$ has a pole whenever~$\tau$ is such that $z\in L_\tau$ (e.g., $\tau=z=\ii$). However, for $\lambda,\mu\in \r$, the function $\wp(\tau,\lambda\tau+\mu)$ is holomorphic, unless both $\lambda\in \z$ and $\mu\in \z$. That is, all poles of $\wp(\tau,z)$ are given by $z\in L_\tau\mspace{1mu}$. On the contrary, we will see that $\wp^{-1}$ is not a strictly meromorphic Jacobi form, as the poles of~$\wp^{-1}$ are not ``constant'' in~$\tau$. 

Before introducing strictly meromorphic Jacobi forms, we first recall the Jacobi group and its action on (meromorphic) functions.

\begin{defn} 
For all~$n\in \n$, the \emph{(discrete) Jacobi group}~$\Gamma^J_n$ of rank~$n$ is defined as the semi-direct product~${\Gamma^J_n := \sltwoz \ltimes \XXZ}$ with respect to the right action of~$\sltwoz$ on~$\XXZ$. 
\end{defn}
That is, an element of~$\Gamma_n^J$ is a pair~$(\gamma,X)$ with~$\gamma \in \sltwoz, X\in \XXZ$ and satisfies the group law~$(\gamma,X)(\gamma',X')=(\gamma\gamma',X+X\gamma')$.

Let $M\in \MM$. We often make use of the associated 
bilinear form $B_M$, given by\footnote{It is standard to let $Q_M(\vec{z})=\tfrac{1}{2}\vec{z}M\vec{z}^t$ be the associated quadratic form. We refrain from using this notation to avoid a clash with the shifted symmetric functions denoted by $Q_k\mspace{1mu}$.}
\[
B_M(\vec{z},\vec{z}') = \vec{z} M \vec{z}'^t.\]

\begin{defn}\label{defn:slash} Given a meromorphic function~$\phi\colon\mathfrak{h}\times \c^n \to \c$, $k\in \z$ and $M\in \MM$, for all $(\gamma,X)\in \Gamma_n^J$ we let
\begin{enumerate}[\upshape(i)]\itemsep10pt
\item $\displaystyle(\phi|_{k,M}\mspace{2mu} \gamma)(\tau,\vec{z}):= (c\tau+d)^{-k}\,\e\Bigl(\frac{-c\, B_M(\vec{z},\vec{z})}{c\tau+d}\Bigr)\,\phi\Bigl(\frac{a \tau+b}{c \tau +d}, \frac{\vec{z}}{c\tau+d}\Bigr)$;
\item\label{it:transfo2} $\displaystyle (\phi|_{M}\,X)(\tau, \vec{z})  :=\e(B_M(\vec{\lambda}+\vec{\mu},\vec{\lambda}+\vec{\mu}))\,\e(B_M(\vec{\lambda},\vec{\lambda}\tau+2\vec{z}))\, \phi(\tau,\vec{z}+\vec{\lambda}\tau+\vec{\mu})$.
\end{enumerate}
(Recall that we write $\gamma = \abcd$ and $X=\lm$.)
Moreover, we let $\phi|_{k,M}(\gamma,X) := (\phi|_{k,M}\mspace{2mu}\gamma)|_M\mspace{2mu}X$, often omitting~$k$ and~$M$ from the notation.
\end{defn} 
\begin{remark}
Given $k\in \z$ and $M\in \MM$, the slash operator defines an action of~$\Gamma_n^J$ of \emph{weight}~$k$ and \emph{index}~$M$ on the space of all meromorphic functions $\phi\colon{\mathfrak{h}\times \c^n \to \c}$. 
\end{remark}
Given $M\in \MM$, for~$X,X' \in \XX$, we let
 \begin{equation}\label{def:rhoXzetaXX}  \rho(X)\defis \e(B(\vec{\lambda},\vec{\lambda})-B(\vec{\lambda},\vec{\mu})+B(\vec{\mu},\vec{\mu})),  \quad
\zeta_{X,X'} \defis \e(B(\vec{\lambda'},\vec{\mu})-B(\vec{\lambda},\vec{\mu'})),	
 \end{equation}
 where we wrote $B=B_M$ for the bilinear form associated to $M$. Observe that
\[ \rho(-X)=\rho(X)^{-1} \qquad \text{and} \qquad \zeta_{X',X} = \zeta_{X,X'}^{-1} = \zeta_{-X,X'} = \zeta_{X,-X'}\mspace{1mu}.\] 
  By extending the slash action to the real Jacobi group
, generalizing \cite[Theorem~1.4]{EZ85} to several variables and half-integral index, we obtain the following functional equations.
\begin{qedprop}\label{prop:app} Given a meromorphic function~$\phi\colon\mathfrak{h}\times \c^n \to \c$, $k\in \z$ and $M\in \MM$, for all~$X,X' \in \XXR$ and $\gamma\in \sltwoz$ one has
\begin{equation} \rho(-X)\,\phi|X|\gamma = \rho(-X\gamma)\,\phi|\gamma|X\gamma \end{equation}
and
\begin{align}
 \rho(-X)\rho(-X')\,\zeta_{X',X}\,\phi|X|X'&=\rho(-X')\rho(-X)\,\zeta_{X,X'}\,\phi|X'|X \\ &= \rho(-X-X')\, \phi|(X+X').\qedhere \end{align}
\end{qedprop}

Classical modular forms are defined as the invariants for a certain group action of the space $\Hol_0(\mathfrak{h})$ of holomorphic functions in~$\mathfrak{h}$ satisfying a certain growth condition (having at most polynomial growth near the boundary). 
\begin{defn}
Let~$\mathrm{Hol}_0(\mathfrak{h})$ be the ring of holomorphic functions~$\phi$ of moderate growth on~$\mathfrak{h}$, i.e.\ for all~$C>0$, $\gamma\in \sltwoz$ and $x\in \r$ one has $\phi(\gamma(x+\ii  y))=O(e^{Cy})$ as~$y\to \infty$ (where $\gamma$ acts on $\mathfrak{h}$ by M\"obius transformations).
\end{defn}

With the remarks at the beginning of this section in mind, we now define strictly meromorphic Jacobi forms. A final subtlety in the definition below is coming from the fact that a meromorphic function in two or more variables always has points of indeterminacy (think of $x/y$ near the origin, whose limiting value depends on the angle of approach). Points of indeterminacy are not ``generic'', and we exclude these points when we say, for instance, that a certain function $\phi(\tau,\vec{z})$ has its poles precisely on certain hyperplanes for all \emph{generic} $\tau\in \mathfrak{h}$. 
\begin{defn}
Given $n\geq 0$, denote by~$\Mer_n$ the space of meromorphic functions $\phi\colon{\mathfrak{h}\times\c^n\to \c}$ such that for all $\vec{\lambda},\vec{\mu}\in \r^n$ either $\vec{z}=\vec{\lambda}\tau+\vec{\mu}$ is a pole of~$\phi(\tau,\cdot)$ for all generic $\tau\in \mathfrak{h}$ or the function $\tau\mapsto \phi(\tau,\vec{\lambda}\tau+\vec{\mu})$ belongs to~$\Hol_0(\mathfrak{h})$. Moreover, given $M\in \MM$, denote by~$\Mer_n^M$ the subspace of $\phi\in \Mer_n$ for which $\phi|_{M}X\in \Mer_n$ for all $X\in \XX$. 
Let~$\Hol_n$ and $\Hol_n^M$ be the subspace in~$\Mer_n$ and $\Mer_n^M$, respectively, of holomorphic functions. 
\end{defn}
\begin{defn}\label{def:jf} Let~$k\in \z$ and $M\in \MM$. 
A \emph{holomorphic}, \emph{weak}, or a \emph{strictly meromorphic Jacobi form} of weight~$k$, index~$M$ and rank~$n$ for the Jacobi group~$\Gamma^J_n$ is a function~$\phi$ in $\Hol_n^M$, $\Hol_n$ or $\Mer_n^M$, respectively, that is invariant under the action of~$\Gamma^J_n$ of weight~$k$ and index~$M$ (i.e., $\phi|_{k,M}\mspace{2mu} g = \phi$ for all $g\in \Gamma^J_n$). We write $J_{k,M}^{\mathrm{hol}}, J_{k,M}^{\mathrm{weak}}$ and $J_{k,M}^{\mathrm{sm}}$ for the vector spaces of {holomorphic}, {weak}, and {strictly meromorphic Jacobi forms} of weight~$k$ and index~$M$ (often omitting the indices). 
\end{defn}
\begin{remark}
Let $M=(m_{ij})\in \MM$. Any space of Jacobi forms is trivial whenever $2 B_M(\vec{z},\vec{z})$ is a non-integral quadratic form, or, equivalently, when $m_{ij}\not \in \frac{1}{4}\z$ or $m_{ii}\not \in \frac{1}{2}\z$ for some $i\neq j$. Namely, let~$\phi$ be a Jacobi form and let $\tau\in \mathfrak{h}$ be fixed. If~$\phi$ is non-zero of rank~$1$, write $M=(m)$ for its index. It follows from the elliptic transformation law (\cref{defn:slash} (\ref{it:transfo2})) that the number of zeros minus the number of poles of $z\mapsto \phi(\tau,z)$ in any fundamental domain for the action of $L_\tau$ on $\c$ is exactly~$2m$. For Jacobi forms of higher rank the integrability of~$2 B_M(\vec{z},\vec{z})$ follows by noting that for fixed $\mu_2,\ldots,\mu_n\in \c$, functions of the form $z\mapsto \phi(\tau,z,\mu_2,\mu_3,\ldots,\mu_n)$ and $z\mapsto \phi(\tau,z,z,\mu_3,\ldots,\mu_n)$ still satisfy the elliptic transformation law.
\end{remark}

A holomorphic Jacobi form of rank~$0$ is just a modular form. More interestingly, the Kronecker--Eisenstein series~(\ref{eq:kronecker-eisenstein}) are examples of a strictly meromorphic Jacobi form of index~$(0)$, with expansions given by
\[E_k(\tau,z) \= \frac{(-1)^k}{(k-1)!}D_y^{k-2}\biggl(\,\sum_{m\in \z} \frac{yq^m}{(1-yq^m)^2} \biggr) \qquad\qquad \Bigl(y=\e(z), D_y= y\pdv{}{y}\Bigr). \]

Closely related are the Weierstrass~$\wp$-function and its derivative, that is, $\wp=E_2-e_2$ and $\wp'=-2E_3$.
By \cite[Theorem~9.4]{EZ85} it follows that the algebra of weak Jacobi forms is given by \[J^{\mathrm{weak}} = \c[A,B,C,e_4,e_6]/(C^2-4AB^3+60e_4A^3B+140e_6A^4),\]
where~$A, B$ and~$C$ are equal to~$\Theta^2, \wp\,\Theta^2$ and~$\wp'\,\Theta^4$ respectively. Note that the relation $C^2=4AB^3-60e_4A^3B-140e_6A^4$ comes from the differential equation satisfied by the Weierstrass~$\wp$-function. 

The following result, yielding an algebraic proof and extending \cite[Proposition~2.8]{Lib11}\footnote{The author states the result for \emph{Jacobi forms}, which obviously is not meant to be holomorphic Jacobi forms. Although not stated explicitly, we assume he refers to strictly meromorphic Jacobi forms with poles only at lattice points.}, gives all strictly meromorphic Jacobi forms of rank~$1$ with only poles at the lattice points. The corresponding algebra is free, as the relation between~$\wp^3$ and~$(\wp')^2$ can be used to express~$e_6$ in terms of the generators. As usual, we write~$m$ (instead of the matrix $M=(m)$) for the index of a Jacobi form of rank~$1$.
\begin{prop}\label{thm:mer0} Let~$\phi \in J^{\mathrm{sm}}$
be of index $m\in\tfrac{1}{2}\z_{\geq 0}$ such that all poles~$(\tau,z)$ of $\phi$ satisfy~$z\in L_\tau\mspace{1mu}$. Then,
\[\phi\in\c[\wp,\wp',e_4]\,\Theta^{2m}.\]
\end{prop}
\begin{proof}
First, we show that~$f=\phi\,\Theta^{-2m}\in J^{\mathrm{sm}}$ is a  strictly meromorphic Jacobi form of index~$0$ with all poles~$(\tau,z)$ satisfying~$z\in L_\tau\mspace{1mu}$. This follows from the claim that $\Theta^{-1}\in J^{\mathrm{sm}}_{1,-1/2}$ and that all the poles of $\Theta^{-1}$ are at the lattice points.
To prove this claim, note that by the Jacobi triple product
\[\Theta \= (y^{1/2}-y^{-1/2})\prod_{n\geq 1} \frac{(1-yq^n)(1-y^{-1}q^n)}{(1-q^n)^2}.\]
It follows that~$\Theta$ is a weak Jacobi form with all zeros at the lattice points~$z\in L_\tau\mspace{1mu}$. Moreover, for all $X\in \XX$ the function~$\Theta|X$ does not vanish at infinity, from which the claim follows. 

From now on, assume that~$\phi$ is of index~$0$, i.e., that $\phi$ is an elliptic function. 
Write~$\phi=\phi_0 + \phi_1$ with~$\phi_0$ and~$\phi_1$ the even and odd part of~$\phi$ respectively. For~$u,v\in \r\tau+\r$, write~$u\sim v$ if~$u \equiv v$ or~$u\equiv-v \mod L_\tau\mspace{1mu}$. Then, for~$i=0,1$ one has
\[\phi_i=(\wp')^i\prod_{j} (\wp-\wp(u_j(\tau)))^{m_j}\]
where~$m_j\in \z$ and~$u_j(\tau)$ are representatives with respect to the above equivalence relation for the zeros and poles of~$\phi_i$ outside~$L_\tau\mspace{1mu}$. As both~$\phi_0$ and~$\phi_1$ do not admit poles outside the lattice, it follows that~$m_j>0$. Hence,~$\phi$ is a polynomial in~$\wp$ and~$\wp'$ where the coefficients are polynomials in the functions~$\wp(u_j(\tau))$. By the modular transformation every such coefficient is a modular form for~$\sltwoz$, hence an element of~$\c[\wp,\wp',e_4]$. 
\end{proof}

\begin{remark}
Although the above result and many examples of strictly meromorphic Jacobi forms in the literature have only poles at~$z\in L_\tau\mspace{1mu},$ one easily constructs a strictly meromorphic Jacobi forms with poles at different places. Namely, if $\phi$ is a Jacobi form with all poles at $z\in L_\tau\mspace{1mu}$, then 
\[ \phi(\tau,z+\tfrac{1}{2})\+ \phi(\tau,z+\tfrac{1}{2}\tau)\+ \phi(\tau,z+\tfrac{1}{2}\tau+\tfrac{1}{2})\]
is a Jacobi form for the same group, but now with the poles at~$\frac{1}{2},\frac{1}{2}+\frac{1}{2}\tau$ and~$\frac{1}{2}\tau$ modulo the lattice~$L_\tau\mspace{1mu}$. 
\end{remark}

\subsection{Poles of Jacobi forms}
In contrast to the space of (weakly) holomorphic Jacobi forms, the space of strictly meromorphic Jacobi forms of given weight and index is not finite-dimensional. However, the latter space is not far from being finite-dimensional. First of all, in contrast to the space of all meromorphic functions, the space of strictly meromorphic Jacobi forms is not a field. Moreover, the poles lie in a finite number of hyperplanes and after fixing finitely many such hyperplanes to contain the poles, the vector space of strictly meromorphic Jacobi forms of given weight and index is finite-dimensional, as we will explain in this section.

Given a meromorphic Jacobi form~$\phi$ of rank~$n$, we write 
\[P_\phi=\{(\tau,\vec{z}) \in \mathfrak{h}\times\c^n \mid \phi \text{ is not holomorphic at } (\tau,\vec{z})\}\]
for the set of poles as well as points of indeterminacy of~$\phi$. We identify two points of~$P_\phi$ if they have same image under the projection $\mathfrak{h} \times \c^n \twoheadrightarrow \XXR$ given by $(\tau,\vec{\lambda}\tau+\vec{\mu})\mapsto \lm$. That is, we define an equivalence relation on~$P_\phi$ by saying that $(\tau,\vec{z})\sim (\tau',\vec{z}')$ whenever, after writing $\vec{z}=\vec{\lambda}\tau+\vec{\mu}$ and $\vec{z}'=\vec{\lambda}'\tau+\vec{\mu}'$ with $\vec{\lambda},\vec{\lambda}',\vec{\mu},\vec{\mu}'\in \r^n$, one has $\vec{\lambda}= \vec{\lambda}'$ and $\vec{\mu}=\vec{\mu}'$. 
We identify the quotient set~$Q_\phi$ with a subset of $\XXR$ by identifying a point of~$P_\phi$ with its image under the projection. From the definition of a strictly meromorphic Jacobi form we obtain the factorisation $\mathfrak{h}\times Q_\phi \simeq P_\phi$ of $P_\phi$, given by $(\tau,\lm)\mapsto (\tau,\vec{\lambda}\tau+\vec{\mu})$. Note that $Q_\phi$ is invariant under translation by $\XXZ$.

As an example of how the definition works, we first prove a simple consequence: 
\begin{prop}
Let~$\phi\in J^{\mathrm{sm}}\mspace{1mu}$. 
If~$\frac{1}{\phi}$ is also a strictly meromorphic Jacobi form, then~$\phi$ is constant.
\end{prop}
\begin{proof}
Let $X=(\vec{\lambda},\vec{\mu})\in M_{1,2}(\r)$ be given with $X\not \in Q_\phi$ and $X\not \in Q_{1/\phi}\mspace{1mu}$. Write $\vec{z}(\tau)=X(\tau,1)^t=\vec{\lambda}\tau+\vec{\mu}$. Then both~$\phi(\tau,\vec{z}(\tau))$ and~$\frac{1}{\phi(\tau,\vec{z}(\tau))}$ are holomorphic as a function of~$\tau\in \mathfrak{h}$. Hence, as a function of~$\tau\in \mathfrak{h}$, both~$\phi(\tau,\vec{z}(\tau))$ and~$\frac{1}{\phi(\tau,\vec{z}(\tau))}$ do not admit any zeros. Similarly, 
it follows that both~$\phi(\tau,\vec{z}(\tau))$ and~$\frac{1}{\phi(\tau,\vec{z}(\tau))}$ are holomorphic (as a function of $\tau\in \mathfrak{h}$) at the cusps. Hence both functions do not admit any zero at the cusps. In other words, $\phi(\tau,\vec{z}(\tau))$ does not admit any zeros and poles on a compact set. So, $\phi(\tau,\vec{z}(\tau))$ is constant as a function of~$\tau$. As this holds for almost all~$X\in \XXR$, we conclude that~$\phi$ is globally constant. 
\end{proof}

The following result, which is crucial for the sequel, tells us that the image of~$Q_\phi$ in the torus~$M_{n,2}(\r/\z)$ consists of finitely many hyperplanes given by linear equations with rational coefficients. In other words, the following result strengthens \cref{thm:poleshyperplane} (the fact that the second conclusion is also true for \emph{quasi}-Jacobi forms, follows immediately after introducing such functions (see \cref{cor:poleshyperplane})).
\begin{thm}\label{prop:finitepoles} Let~$\phi\in J^{\mathrm{sm}}\mspace{1mu}$, and
let~$P_\phi \simeq \mathfrak{h}\times Q_\phi$ be the set of non-holomorphic points as above. Then, we have:
\begin{enumerate}[\upshape (i)]\itemsep2pt 
\item\label{it:finite-ii} If~$X\in Q_\phi$, then~$X\gamma \in Q_\phi$ for all~$\gamma \in \sltwoz$.
\item\label{it:finite-iii} There exist finitely many hyperplanes of the form
\begin{equation}\label{eq:linfunc} \vec{s}\cdot X \in (u,v)\end{equation}
with $\vec{s}\in \z^n$ primitive {\upshape(}i.e., with coprime entries{\upshape)} and $u,v\in \q/\z$, such that $X\in Q_\phi$ precisely if~$X$ lies on such a hyperplane. 
\end{enumerate}
\end{thm}
\begin{proof}
(\ref{it:finite-ii}): Let $X\in Q_\phi$ and $\gamma=\abcd\in \sltwoz$ be given. Write $\vec{x}(\tau)=X(\tau,1)^t=\vec{\lambda}\tau+\vec{\mu}$. Note that $\phi$ has a pole at $\vec{x}(\tau)$ for generic $\tau\in \mathfrak{h}$. By the modular transformation behaviour of $\phi$ under~$\gamma^{-1}$, it follows that \[\phi\Bigl(\frac{d\tau-b}{-c\tau+a},\frac{\vec{z}}{-c\tau+a}\Bigr)\] also has a pole at $\vec{z}=\vec{x}(\tau)$ for generic~$\tau\in \mathfrak{h}$. Now, let $\tau'=\gamma^{-1}\tau$. We find
\[\frac{\vec{x}(\tau)}{-c\tau+d}\=\frac{\vec{x}(\gamma\tau')}{-c\gamma\tau'+a}\=\frac{\vec{\lambda}\gamma\tau'+\vec{\mu}}{-c\gamma\tau'+a} \= \vec{\lambda}(a\tau'+b)+\vec{\mu}(c\tau'+d) \=  \lm \gamma\, (\tau',1)^t.\]
Hence, the function
$\phi\left(\tau,\vec{z}\right)$ has a pole at~$\vec{z}=X\gamma\,(\tau,1)^t$ for generic $\tau\in \mathfrak{h}$. 

(\ref{it:finite-iii}): 
Let $X=\lm\in Q_\phi\mspace{1mu}$. The set~$Q_\phi$ is closed and~$(n-1)$-dimensional, hence there is a $Y\in \XXR$ with $Y\not \in Q_\phi$ and which is bounded away from~$X\gamma$. 

First, we treat the case that the rank~$n$ is~$1$. 
Then, by \cref{lem:approx1}, proven in the next section, the fact that $Y$ is bounded away from $X\gamma$ for all $\gamma\in \sltwoz$ implies that both~$\lambda$ and~$\mu$ are rational. Therefore, we can take $s=1$ and $(u,v)=(\lambda,\mu) \bmod \z^2$.

Next, let $n\geq 2$. 
By an approximation property proven in the next section (see \cref{prop:approx}), the fact that $Y$ is bounded away from $X$ implies that there are non-trivial $\vec{s}\in \z^n, t\in\z$ 
with 
$\vec{s}\cdot \vec{\lambda}=t$. 
(In fact, the result states that if $Y$ is bounded away from $X$ then for all $\alpha,\beta\in \z$, there exist non-trivial $\vec{s}\in \z^n, t\in\z$ with $\vec{s}\cdot (\alpha \vec{\lambda}+\beta\vec{\mu})=t$. We choose $\alpha=1,\beta=0$.)
Note that the value of~$\vec{s}$ and  $t$ is a function of~$X$, i.e., for all elements $X\in Q_\phi$ there exists an element $(\vec{s},t) \in \mathrm{Rel}$ giving the relation, where
\[\mathrm{Rel} \defis  (\z^n\backslash\{\vec{0}\})\times \z.\]

Let $\tau\in \mathfrak{h}$ be given and identify~$\XXR$ with~$\c^n$ via $X\mapsto \vec{x}=X(\tau,1)^t$. For almost all $X\in Q_\phi\mspace{1mu}$, the conditions of the implicit function theorem for the function~$\phi$ are satisfied. In this case, there exists an open set $U\subset \c^{n-1}$ containing~$\vec{0}$, an open neighbourhood $V\subset \c^n$ of~$\mathbf{x}$, and a holomorphic function $g:U\to V$, such that $\vec{z}\in V$ is a pole of~$\phi$ precisely if~$\vec{z}$ lies in the image of~$g$. Consider the isomorphism $(\mathrm{\pi}_\tau,\mathrm{\rho}_\tau):\c\simeq \r\tau+\r$. Let~$g_i$ be the $i$th component function of~$g$, which takes values in $\r\tau+\r$ under this isomorphism.  For example, taking $\tau=\ii$, we find $\mathrm{\pi}_{\ii}(g_i)=\Im(g_i)$ and $\rho_{\ii}(g_i)=\Re(g_i)$. For every $\vec{u}\in U$ we find a relation of the form
\begin{align}\label{eq:rel}\sum_{i=1}^n s_i\, \pi_\tau(g_i(\vec{u}))\=t,\end{align}
where $(\vec{s},t)\in\mathrm{Rel}$, possibly depending on the choice of~$\vec{u}$. We now show that we can choose $(\vec{s},t)\in \mathrm{Rel}$ such that~\eqref{eq:rel} holds for all $\vec{u}\in U$.  Recall that the set of zeros of a non-constant real-analytic function has measure~$0$. As~$\pi_\tau(g_i)$ 
is a real-analytic function, either~(\ref{eq:rel}) holds for all $\vec{u}\in U$, or it holds for a real subspace of~$U$ of measure~$0$. Now, note that~$\mathrm{Rel}$ is a countable set, whereas countably many subspaces of measure~$0$ do not cover~$U$. Hence, we find a relation~(\ref{eq:rel}) which holds for all $\vec{u}\in U$. By the Cauchy--Riemann equations for the holomorphic functions~$g_i$ we can `upgrade' this relation to the statement that $\vec{s}\cdot \vec{z}$ takes a constant value $u\tau+v$ in $\r\tau+\r$ for all $\vec{z}\in g(U)$. 
Without loss of generality we assume that $\gcd(s_1,\ldots,s_n)=1$; if not, we scale~$\vec{s}$,~$u$ and~$v$ appropriately. 

Now, by the Weierstrass preparation theorem, locally around~$\vec{x}$ the set of poles is given by~$k$ branches coming together, where each branch is an~$(n-1)$-dimensional space given by the zeros of a holomorphic function and~$k$ equals the multiplicity of the pole at~$\vec{x}$ (see \cite[Lemma~6.1]{CJ12} for the same argument in a different setting). We know that almost all (hence all) the elements in such a branch satisfy $\vec{s}\cdot \vec{x}=u\tau+v$. Write $T$ for $(\r\tau+\r)/(\z\tau+\z)$. By analytically extending such a local branch, we find that all solutions~$\vec{z}\in T$ of $\vec{s}\cdot\vec{z}= u\tau+v$ are poles of $\phi(\tau,\cdot)$ as long as they are in the same connected component as~$\vec{x}$. Because $\gcd(s_1,\ldots,s_n)=1$ the solution space of $\vec{s}\cdot \vec{z}=u\tau+v$ in $T$ is connected, so that all solutions~$\vec{z}$ of $\vec{s}\cdot\vec{z}\equiv u\tau+v \bmod L_\tau$ are poles of $\phi(\tau,\cdot)$ for generic $\tau\in \mathfrak{h}$. 
Moreover, $(u,v)\in \q^2$, as the image of the poles of $\phi$ in $T$ is not dense, but rather~$(n-1)$-dimensional. 

Note that~$\vec{z}\mapsto\phi(\vec{z})\,\Theta(\vec{s}\cdot \vec{z}-u\tau-v)$ has exactly the same poles as~$\vec{z}\mapsto\phi(\vec{z})$ except for the poles which are zeros of~$\vec{s}\cdot \vec{z}\equiv u\tau+v\bmod L_\tau\mspace{1mu}$. 
Hence, the statement now follows inductively. By compactness of~$T$ it follows that one can restrict to only finitely many linear functions.
\end{proof}
\begin{remark}
Note that if $\vec{s}\cdot X \in (u,v)$ is one of the equations determining $Q_\phi\mspace{1mu}$, then $\vec{s}\cdot X = (u',v')$ is another equation whenever $(u,v)\gamma = (u',v')$ for some $\gamma \in \sltwoz$. 
\end{remark}

\begin{cor}\label{cor:finitepoles}
The vector space of strictly meromorphic Jacobi forms of some weight~$k$, index~$M$ and poles~$(\tau,\vec{z})$ only in finitely many fixed hyperplanes of the form
\[\vec{s} \cdot \vec{z} \in u\tau+v  \qquad\qquad (\vec{s}\in \z^n, u,v\in \q/\z)\]
is finite-dimensional.
\end{cor}
\begin{proof}
Recall that the $\vec{\ell}$th Taylor coefficient of a Jacobi form is a quasimodular form of weight $k+|\vec{\ell}|$. Therefore, the multiplicity at a pole is bounded by the weight~$k$. Now, writing $\vec{s}_i \cdot \vec{z}\in u_i\tau-v_i+L_\tau$ for the hyperplanes in the statement, indexed by $i\in I$, we find that the function
$\phi(\vec{z})\prod_{i\in I}\Theta(\vec{s}_i\cdot \vec{z}-u_i\tau-v_i)^k$ is a weak Jacobi form with weight and index uniquely determined by~$\phi$ and the~$\vec{s}_i\mspace{1mu}$. The statement now follows directly from the finite-dimensionality of the space of weak Jacobi forms of fixed weight and index. 
\end{proof}

\subsection{An approximation lemma}
In this section, we prove the approximation properties that were used in the proof of \cref{prop:finitepoles}(\ref{it:finite-iii}).  That is, we prove a result indicating when for given $X,Y\in M_{n,2}(\r)$, there exists a $\gamma \in \sltwoz$ such that~$X \gamma$ lies arbitrarily close to~$Y$. For $Z\in M_{n,m}(\r)$, write~$\|Z\|$ for the distance to the closest integer matrix, i.e., 
\[\|Z\|\defis\max_{i,j}\min_{\ell\in \z}|Z_{i,j}-\ell|.\]
\begin{defn}\label{problem}
Given $X,Y\in \XXR$, we say~$Y$ is $\sltwoz$-\emph{approximable} by~$X$ if
\begin{align}\label{eq:approx}\forall\hspace{1pt}\epsilon>0\;  \exists\hspace{1pt}\gamma\in \sltwoz\colon \| X\gamma -Y \|<\epsilon. \end{align}
\end{defn}
We are interested in conditions on~$X$ and $Y$ such that $Y$ is $\sltwoz$-\emph{approximable} by~$X$.
For example, in case $n=1$ a sufficient (but not necessary) condition is given by $X\not \in M_{1,2}(\q)$:
\begin{lem}\label{lem:approx1}
Given $X,Y\in M_{1,2}(\r)$ and $X\not \in M_{1,2}(\q)$, then
\[\forall\hspace{1pt}\epsilon>0\;  \exists\hspace{1pt}\gamma\in \sltwoz\colon \|X\gamma-Y \|<\epsilon. \]
\end{lem}
\begin{proof}
Write $X=\lms$. First of all, if~$\lambda/\mu$ is irrational, then the stronger statement that the orbit of $\lms$ under~$\sltwoz$ lies dense in~$\r^2$ holds (see, e.g.,~\cite{LN12}). Also, if $\mu=0$ the orbit of $\lms$ lies dense in $\r^2/\z^2$ whenever~$\lambda$ is irrational.

From now on, we assume both~$\lambda$ and~$\mu$ are irrational, but with rational ratio. Then,
\[\lms\left(\begin{matrix} a & b \\ c & d \end{matrix}\right)  = \mu \left(\begin{matrix} a \frac{\lambda}{\mu}+ c \\[5pt] b \frac{\lambda}{\mu}+ d \end{matrix}\right).\]
Note that the matrix on the right-hand side parametrizes $N^{-1} \left(\z\wedge \z\right)$, where~$N$ is the denominator of~$\lambda/\mu$ and $\z\wedge \z$ denotes the subset of~$\z^2$ consisting of coprime integers. As~$\mu$ is irrational, we conclude that the orbit of~$\lms$ under~$\sltwoz$ lies dense in~$\r^2/\z^2$.  Hence, if $\lambda$ or~$\mu$ is irrational, then~$\sltwoz$-approximability of~$Y$ by~$X$ holds. 
\end{proof}

If we replace~$\sltwoz$ by~$M_{n,m}(\z)$, 
there is a renowned criterion for when $X$ approximates~$Y$~\cite{Kro1884}.
\begin{thm}[Kronecker's theorem on Diophantine approximation]\ \\ 
For $X,Y\in M_{n,m}(\r)$ it holds that
\[\forall\hspace{1pt}\epsilon>0\;  \exists\hspace{1pt}\gamma\in M_{m}(\z)\colon \| X\gamma -Y \|<\epsilon \]
if and only if
\begin{align}\label{eq:condkronecker} \vec{s}\in \z^n \text{ with } \vec{s}\cdot X \in \z^m \text{ implies } \vec{s}\cdot Y \in \z^m.\end{align}
\end{thm}

This is almost the result we are looking for, except that we want to replace~$M_{2}(\z)$ by~$\sltwoz$. If we were instead considering the action of $\mathrm{SL}_{n}$ on $\XXR$, the results of, for example, \cite{Lau16} would suffice. 
The action of $\sltwoz$ on $M_{1,2}(\r)$ is dealt with in, for example, \cite{Gui10,LN12}. 
 The latter work already hints at the fact that the condition~(\ref{eq:condkronecker}) should be altered if one replaces~$M_{2}(\z)$ by~$\sltwoz$. Namely, if~$\lambda,\mu$ are coprime integers, then~$(\lambda,\mu) \gamma$ is a vector of coprime integers for all $\gamma \in \sltwoz$. Hence, if $X=(\frac{1}{2},\frac{1}{3})$, then~$(0,0)$ is not in the orbit of~$X$ for~$\sltwoz$, although it is in the orbit of~$X$ for~$M_{2}(\z)$. Observe that in this case the smallest $s\in \z_{\geq 1}$ for which $s\cdot(0,0)\in \z^2$ does not equal the smallest $s\in \z_{\geq 1}$ for which $s\cdot X\in \z$ (i.e.\ $1\neq 6$), which is formalized in~(\ref{eq:cond}). 

For almost all $X=\lm\in \XXR$ we have that $1,\lambda_1,\ldots,\lambda_n,\mu_1,\ldots,\mu_n$ are linearly independent over~$\q$. In this case, and, more generally, for generic $X$ defined below, we show one has a Diophantine approximation theorem for~$\sltwoz$.
\begin{defn}\label{def:generic}
Given $X\in \XXR$, we say~$X$ is \emph{generic} when there are $\alpha,\beta\in \z$ such that for all $\vec{s}\in \z^n$ with $\vec{s}\neq \vec{0}$ one has
\[ \vec{s}\cdot (\alpha\vec{\lambda} + \beta \vec{\mu}) \not\in \z.\]
\end{defn}

For $\vec{s}\in \z^n$ and $X\in \XXR$ with $\vec{s}\neq \vec{0}$ and $\vec{s}\cdot X^t\in \z^2$, we write 
\[ (\vec{s},\vec{s}\cdot X) \defis  \max \Bigl\{N\in \z \colon \frac{\vec{s}}{N}\in \z^n,\, \frac{\vec{s}\cdot X}{N}\in \z^2\Bigr\} \]
for the greatest common divisor of the entries of $\vec{s}$ and $\vec{s}\cdot X$.

\begin{prop}[Partial result on Diophantine approximation for~$\sltwoz$]\label{prop:approx}\mbox{}\\
Let~$X,Y\in \XXR$. Then, if~$X$ is generic one has
\begin{align}\label{eq:sl2approx}\forall\hspace{1pt}\epsilon>0\;  \exists\hspace{1pt}\gamma\in \sltwoz\colon \| X\gamma -Y \|<\epsilon.\end{align}
Conversely, if {\upshape(\ref{eq:sl2approx})} holds, then for all non-trivial $\vec{s}\in \z^n$ for which $\vec{s}\cdot X\in \z^2$ one has that $\vec{s}\cdot Y \in \z^2$ and
\begin{align}\label{eq:cond} (\vec{s},\vec{s}\cdot X)  \=  (\vec{s},\vec{s}\cdot Y) .
\end{align}
\end{prop}
\begin{proof}
Suppose that $\vec{s}\in \z^n$ is such that $\vec{s}\cdot X\in \z^2$ (as in the second part of the statement). Observe that $\vec{s}\cdot (X\gamma)\in \z^2$ for all $\gamma\in \sltwoz$. Hence, 
for any $\gamma\in \sltwoz$ one has 
\[ \|X\gamma-Y\|\geq \frac{1}{n} \|\vec{s}\cdot (X\gamma)-\vec{s}\cdot Y\| = \frac{1}{n}\|\vec{s}\cdot Y\|.\]
Therefore, if {\upshape(\ref{eq:sl2approx})} holds, then for all $\vec{s}\in \z^n$ with $\vec{s}\cdot X\in \z^2$ one has $\vec{s}\cdot Y \in \z^2$. 

Next, suppose $\vec{s}\in \z^n$ such that $\vec{s}\cdot X, \vec{s}\cdot Y\in \z^2$ and $\vec{s}\neq \vec{0}$. 
Write $N=(\vec{s},\vec{s}\cdot Y)$  
and $\vec{s'}=N^{-1}\vec{s}$. 
Suppose $\vec{s'}\cdot (X\gamma)\not \in \z^2$ for all $\gamma\in \sltwoz$, i.e., $\vec{s'}\cdot (X\gamma)\in (N^{-1}\z^2)\backslash\z^2$. Therefore, if~\eqref{eq:sl2approx} holds, for any $\gamma\in \sltwoz$ one has 
\[ \|X\gamma-Y\|\geq \frac{1}{n} \|\vec{s'}\cdot (X\gamma)-\vec{s'}\cdot Y\| = \frac{1}{n}\|\vec{s'}\cdot (X\gamma)\| \geq \frac{1}{Nn},\]
which is a contradiction. Therefore, we have that $\vec{s'}\cdot (X\gamma) \in \z^2$ for some $\gamma \in \sltwoz$. This implies that $\vec{s'}\cdot  X \in \z^2$. Hence, $(\vec{s},\vec{s}\cdot Y)\leq (\vec{s},\vec{s}\cdot X)$. Analogously, the other inequality $(\vec{s},\vec{s}\cdot X)\leq (\vec{s},\vec{s}\cdot Y)$ holds. We conclude that $(\vec{s},\vec{s}\cdot X)  \=  (\vec{s},\vec{s}\cdot Y) .$

Next, let~$X=\lm$ be generic, i.e., let $\alpha,\beta\in \z$ be such that for all non-trivial $\vec{s}\in \z^n$ we have $\vec{s}\cdot(\alpha\vec{\lambda}+\beta\vec{\mu})\not \in \z$. Assume without loss of generality that $\alpha$ and $\beta$ are coprime. Choose $\gamma_0\in \sltwoz$ such that its first column is given by $(\alpha,\beta)^t$ and write $X\gamma_0= (\vec{\tilde{\lambda}},\vec{\tilde{\mu}})$. By construction, $\vec{\tilde{\lambda}}=\alpha\vec{\lambda}+\beta\vec{\mu}$. Hence, by Kronecker's theorem we find infinitely many $b\in \z$ such that
\begin{align}\label{eq:b}\|b\vec{\tilde{\lambda}}+\vec{\tilde{\mu}}-\vec{\xi}\|<\epsilon,\end{align}
where we denoted $Y=(\vec{\nu},\vec{\xi})$. 
Suppose for all such $b$ there exists a non-trivial $\vec{s}(b)\in \z^n$ such that $\vec{s}(b)\cdot(b\vec{\tilde{\lambda}}+\vec{\tilde{\mu}})\in \z$. Note that there exist finitely many integers $a_b$ with $\sum_{b} a_b\, \vec{s}(b)=\vec{0}$ and $\sum_{b} b\, a_b\, \vec{s}(b)\neq \vec{0}$. Hence, we find
\[ \z \,\ni\, \sum_{b} a_b\, s(\vec{b})\cdot (b\vec{\tilde{\lambda}}+\vec{\tilde{\mu}}) \= \sum_{b} b\,a_b\,\vec{s}(b)\cdot \vec{\tilde{\lambda}},\]
contradicting our assumption that $X$ is generic. 
 Hence, for a suitably chosen $b\in\z$ both~\eqref{eq:b} holds and there exists a $c\in \z$ such that
\[ \|c(b\vec{\tilde{\lambda}}+\vec{\tilde{\mu}})+\vec{\tilde{\lambda}}-\vec{\nu}\|<\epsilon.\]
In other words, for $\gamma=\left(\begin{smallmatrix} bc+1 & b \\ c & 1 \end{smallmatrix}\right)\in \sltwoz$, we have
\[\|X\gamma_0\gamma -Y\|<\epsilon,\]
as desired.
\end{proof}
\begin{remark}
Condition~(\ref{eq:cond}) is necessary for~$\sltwoz$-approximability of~$Y$ by~$X$. Whether this condition also suffices, or, if not, how it should be strengthened to a necessary and sufficient condition remains an open problem. 
\end{remark}

\subsection{Quasi-Jacobi forms}\label{sec:quasiJacforms}
Consider the real-analytic functions $\nu:\mathfrak{h}\to \c$ and $\xi:\mathfrak{h}\times \c\to \c$, given by
\[\nu(\tau) \defis \frac{1}{2\ii\,\Im(\tau)}, \qquad \xi(\tau,z) \defis  \frac{\Im(z)}{\Im(\tau)}. \]
These functions almost transform as a Jacobi form: 
\begin{align}
(\nu|_{2,0} \gamma)(\tau) &\= \nu(\tau) - \,\frac{c}{c\tau+d} & \qquad  (\xi|_{1,0} \gamma)(\tau,z) &\= \xi(\tau,z)-\frac{cz}{c\tau+d} \\
(\nu|_0 X)(\tau) &\= \nu(\tau) & \qquad (\xi|_0 X)(z) &\= \xi(z) + \lambda.
\end{align}

\begin{defn} Let~$k\in \z$,~$M\in \MM$. Denote by~$C_n$ a subspace of all strictly meromorphic functions~$\mathfrak{h}\times\c^n\to \c$. \emph{An almost Jacobi form}~$\Phi$ of rank~$n$, weight~$k$, index~$M$ and analytic type~$C_n$ satisfies:
\begin{enumerate}[\upshape(i)]\itemsep2pt
\item $\Phi\in C_n[\nu(\tau),\xi(\tau,z_1),\ldots,\xi(\tau,z_n)]$;
\item $\Phi|_{k,M}\,g=\Phi$ for all $g\in \Gamma^J_n \mspace{1mu}$.
\end{enumerate}
\end{defn}
\begin{defn} A \emph{quasi-Jacobi form}~$\phi$ is the constant term with respect to~$\nu$ and~$\xi$ of an almost Jacobi form. If $C_n$ equals $\Hol_n^M$, $\Hol_n$ or $\Mer_n^M$, $\phi$ is a \emph{holomorphic}, \emph{weak} or \emph{strictly meromorphic quasi-Jacobi form}, respectively. We write $\widetilde{J}^{\mathrm{hol}}, \widetilde{J}^{\mathrm{weak}}$ and $\widetilde{J}^{\mathrm{sm}}$ for the algebras of {holomorphic}, {weak}, and {strictly meromorphic quasi-Jacobi forms}. 
\end{defn}

As a first example, quasimodular forms are quasi-Jacobi forms of rank~$n=0$. 

More interestingly, the functions
\[ E_2(\tau,z) - 2\pi\ii\,\nu(\tau), \qquad E_1(\tau,z)+2\pi\ii\,\xi(\tau,z),\]
are almost strictly meromorphic Jacobi forms of index~$0$ and weight~$2$ and~$1$ respectively. Hence, $E_2$ and~$E_1$ are strictly meromorphic quasi-Jacobi forms. Observe that
\[
 E_2(\tau,z)=-D_zE_1(\tau,z) \qquad E_1(\tau,z) = (2\pi\ii) \,D_z \log \Theta(\tau,z) \qquad\biggl(D_z=\frac{1}{2\pi\ii}\pdv{}{z}\biggr),\]
which reminds one of 
\[e_2 = 8\pi^2 D_\tau \log \eta,  \qquad\biggl(\eta(\tau) = q^{1/24}\prod_{n}(1-q^n), \quad D_\tau = \frac{1}{2\pi\ii}\pdv{}{\tau}=q\pdv{}{q}\biggr).\]

\begin{remark}
Let $\Phi$ be the meromorphic almost Jacobi form corresponding to~$\phi$ and write
\[\Phi(\tau,\vec{z}) \= \sum_{i,\vec{j}}\phi_{i,\vec{j}}(\tau,\mathbf{z})\,\nu(\tau)^i\,\xi(\tau,z_1)^{j_1}\cdots\xi(\tau,z_n)^{j_n}.\]
Making use of the algebraic independence of~$\nu$ and~$\xi$ over the field of meromorphic functions, 
we find $\phi\in \widetilde{J}^{\mathrm{sm}}$ precisely if
$\phi\in\Mer_n^M$ and there exist a finite number of $\phi_{i,\vec{j}}\in \Mer_n^M$, indexed by a subset of $\z_{\geq 0}\times \z_{\geq 0}^n\mspace{1mu}$, satisfying
 \begin{align}
(\phi|\gamma)(\tau,\vec{z}) &\= \sum_{i,\vec{j}}\phi_{i,\vec{j}}(\tau,\mathbf{z})\Bigl(\frac{c}{c\tau+d}\Bigr)^{\! i+|\vec{j}|}\frac{\vec{z}^{\vec{j}}}{(2\pi\ii)^{i}}\, ; \\
(\phi|X)(\vec{z}) &\= \sum_{\vec{j}}\phi_{0,\vec{j}}(\vec{z})\,(-\vec{\lambda})^{\vec{j}}.
\end{align}
These are equations \eqref{eq:action-quasimod} and \eqref{eq:action-quasiell} in the introduction.
(Recall that for vectors~$\vec{a}\in \c^n,\vec{b}\in \z^n$ we write $\vec{a}^{\vec{b}} = \prod_{r} a_r^{b_r}$.)
\end{remark}

The strictly meromorphic quasi-Jacobi forms~$e_2$ and~$E_1$ play a central role as building blocks of quasi-Jacobi forms out of Jacobi forms. For convenience, we introduce the following alternative normalizations
\begin{align}\label{eq:e2strange}\mathbbm{e}_2:=\frac{1}{4\pi^2}e_2=\frac{1}{12}-2\sum_{m,r\geq 1}m \,q^{mr}, \qquad A = \frac{1}{2\pi\ii} E_1.\end{align}
Quasi-Jacobi forms are not invariant under the action of the Jacobi group. However, the fact that an almost Jacobi form is a polynomial in~$\nu$ and~$\xi$ implies that a quasi-Jacobi form transforms ``up to a polynomial correction'' as a Jacobi form, or, equivalently, that it is a polynomial in the strictly meromorphic quasi-Jacobi forms~$\mathbbm{e}_2(\tau)$ and~$A(\tau,z)$ with Jacobi forms as coefficients.

\begin{prop}\label{prop:qjf} An equivalent definition for a strictly meromorphic quasi-Jacobi form is as follows: $\phi\in \widetilde{J}^{\mathrm{sm}}_{k,M}$ if 
$\phi\in\Mer_n^M$ and there exist a finite number of 
$\psi_{i,\vec{j}}\in J_{k-2i-|\vec{j}|,M}^{\mathrm{sm}}\mspace{1mu}$, indexed by a subset of $\z_{\geq 0}\times \z_{\geq 0}^n\mspace{1mu}$, such that
\begin{align}\label{eq:phiet} \phi(\vec{z}) \= \sum_{i,\vec{j}} \psi_{i,\vec{j}}(\vec{z}) \, \mathbbm{e}_2^i \, A(z_{1})^{j_1}\cdots A(z_{n})^{j_n}. \end{align}
\end{prop}
\begin{proof}
We define the coefficients~$\psi_{i,\vec{j}}$ by the expansion of the almost Jacobi form~$\Phi$ corresponding to~$\phi$, i.e.,
\begin{align}\label{eq:phitopsi}\Phi(\tau,\vec{z})\,=:\,\sum_{i,\vec{j}} \psi_{i,\vec{j}}(\tau,\vec{z}) \, \Bigl(\mathbbm{e}_2(\tau)+\frac{\nu(\tau)}{2\pi\ii}\Bigr)^{\!i} \, \prod_{r}\left(A(\tau,z_r)+\xi(\tau,z_r)\right)^{j_r}.
\end{align}
Note that~$\mathbbm{e}_2(\tau)+\frac{\nu(\tau)}{2\pi\ii}$ and~$A(\tau,z_{r})+\xi(\tau,z_r)$ transform as Jacobi forms. Moreover, they are algebraically independent over the space of all meromorphic functions. Hence, it follows that the coefficients~$\psi_{i,\vec{j}}$ are strictly meromorphic Jacobi forms. The constant term with respect to~$i$ and~$\vec{j}$, by definition equal to~$\phi$, is now easily seen to equal to the right-hand side of~(\ref{eq:phiet}).
\end{proof}
\begin{remark}
From~\eqref{eq:phitopsi} it follows that the~$\phi_{i,\vec{j}}$ are quasi-Jacobi forms related to the~$\psi_{i,\vec{j}}$ by 
\begin{align}\label{eq:phipsi}
\phi_{i,\vec{j}}(\vec{z}) \= \sum_{i',\vec{j}'}  \binom{i+i'}{i}\binom{\vec{j}+\vec{j'}}{\vec{j}}\,\psi_{i+i',\vec{j}+\vec{j'}}(\vec{z})\,\mathbbm{e}_2^{i'}\,\prod_rA(z_r)^{j_r'},\end{align}
where
\[\binom{\vec{j}+\vec{j'}}{\vec{j}} \defis \prod_{r} \binom{j_r+j_r'}{j_r}.\]
Also, note that given a representation for~$\phi$ as in~(\ref{eq:phiet}) one has 
\begin{align}&\phi_{1,\vec{0}} \=  \pdv{}{\mathbbm{e}_2}\phi \quad \text{ and } \quad \phi_{0,e_i}(\vec{z})\=\pdv{}{A(z_i)} \phi(\vec{z}).\qedhere
\end{align}
\end{remark}

By~(\ref{eq:phiet}) quasi-Jacobi forms share the properties of Jacobi forms with respect to the location of the poles:
\begin{cor}\label{cor:poleshyperplane} The statement of \cref{prop:finitepoles} also holds when~$\phi$ is a strictly meromorphic \emph{quasi}-Jacobi form.
\end{cor}

As a corollary of \cref{thm:mer0}, we have the following representations for the algebras~$J^0$ and~$\widetilde{J}^0$ of all rank~one strictly meromorphic Jacobi forms and strictly meromorphic quasi-Jacobi forms with all poles at the lattice points~$L_\tau\mspace{1mu}$:
\begin{align}\label{eq:mer0}
J^0&=\c[E_2-e_2,E_3,e_4,\Theta], & \widetilde{J}^0&=\c[E_1,E_2,E_3,e_2,e_4,\Theta].
\end{align}
Observe the vector subspaces for fixed weight and index are finite-dimensional. This holds more generally.
\begin{cor}\label{cor:polesquasiJ} The space of all meromorphic quasi-Jacobi forms of some weight~$k$, index~$M$, and with all poles in a finite union of fixed rational hyperplanes as in \cref{thm:poleshyperplane} is finite-dimensional.
\end{cor}
\begin{proof}
This follows directly from the previous proposition as by \cref{cor:finitepoles} the number of linearly independent~$\psi_{i,\vec{j}}$ is finite.
\end{proof}

The operators $D_\tau=\frac{1}{2\pi\ii}\pdv{}{\tau}$ and $D_{z_i}=\frac{1}{2\pi\ii}\pdv{}{z_i}$ preserve the space of quasi-Jacobi forms ($1\leq i\leq n$). This leads to yet another equivalent definition of quasi-Jacobi forms, as derivatives of Jacobi forms.  Note that no power of the quasi-Jacobi form~$e_2$ (which is in fact a quasimodular form and hence of trivial index) can be written in terms of derivatives of Jacobi forms. However, in case the index is positive definite, by \cite[Proposition~1(i)]{OP19} we have the following.
\begin{prop} Let~$\phi$ be a quasi-Jacobi form of weight~$k$ and positive definite index~$M$. Then, there exist unique Jacobi forms~$\psi_{\vec{d}}$ with~$\vec{d}\in \z_{\geq 0}^{n+1}$ of weight~$k-2d_0-d_1-\ldots-d_n$ and index~$M$ such that
\[\phi \= \sum_{\vec{d}} D_\tau^{d_0} D_{z_1}^{d_1}\cdots D_{z_n}^{d_n} \psi_{\vec{d}}\, .\]
\end{prop}
\begin{proof}
Choose an ordering on~$\z^{n+1}$ respecting the ordering on~$\z$. Given a Jacobi form~$\phi$, let~$(i,\vec{j})$ be maximal (with respect to this ordering) for which~$\phi_{i,\vec{j}}$ in \cref{prop:qjf} exists and is non-zero. A direct check using the same proposition shows that~$\phi$ minus a multiple of ${D}_\tau^{j_0}{D}_{z_{j_1}}^{j_1}\cdots D_{z_{j_n}}^{j_n}\phi_{i,\vec{j}}$ is a quasi-Jacobi form for which this maximal index is smaller. Here, as $M$ is positive definite, this multiple is non-zero.
\end{proof}

\subsection{Action of the Jacobi Lie algebra by derivations}\label{sec:j-alg}
The last proposition depended on the derivations $D_\tau$ and $D_{z_i}\mspace{1mu}$.  Here, we study natural derivations on the space of quasi-Jacobi forms. Given a quasi-Jacobi form $\phi$, recall the quasi-Jacobi forms $\phi_{i,\vec{j}}$ (see Equations~\eqref{eq:action-quasimod} and~\eqref{eq:action-quasiell} or the previous section) are defined for $(i,\vec{j})$ in a finite index set~$I$ with $I\subset\z\times \z^n$. By convention, we let $\phi_{i,\vec{j}}:=0$ if $(i,\vec{j})\not \in I$. 
\begin{defn}
Let~$\dtau$ and~$\dd_{z_i}$ be the derivations on the space of quasi-Jacobi forms given by $\phi\mapsto \phi_{1,\vec{0}}$ and $\phi \mapsto \phi_{0,e_i}$ respectively (with $e_i$ the standard $i$th basis vector of~$\r^n$). 
\end{defn}
Observe that the functions $\phi_{i,\vec{j}}$ are given by 
\begin{align}\phi_{i,\vec{j}} = \frac{\dtau^i}{i!}\frac{\dd_{\vec{z}}^{\vec{j}}}{\vec{j}!} \phi &&  (\dd_{\vec{z}}^{\vec{j}} = \dd_{z_1}^{j_1}\cdots \dd_{z_n}^{j_n},\ \vec{j}! = j_1!\cdots j_n!).\end{align}
Now, a \emph{key observation} for the rest of this work is that the transformation behaviour of~$\phi$ is uniquely determined by the action of the operators 
 ${\dtau^i}{\dd_{\vec{z}}^{\vec{j}}}$ on~$\phi$. 
In particular, the transformation behaviour of the Taylor coefficients of~$\phi$---which are quasimodular forms---is determined by the action of ${\dtau^i}$ on these coefficients.  
 In the next section we investigate how the transformation behaviour of a quasi-Jacobi form determines the transformation of its Taylor coefficients, and vice versa, by studying the action of these derivations. 

\begin{remark}
Writing~$\phi$ as in~(\ref{eq:phiet}) yields 
\begin{align}\dtau\phi\=\pdv{}{\mathbbm{e}_2}\phi, \qquad (\dd_{z_i}\phi)(\vec{z})\=\pdv{}{A(z_i)}\phi(\vec{z}).&\qedhere\end{align}
\end{remark}

The operators $D_\tau,D_{z_i},\dtau$ and $\dd_{z_i}$ are part of a Lie algebra of operators acting on quasi-Jacobi forms by derivations, as we explain now. 
Following a suggestion by Zagier, we consider the notion of a~$\mathfrak{g}$-algebra for any Lie algebra~$\mathfrak{g}$, defined as follows.
\begin{defn}\label{def:g-alg}
Given a Lie algebra~$\mathfrak{g}$, a~$\mathfrak{g}$-algebra is an algebra~$A$ together with a Lie homomorphism~$\mathfrak{g}\to \mathrm{Der}(A),$ where $\mathrm{Der}(A)$ denotes the Lie algebra of all derivations on~$A$. 
\end{defn}

Denote by~$W$ and~$\ind_{ij}$ the weight and index operators acting diagonally by multiplying with the weight~$k$ and~$B_M(e_i,e_j)$ respectively, where~$B_M$ is the bilinear form corresponding to the index~$M$. 
Let~$\mathfrak{j}$ be the Lie algebra of the Jacobi group. By \cite[Eqn.~(12)]{OP19} the Lie algebra of the Jacobi group acts by the aforementioned derivations on the space of quasi-Jacobi forms. 
\begin{prop}\label{prop:j-algebra} The algebra of quasi-Jacobi forms is a~$\mathfrak{j}$-algebra, i.e., the algebra of derivations~$D_\tau,D_{z_i}, \dtau,\dd_{z_i}, W$ and~$\ind_{ij}$ is isomorphic to~$\mathfrak{j}$ and acts on the space of quasi-Jacobi forms.
\end{prop}
\begin{remark}
More concretely, the commutation relations of (i) the modular operators, (ii) the elliptic operators and (iii) their interactions are given by
\begin{flalign} 
\text{\upshape (i)}&& [\dtau,D_\tau]&=W\mspace{1mu} ,  &[W,D_\tau]&=2D_\tau\mspace{1mu} , &[W,\dtau]&=-2\dtau\mspace{1mu} , & \hspace{30pt}  \\
\text{\upshape (ii)}&& [\dd_{z_i},D_{z_j}] &= 2\ind_{i,j}\mspace{1mu} ,  & [\ind_{ij},D_{z_i}]&=0\mspace{1mu} , & [\ind_{ij},\dd_{z_i}]&=0\mspace{1mu} ,	 \\
\text{\upshape (iii)}&& [\dd_{z_i},D_\tau] &= D_{z_i}\mspace{1mu} , &[\dtau,D_{z_i}]&=\dd_{z_i}\mspace{1mu} ,  & [W,D_{z_i}] &= D_{z_i}\mspace{1mu} .
\end{flalign}
The other commutators vanish. As the spaces of almost Jacobi forms and quasi-Jacobi forms are isomorphic, the same result holds for almost Jacobi forms when one replaces~$\dtau$ by~$2\pi\ii \pdv{}{\nu}$ and~$\dd_z$ by~$\pdv{}{\xi(z)}$. 
\end{remark}

\subsection{The double slash operator}
A holomorphic Jacobi form has two important representations: the \emph{theta expansion} and the \emph{Taylor expansion}. We generalize the Taylor expansion to strictly meromorphic quasi-Jacobi forms in such a way that the Taylor coefficients are quasimodular forms. Moreover, we give criteria based on the coefficients in these representations for a meromorphic function to be a quasi-Jacobi form. 

Given a Jacobi form~$\phi$ and~$X\in \XX$, the Taylor coefficients of~$(\phi|X)(\vec{z})$ around $\vec{z}=\vec{0}$ are quasimodular forms for the group
\begin{align}\label{eq:gammaX}\Gamma_X = \{\gamma\in \sltwoz \mid X\gamma- X \in \XXZ \text{, } \rho(X-X\gamma)=\zeta_{X,X\gamma-X}\},\end{align}
where
~$\rho$ and~$\zeta_{X,X'}$ are defined by~(\ref{def:rhoXzetaXX}). In contrast to Jacobi forms, it is not true that the Taylor coefficients of quasi-Jacobi forms are quasimodular. Namely, as stated in the introduction, for $X=(\lambda,\mu)\in \XX$ one has that~$(\Theta'|X|\gamma)(\tau,z)$ equals---up to the  multiplicative constant~$\rho(X)\rho(-X\gamma)$---
\begin{align}\label{eq:theta'} (\Theta'|X\gamma)(\tau,z) + \frac{cz}{c\tau+d}(\Theta|X\gamma)(\tau,z)+\lambda \,(\Theta|X\gamma)(\tau,z) - \frac{\lambda}{c\tau+d}(\Theta|X\gamma)(\tau,z),\end{align}
for all~$\gamma \in \Gamma_X\mspace{1mu}$. 
All but the last term~$- \frac{\lambda}{c\tau+d}(\Theta|X)(\tau,z)$ of~(\ref{eq:theta'}) depend polynomially on~$\frac{c}{c\tau+d}$, so that the Taylor coefficients of~$(\Theta'|X)(\tau,\vec{z})$ at $\vec{z}=\vec{0}$ are not transforming in accordance with the quasimodular transformation formula.  Note that this last term can be written as~$-\lambda\,(\Theta|X|_0\gamma)(\tau,z)$. Here, it should be noted that the weight~$0$ in the slash operator is unusual. Namely,~$\Theta$ is of weight~$-1$, whereas~$\Theta$ is of weight~$0$. In conclusion, the function
\begin{align}\Theta'\|X:=\rho(-X) \left(\Theta'|X + \lambda \,\Theta|X\right)\end{align}
rather than~$\Theta'|X$ transforms as a quasi-Jacobi form of weight~$0$, i.e., for~$\gamma\in \Gamma_X$ one has
\[(\Theta'\|X|_0\gamma)(\tau,z)  \= (\Theta'\|X)(\tau,z)  \+ \frac{cz}{c\tau+d}(\Theta|X)(\tau,z).\]
We now introduce the double slash action~$(\phi\|X)$ for any quasi-Jacobi form $\phi$ and all $X\in \XXR$. In the next section, we will use this notation to define the Taylor coefficients of~$\phi$ at $X$. 
\begin{defn}\label{defn:doubleslash}
Given $M\in \MM$ and a family of functions~$\phi_{0,\vec{j}}\colon\mathfrak{h}\times \c^n\to \c$ indexed by a finite subset of $\z_{\geq 0}^n$ (with $\phi:=\phi_{0,\vec{0}}$), define the \emph{double slash operator} by
\[\phi\|_M\mspace{2mu}X \defis \rho(-X)\sum_{\vec{j}}(\phi_{0,\vec{j}}|_M\mspace{2mu} X) \,\vec{\lambda}^{\vec{j}},\]
where~$\rho$ is given by~(\ref{def:rhoXzetaXX}). 
\end{defn}
\begin{convt}\label{convt}
In case~$\phi$ is a quasi-Jacobi form, in this definition, we always take the family~$\phi_{0,\vec{j}}$ determined by the elliptic transformation~(\ref{eq:phiet}).
\end{convt}
\begin{prop}\label{prop:phi||X}
Given a family of functions $\phi_{i,\vec{j}}\colon\mathfrak{h}\times \c^n\to \c$ indexed by a finite subset of $\z_{\geq 0}\times \z_{\geq 0}^n$ (with $\phi=\phi_{0,\vec{0}}$) and $X\in \XXR$, one has
\begin{enumerate}[{\upshape (i)}]
\item\label{it:qjf} If~$\phi$ satisfies the quasimodular transformation {\upshape(\ref{eq:action-quasimod})} for~$\Gamma$, then
\[(\phi\|X|\gamma)(\tau,\vec{z}) \= \sum_{i,\vec{j}}(\phi_{i,\vec{j}}\|X\gamma)(\tau,\vec{z}) \, \Bigl(\frac{c}{c\tau+d}\Bigr)^{\! i+|\vec{j}|} \frac{\vec{z}^{\vec{j}}}{(2\pi\ii)^i} \= (\phi|\gamma\|X\gamma)(\tau,\vec{z})\]
for all $\gamma \in \Gamma$. 
\item \label{it:el} If~$\phi$ satisfies the quasi-elliptic transformation {\upshape(\ref{eq:action-quasiell})}, then 
\[\zeta_{X',X}\,\phi\|X\|X'\=\zeta_{X,X'}\,\phi\|X'\|X\=\rho(-X')\,\zeta_{X,X'}\,\,\phi\|X \= \phi\|(X+X')
\]
for all~$X'\in \XXZ$, where the root of unity~$\zeta_{X,X'}$ is defined by {\upshape(\ref{def:rhoXzetaXX})}. 
\item\label{it:cor} If~$\phi$ is a quasi-Jacobi form for~$\sltwoz$, then~$\phi\|X$ is a quasi-Jacobi form for~$\Gamma_X$, and
\begin{align}\phi\|(X+X')\=\rho(-X')\,\zeta_{X,X'}\,\phi\|X\end{align}
for all $X'\in \XXZ$.
\end{enumerate}
\end{prop}
\begin{proof}
The transformation of~$\phi\|X$ under the Jacobi group follows by direct computations. 
We often make use of
\[
\frac{\dtau^{i'}}{i'!}\frac{\dd_{\vec{z}}^{\vec{j'}}}{\vec{j'}!} \phi_{i,\vec{j}} \= \frac{\dtau^{i'}}{i'!}\frac{\dd_{\vec{z}}^{\vec{j'}}}{\vec{j'}!} \frac{\dtau^i}{i!}\frac{\dd_{\vec{z}}^{\vec{j}}}{\vec{j}!} \phi 
\= \binom{i+i'}{i} \, \binom{\vec{j}+\vec{j'}}{\vec{j}}\,\phi_{i+i',\vec{j}+\vec{j'}}\, ,\]
where (as before)
\[\binom{\vec{j}+\vec{j'}}{\vec{j}} \=\prod_{r} \binom{j_r+j_r'}{j_r}.\]

The first property follows from the following:
\begin{align}
\phi\|X|_k\gamma &\= \rho(-X)\sum_{\vec{\ell}} (\phi_{0,\vec{\ell}}|X|_k\gamma)\,\vec{\lambda}^{\vec{\ell}} \\
&\= \rho(-X\gamma)\sum_{\vec{\ell}} (\phi_{0,\vec{\ell}}|_k\gamma|X\gamma)\,\vec{\lambda}^{\vec{\ell}} \\
&\= \rho(-X\gamma)\sum_{i,\vec{j},\vec{\ell}} (\phi_{i,\vec{j}+\vec{\ell}}|X\gamma)\,\Bigl(\frac{1}{2\pi\ii}\frac{c}{c\tau+d}\Bigr)^{\! i} \,  \binom{\vec{j}+\vec{\ell}}{\vec{j}}\, \Bigl(\frac{c(\vec{z}+\vec{\lambda}^\gamma\tau+\vec{\mu}^\gamma)}{c\tau+d}\Bigr)^{\! \vec{j}}
\,\Bigl(\frac{\vec{\lambda}}{c\tau+d}\Bigr)^{\! \vec{\ell}} \\
&\=\rho(-X\gamma) \sum_{i,\vec{j}} (\phi_{i,\vec{j}}|X\gamma)\,\Bigl(\frac{1}{2\pi\ii}\frac{c}{c\tau+d}\Bigr)^{\! i} \,   \Bigl(\frac{c(\vec{z}+\vec{\lambda}^\gamma\tau+\vec{\mu}^\gamma)+\vec{\lambda}}{c\tau+d}\Bigr)^{\vec{j}} \\
&\= \rho(-X\gamma)\sum_{i,\vec{j}} (\phi_{i,\vec{j}}|X\gamma)\,\Bigl(\frac{1}{2\pi\ii}\frac{c}{c\tau+d}\Bigr)^{\! i} \,   \Bigl(\frac{c\vec{z}}{c\tau+d}+\vec{\lambda}^\gamma\Bigr)^{\! \vec{j}} \\
&\= \rho(-X\gamma)\sum_{i,\vec{j},\vec{\ell}} (\phi_{i,\vec{j}+\vec{\ell}}|X\gamma)
\,\Bigl(\frac{1}{2\pi\ii}\frac{c}{c\tau+d}\Bigr)^{\! i} \, \binom{\vec{j}+\vec{\ell}}{\vec{j}}
\,\Bigl(\frac{c\vec{z}}{c\tau+d}\Bigr)^{\! \vec{j}} 
\,(\vec{\lambda}^\gamma)^{\vec{\ell}} \\
&\= \sum_{i,\vec{j}} (\phi_{i,\vec{j}}\|X\gamma)
\,\Bigl(\frac{1}{2\pi\ii}\frac{c}{c\tau+d}\Bigr)^{\! i} 
\,\Bigl(\frac{c\vec{z}}{c\tau+d}\Bigr)^{\! \vec{j}} .
\end{align}

For the second property, observe that
\begin{align}
\phi\|X\|X' &\= \rho(-X')\,\sum_{\vec{\ell}} (\phi_{0,\vec{\ell}}\|X|X') \, (\vec{\lambda}')^{\vec{\ell}} \\
&\= \rho(-X)\rho(-X')\,\sum_{\vec{j},\vec{\ell}} (\phi_{0,\vec{j}+\vec{\ell}}|X|X') \, \binom{\vec{j}+\vec{\ell}}{\vec{j}}\,\vec{\lambda}^{\vec{j}}\,(\vec{\lambda}')^{\vec{\ell}} ,
\end{align}
from which it is clear that $\zeta_{X',X}\,\phi\|X\|X'=\zeta_{X,X'}\,\phi\|X'\|X$. Moreover, by the elliptic transformation
\begin{align}
 \phi\|X'\|X \= &\rho(-X)\rho(-X')\,\sum_{\vec{j},\vec{\ell},\vec{m}} (\phi_{0,\vec{j}+\vec{\ell}+\vec{m}}|X) \, \binom{\vec{j}+\vec{\ell}+\vec{m}}{\vec{j},\vec{\ell},\vec{m}}\,\vec{\lambda}^{\vec{j}}\,(\vec{\lambda}')^{\vec{\ell}}\,(-\vec{\lambda}')^{\vec{m}} \\
 \= &
 \rho(-X)\rho(-X')\, \sum_{\vec{j}} (\phi_{0,\vec{j}}|X) \, \vec{\lambda}^{\vec{j}} \\
 \= & \rho(-X')\, \phi\|X.
\end{align}	
Next, one has
\begin{align}
(\phi\|X+X') &\= \rho(-X-X')\,\sum_{\vec{\ell}} (\phi_{0,\vec{\ell}}|X+X') \, (\vec{\lambda}+\vec{\lambda}')^{\vec{\ell}} \\
&\= \rho(-X)\rho(-X')\,\zeta_{X,X'}\,\sum_{\vec{\ell}} (\phi_{0,\vec{\ell}}|X'|X) \, (\vec{\lambda}+\vec{\lambda}')^{\vec{\ell}} \\
&\= \rho(-X)\rho(-X')\,\zeta_{X,X'}\,\sum_{\vec{j},\vec{\ell}} (\phi_{0,\vec{j}+\vec{\ell}}|X) \, \binom{\vec{j}+\vec{\ell}}{\vec{j}}\,(-\vec{\lambda'})^{\vec{j}} (\vec{\lambda}+\vec{\lambda}')^{\vec{\ell}}  \\
&\= \rho(-X)\rho(-X')\,\zeta_{X,X'}\,\sum_{\vec{j}} (\phi_{0,\vec{j}}|X) \,\vec{\lambda}^{\vec{j}} \\
&\= \rho(-X')\, \zeta_{X,X'}\,\phi\|X.	
\end{align}

Finally, the fact that~$\phi\|X$ is a quasi-Jacobi form follows directly from the definition of~$\Gamma_X$ and the previous properties. 
\end{proof}

\subsection{Taylor coefficients}\label{ch:Taylor}							 
Let $X=\lm\in \XX, M\in \MM$ and $\phi\in \Mer_n\mspace{1mu}$. We now study the Taylor coefficients of~$\phi\|_M\mspace{2mu} X$ around $\vec{z}=\vec{0}$. In case~$\phi$ is a strictly meromorphic quasi-Jacobi form, recall that all poles~$\vec{z}$ lie on hyperplanes of the form $\vec{s}\cdot \vec{z}\in u\tau+v$ for some~$\vec{s}\in \z^n$ and $u,v\in \q/\z$ by \cref{thm:poleshyperplane}.  From now on, we assume that $\vec{s}=e_i$ for some~$i$ so that a Laurent series of $\phi\|X$ of the form
\[ \sum_{\ell_1\geq L} \cdots \sum_{\ell_n\geq L} a_{\ell_1,\ldots,\ell_n} (z_1-\lambda_1\tau-\mu_1)^{\ell_1} \cdots (z_n-\lambda_n\tau-\mu_n)^{\ell_n}\]
for some $L\in \z$ and $a_{\vec{\ell}}\in \c$ exists. For example, the poles of all the meromorphic quasi-Jacobi forms we encounter in the applications lie on the coordinate axes. 

\begin{defn}\label{defn:orthogonal}
We call the poles of a meromorphic function $\phi\colon\mathfrak{h}\times \c^n\to\c$ \emph{orthogonal} if the set of poles of~$\phi(\tau,\cdot)$ is given by a union of special hyperplanes of the form
\[ z_j \in u\tau+v \]
for some $j\in\{1,\ldots,n\}$ and $u,v\in \q/\z$. 
\end{defn}
Making use of the notions of \emph{orthogonal} poles (defined above) and the double slash action (defined in \cref{defn:doubleslash}), we will now define the ``Taylor coefficients'' of a family of functions in the following way. Recall that in case~$\phi$ is a Jacobi form, there is a canonical choice for the family of functions $\mathbf{\phi}$ which is part of the data of these ``Taylor coefficients'' (see \cref{convt}). 

\begin{defn}\label{defn:Taylor}
Let~$M\in \MM$ and $\mathbf{\phi}=\{\phi_{i,\vec{j}}\}$, where $\phi_{i,\vec{j}}\colon\mathfrak{h}\times \c^n\to \c$ is a family of meromorphic functions indexed by a finite subset of $\z_{\ge 0}\times \z_{\ge 0}^n\mspace{1mu}$, with $\phi:=\phi_{0,\vec{0}}\in \Mer_n^M$ such that all poles of~$\phi$ are orthogonal. Define~${g}_{\vec{\ell}}(\phi)$ as the~$\vec{\ell}$th Laurent coefficient of~$\phi$:
\[\phi(\tau,\vec{z}) \,=:\, \sum_{\vec{\ell}} {g}_{\vec{\ell}}(\phi)(\tau) \,\vec{z}^{\vec{\ell}}.\]
For all $X
\in \XXR$, we define the $\vec{\ell}$th \emph{``Taylor coefficient''}~${g}_{\vec{\ell}}^X(\mathbf{\phi})$ of $\phi$ as ${g}_{\vec{\ell}}(\phi\|_M\mspace{2mu} X)$. 

Also, denote
\[{g}_{\vec{\ell},s}^{X,r}(\mathbf{\phi}) \= {g}_{\vec{\ell}}\Bigl(B_M^r(\vec{z},\vec{z})\,\sum\nolimits_{i+|\vec{j}|=s}(\phi_{i,\vec{j}}\|X)(\vec{z}) \,\frac{\vec{z}^{\vec{j}}}{(2\pi\ii)^i}\Bigr) \qquad (r,s\in \z_{\geq 0}).\]
We often omit $r,s$ or $X$ from the notation if $r=0,s=0$ or $X=(\vec{0},\vec{0})$.
\end{defn}
\begin{remark}\mbox{}\\[-15pt]
\begin{itemize}\itemsep0pt
\item The functions~${g}_{\vec{\ell},s}^{X,r}(\mathbf{\phi})$ naturally appear if one studies the action of~$\dtau$ on the ``Taylor coefficients''~${g}_{\vec{\ell}}^X(\mathbf{\phi})$ of $\phi$.
\item $\displaystyle g_{\vec{\ell},s}^r(\phi) \= \frac{1}{s!}g_{\vec{\ell}}(B_M(\vec{z},\vec{z})^r\,(\dtau+z_1\dd_{z_1}+\ldots+z_n\dd_{z_n})^s\phi).$
\item In case $r=0$, one may be tempted to write
\[ ``\,{g}_{\vec{\ell},s}^{X,0}(\mathbf{\phi}) \= \sum_{i+|\vec{j}|=s}\frac{1}{(2\pi\ii)^i}\,{g}_{\vec{\ell}-\vec{j}}\left(\phi_{i,\vec{j}}\|X\right) ".\]
However, we do not assume that the functions~$\phi_{i,\vec{j}}$ admit orthogonal poles, so the Taylor expansion of~$\phi_{i,\vec{j}}$ may not exist. For example, taking $\phi=F_2$ (defined by \cref{defn:BOnpoint}), we will see later that $\phi_{0,e_1}(z_1,z_2)=\frac{1}{\Theta(z_1+z_2)}$, which has a pole whenever $z_1+z_2=0$. \Cref{thm:3.2} implicitly shows that the notation~${g}_{\vec{\ell},s}^{X,r}(\mathbf{\phi})$ is well-defined for a quasi-Jacobi form~$\phi$ with all poles orthogonal. \qedhere
\end{itemize}
\end{remark}

The data~$\{{g}_{\vec{\ell}}^X(\mathbf{\phi})\}$ uniquely determines~$\phi$ as well as the family $\mathbf{\phi}=\{\phi_{i,\vec{j}}\}$. Hence, it is natural to ask under which conditions on~${g}_{\vec{\ell}}^X(\mathbf{\phi})$ the function~$\phi$ is a meromorphic quasi-Jacobi form. Before we answer this question, we study the modular properties of~${g}_{\vec{\ell}}^X(\mathbf{\phi})$ given~$\phi$ is a quasi-Jacobi form. As a corollary of the previous proposition on~$\phi\|X$, generalizing \cite[Theorem~1.3]{EZ85} to quasi-Jacobi forms, we first show that~${g}_{\vec{0}}^X(\mathbf{\phi})$ is a quasimodular form in case $\phi$ is holomorphic. 

\begin{cor}
Let~$\phi$ be a holomorphic quasi-Jacobi form of weight~$k$ and index~$M$. For all $X\in \XX$, the function~${g}_{\vec{0}}^X(\mathbf{\phi})$ is a holomorphic quasimodular form of weight~$k$ for the group
$\Gamma_X$ \textup{(}defined by~{\upshape(\ref{eq:gammaX})}\textup{)}. Moreover,
\[\dtau{g}_{\vec{0}}^X(\mathbf{\phi}) = {g}_{\vec{0}}^X(\dtau\mathbf{\phi}).\]
\end{cor}
\begin{proof}
Let~$X\in \XX$ and~$\gamma \in \Gamma_X\mspace{1mu}$. Then, by \cref{prop:phi||X}(\ref{it:cor}) one finds
\begin{equation}\label{eq:chiXgamma}({g}_{\vec{0}}^X(\mathbf{\phi})|\gamma)(\tau) \= (\phi\|X|\gamma)(\tau,\vec{0}) 
\=\sum_{i}{g}_{\vec{0}}^X(\mathbf{\phi}_{i}) \, \Bigl(\frac{1}{2\pi\ii}\frac{c}{c\tau+d}\Bigr)^{\! i},
\end{equation}
where $\mathbf{\phi}_i$ denotes the family corresponding to $\phi_{i,\vec{0}}\mspace{1mu}$. 
Holomorphicity in~$\mathfrak{h}$ and at infinity follows directly as~$\phi$ is a holomorphic Jacobi form. 
Hence,~$g_{\vec{0}}^X(\mathbf{\phi})$ is a quasimodular form for this group and $\dtau^r {g}_{\vec{0}}^X(\mathbf{\phi}) = \frac{1}{r!}{g}_{\vec{0}}^X(\mathbf{\phi}_{r}) = {g}_{\vec{0}}^X(\dtau^r\phi).$ 
\end{proof}
In the above corollary, observe that if $\phi$ is a true Jacobi form (rather than a quasi-Jacobi form), then $\dtau\phi=0$. Correspondingly, in that case $\dtau g_{\vec{0}}^X(\phi)=0$ and $g_{\vec{0}}^X(\phi)$ is a true modular form (rather than a quasimodular form). 

The quasimodularity of the other coefficients~${g}_{\vec{\ell}}^X(\mathbf{\phi})$ can be understood in terms of lower coefficients in two ways. First of all, certain linear combinations of derivatives of these coefficients are modular. Secondly the action of~$\dtau^i$ on~${g}_{\vec{\ell}}^X(\mathbf{\phi})$ can be expressed in terms of other coefficients. 

We first show that these two ways are equivalent. Recall that $(x)_n$ denotes the Pochhammer symbol $(x)_n=x(x+1)\cdots (x+n-1)$. 
\begin{prop}\label{prop:equiv}
Let~$g=g_k,g_{k-2}\ldots,g_{k-2p}$ be quasimodular forms of depth at most~$p$ and weight~$k,k-2\ldots,k-2p$ respectively. Then, the following are equivalent:
\begin{enumerate}[{\upshape (i)}]
\item\label{it:equiv1} $\dtau^ig=g_{k-2i}$ for~$i=0,\ldots, p;$
\item\label{it:equiv2} The functions \begin{equation}\label{eq:modular} 
\begin{cases}\displaystyle\sum_{0\leq m\leq p-i} (-1)^m \frac{D_\tau^mg_{k-2i-2m}}{(k-2i-m-1)_{m}\, m!} & \text{ if } k>2p \text{ or } i<p-1\\
\displaystyle g_2-\mathbbm{e}_2g_0 & \text{ if } k=2p \text{ and } i=p-1
\end{cases}
\end{equation}
for~$i=0,\ldots, p-1$ are modular forms of weight~$k-2i;$ 
\item The functions \begin{equation}\label{eq:modular2}
\sum_{0\leq m\leq p-i}  (-1)^m\frac{(D_\tau+\mathbbm{e}_2)^mg_{k-2i-2m}}{(k-2i-m-\frac{3}{2})_{m}\, m!}
\end{equation}
for~$i=0,\ldots, p-1$ are modular forms of weight~$k-2i$. 
\end{enumerate}
\end{prop}
\begin{remark}
Observe that if~$m\leq p-i$ and~$i\leq p-1$ one has that
\[k-2i-m-1 \geq k-2p.\]
Let $k=2p, i=p-1$ and take $m=p-i$; i.e., $m=1$. 
Then, the numerator~$(k-2i-m-1)_{m}$ vanishes. Moreover, for these values of $k,i$ and $m$, the function~$g_{k-2i-2m}$ is a modular form of weight~$0$. Hence it is a constant function. Therefore, also the numerator~$D_\tau^mg_{k-2i-2m}$ vanishes in this case. One can think of~$\mathbbm{e}_2$ as being the appropriate regularization of the corresponding ill-defined ratio. 

Observe that in the third equivalence we replaced~$D_\tau$ by~$D_\tau+\mathbbm{e}_2$, in which case the corresponding numerator never vanishes. If one would replace $D_\tau$ by $\mathbbm{e}_2$, one would obtain a generalisation of the functions $\phi_n$ of \cite[Proposition~3.1]{Bri18}. 
\end{remark}
\begin{proof}
Assume that (\ref{it:equiv1}) holds. Define $g_{k-2p-2}:=0$. Then, using $[\dtau,D_\tau]=W$ (as follows from \cref{prop:j-algebra}), it follows that applying~$\dtau$ to a term in the sum in~(\ref{it:equiv2}) yields
\[ (-1)^m \frac{D_\tau^m g_{k-2i-2m-2}}{(k-2i-m-1)_{m}\, m!} -(-1)^{m-1}\frac{D_\tau^{m-1} g_{k-2i-2m}}{(k-2i-(m-1)-1)_{m-1} \,(m-1)!}, \]
where the second term is taken to be zero when~$m=0$. Also the first term vanishes when $m=p-i$ as~$g_{k-2p-2}$ is set to be zero. Hence, after applying~$\dtau$ 
the sum becomes a telescoping sum, equal to zero. Also~$\dtau(g_2-\mathbbm{e}_2g_0) = g_0-g_0=0$. The third statement follows from the first by the same argument, mutatis mutandis. 

Conversely, the first statement follows inductively from the second (or third) by using that all but two terms in the same sum, which is equal to zero, cancel. Hence, these terms are equal. 
\end{proof}

Inspired by the above result, we now also introduce certain linear combinations which later turn out to be modular. Recall $(x)_n=x(x+1)\cdots (x+n-1)$ denotes the Pochhammer symbol.
\begin{defn}\label{defn:xi} Let $k\in \z, M\in \MM$, $\phi\in \Mer_n^M$ and a family $\mathbf{\phi}$ as before. 
For $\vec{\ell}\in \z^n$ and $X\in \XX$ with $k+|\vec{\ell}|\geq 0$, the functions~$\xi^X_{\vec{\ell}}(\phi)(\tau):\mathfrak{h}\to \c$ are defined by
\begin{equation}\label{eq:xi}\xi_{\vec{\ell}}^X(\mathbf{\phi})\defis\begin{cases}
\displaystyle\sum_{r=0}^{\lfloor\frac{1}{2}(k+|\vec{\ell}|)\rfloor}\!(-1)^{r}\,\sum_{s=0}^r \frac{D_\tau^{r}\,({g}^{X,r-s}_{\vec{\ell},s})(\mathbf{\phi})}{(k+|\vec{\ell}|-r-1)_{r}\,(r-s)!} &\quad k+ |\vec{\ell}|\neq 2,\\[17pt]
\displaystyle {g}^X_{\vec{\ell}}(\mathbf{\phi})-\mathbbm{e}_2\,\bigl({g}^{X,1}_{\vec{\ell},0}(\mathbf{\phi})+ {g}^{X,0}_{\vec{\ell},1}(\mathbf{\phi})\bigr) &\quad k+|\vec{\ell}|=2.
\end{cases}\end{equation}
Abbreviate~$\xi^X_{\vec{\ell}}(\mathbf{\phi})$ by~$\xi_{\vec{\ell}}(\phi)$ if~$X$ is the zero matrix.
\end{defn}
\begin{remark}
We formulate all results below for~$\xi_{\vec{\ell}}^X(\mathbf{\phi})$ as defined above, but by \cref{prop:equiv} all results remain valid after replacing~$\xi_{\vec{\ell}}^X(\mathbf{\phi})$ by
\begin{align} \sum_{r=0}^{\lfloor\frac{1}{2}(k+|\vec{\ell}|)\rfloor} \!(-1)^{r}\,\sum_{s=0}^r\frac{(D_\tau+\mathbbm{e}_2)^r\,({g}^{X,r-s}_{\vec{\ell},s})(\mathbf{\phi})}{(k+|\vec{\ell}|-r-\frac{3}{2})_{r}\, (r-s)!}.\end{align}
Note that this equation, as well as Equation~\ref{eq:xi}, can be inverted, expressing $\xi_{\vec{m}}^X(\mathbf{\phi})$ as linear combination of derivatives of certain $g_{\vec{\ell}}^X(\mathbf{\phi})$.
\end{remark}

Now, we have completed the full set-up for the main results on Taylor coefficients of quasi-Jacobi forms. In the first theorem, we characterize when a family of meromorphic functions is invariant under the quasimodular action in terms of its Taylor coefficients, generalizing {\cite[Theorem~3.2]{EZ85}}:
\begin{thm}\label{thm:3.2}
Let~$\Gamma$ be a congruence subgroup, $k\in \z$ and $M\in \MM$. As before, let $\mathbf{\phi}=\{\phi_{i,\vec{j}}\}$ be a family of meromorphic functions $\mathfrak{h}\times \c^n\to \c$ with $\phi:=\phi_{0,\vec{0}}\in \Mer_n^M$ 
 and admitting a Laurent expansion around $\vec{z}=\vec{0}$. Then, the following are equivalent:
\begin{enumerate}[{\upshape (i)}]
\item The function~$\phi$ satisfies the quasimodular transformation~{\upshape(\ref{eq:action-quasimod})}
\[(\phi|_{k,M}\mspace{2mu}\gamma)(\tau,z) \= \sum_{i,\vec{j}}\phi_{i,\vec{j}}(\tau,\mathbf{z})\Bigl(\frac{c}{c\tau+d}\Bigr)^{\! i+|\vec{j}|}\frac{\vec{z}^{\vec{j}}}{(2\pi\ii)^i}\] 
for all~$\gamma\in \Gamma;$
\item\label{eq:3.2ii} The coefficients~${g}_{\vec{\ell}}(\phi)$ are quasimodular forms of weight~$k+{|\vec{\ell}|}$ on $\Gamma$, of which the transformation is uniquely determined by the coefficients of the~$\phi_{i,\vec{j}}\mspace{1mu}$, i.e.,
 \[\dtau^r{g}_{\vec{\ell}}({\phi}) \= \sum_{s=0}^r\frac{r!}{(r-s)!}\,{g}^{r-s}_{\vec{\ell},s}(\phi)\, ;\]
\item The functions~$\xi_{\vec{\ell}}(\phi)$ are  modular forms of weight~$k+|\vec{\ell}|$ on~$\Gamma$. 
 \end{enumerate}
\end{thm}
\begin{proof}
Expanding~(\ref{eq:action-quasimod}) yields
\[ \sum_{\vec{\ell}} \frac{{g}_{\vec{\ell}}(\phi)(\gamma\tau)}{(c\tau+d)^{k+|\vec{\ell}|}}\,\vec{z}^{\vec{\ell}} \= 
\sum_{\vec{\ell}}\sum_{r,s} \frac{{g}_{\vec{\ell},s}(\phi)(\tau)}{r!}\Bigl(\frac{c}{c\tau+d}\Bigr)^{\! r+s}\,(B_M(\mathbf{z},\mathbf{z}))^r \vec{z}^{\vec{\ell}}\]
Extracting on both sides the coefficient of~$\vec{z}^{\vec{\ell}}$ yields
\[\dtau^r{g}_{\vec{\ell}}(\phi) \= \sum_{s}\frac{r!}{(r-s)!}\,{g}^{r-s}_{\vec{\ell},s}(\phi)\, .\]
 Also, the coefficient~${g}_{\vec{\ell}}(\phi)$ is holomorphic in~$\mathfrak{h}$ as well as at the cusps of~$\Gamma$, because of the analytic properties of the functions~$\phi_{i,\vec{j}}\mspace{1mu}$. Hence, the first statement is equivalent to the second. The equivalence of the second to the third follows from \cref{prop:equiv}. 
\end{proof}

Sometimes one is interested in comparing the action of~$\dtau$ on some quasi-Jacobi form~$\phi$ with the action of~$\dtau$ \emph{on the Taylor coefficients} of~$\phi$. By expanding the equality in~(\ref{eq:3.2ii}) above we find that the action of $\dtau$ on the Taylor coefficients of~$\phi$ corresponds to the action of $B_M(\vec{z},\vec{z})+\dtau+\sum_{j} z_j\,\dd_{z_j}$ on~$\phi$. In \cite[Proposition 2(ii)]{OP19} an analogous result for the Fourier coefficients of a weakly holomorphic quasi-Jacobi form is derived. 
\begin{cor} If $\phi=\sum_{\vec{\ell}} {g}_{\vec{\ell}}(\phi)(\tau)\,\vec{z}^{\vec{\ell}}$ is a strictly meromorphic quasi-Jacobi form, then
\[ \sum_{\vec{\ell}} \dtau{g}_{\vec{\ell}}(\phi)\,\vec{z}^{\vec{\ell}} \= \bigl(B_M(\vec{z},\vec{z})+\dtau+\sum_{j} z_j\,\dd_{z_j} \bigr)\, \phi. \] 
\end{cor}

One can characterize a quasi-Jacobi form~$\phi$ by its Taylor coefficients in three ways: by considering~$g_{\vec{\ell}}^X(\mathbf{\phi})$ as a vector-valued quasimodular form, by the modularity of the functions~$\xi_{\vec{\ell}}^X(\mathbf{\phi})$, and, finally, by the action of~$\dtau$ on the quasimodular form~${g}^{X}_{\vec{\ell}}(\mathbf{\phi})$. Write~$f^X$ for~$g_{\vec{\ell}}^X(\mathbf{\phi})$ or~$\xi_{\vec{\ell}}^X(\mathbf{\phi})$. Then, the `elliptic transformation' of the (quasi)modular form $f^X$ is given by
\begin{align}\label{eq:ellcoeff} \rho(X')\,\zeta_{X',X} \, f^{X+X'} = f^{X} \qquad \text{ for all } X'\in \XXZ.\end{align}
Recall that~$g_{\vec{\ell}}^X(\mathbf{\phi})$ and~$\xi_{\vec{\ell}}^X(\mathbf{\phi})$ are only defined when the zeros of $\phi$ are orthogonal (see \cref{defn:orthogonal}) and depend on a family of functions~${\phi_{i,\vec{j}}\colon\mathfrak{h}\times \c^n\to\c}$. Moreover, recall that in case $\phi$ is a quasi-Jacobi form by Convention~\ref{convt} this family~$\mathbf{\phi}=\{\phi_{i,\vec{j}}\}$ determines the transformation of~$\phi$. 
\Cref{thm:1} follows from the following result. 
\begin{thm}\label{thm:main}
Let~$k\in \z, M\in \MM$ and~$\phi\in \Mer_n^M$ be such that the poles of $\phi$ are orthogonal. Given a family $\mathbf{\phi}=\{\phi_{i,\vec{j}}\}$, indexed by $(i,\vec{j})$ in a finite subset of $\z_{\geq 0}\times \z_{\geq 0}^n$, of meromorphic functions $\phi_{i,\vec{j}}\colon\mathfrak{h}\times \c^n\to \c$ with $\phi=\phi_{0,\vec{0}}\mspace{1mu}$, the following are equivalent:
\begin{enumerate}[\upshape(i)]
\item\label{it:jf1} The function~$\phi$ is a strictly meromorphic quasi-Jacobi form of weight~$k$ and index~$M$ for which the functions~$\phi_{i,\vec{j}}$ determine its transformation behaviour as in~\eqref{eq:action-quasimod} and~\eqref{eq:action-quasiell}.
\item\label{it:jf2} For all~$X=\lm\in \XX$ with~$(\tau,\vec{\lambda}\tau+\vec{\mu})$ not a pole of $\phi$ for some ${\tau\in \mathfrak{h}}$, the function~${g}^X_{\vec{0}}(\mathbf{\phi})$ is a vector-valued quasimodular form satisfying~\eqref{eq:ellcoeff} and transforming as 
\[{g}^X_{\vec{0}}(\mathbf{\phi})|_k\gamma \= 
\sum_{s} {g}^{X\gamma}_{\vec{0},s}(\mathbf{\phi})\,\Bigl(\frac{c}{c\tau+d}\Bigr)^{\! s}.\]
\item[\upshape(ii$'$)]\label{it:jf2'} For all~$X\in \XX$ the function~${g}^X_{\vec{\ell}}(\mathbf{\phi})$ is a vector-valued quasimodular form satisfying~\eqref{eq:ellcoeff} for $\vec{\ell}=\vec{0}$ and transforming as 
\[{g}^X_{\vec{\ell}}(\mathbf{\phi})|_k\gamma \= 
\sum_{r} \sum_{s}\frac{1}{(r-s)!}\,{g}_{\vec{\ell},s}^{X\gamma,r-s}(\mathbf{\phi})\,\Bigl(\frac{c}{c\tau+d}\Bigr)^{\! r}.\]
\item\label{it:jf3} For all~$\vec{\ell}\in \z^n$ the functions~$\xi_{\vec{\ell}}(\phi)$ are modular forms of weight~$k+|\vec{\ell}|$ for~$\sltwoz$ and 
for all $X\in \XX$ the functions ${\xi}^{X}_{\vec{0}}(\mathbf{\phi})$ satisfy~\eqref{eq:ellcoeff}.
\item[\upshape(iii$'$)]\label{it:jf3'} For all~$X\in \XX$ and~$\vec{\ell}\in \z^n$ the functions~$\xi^{X}_{\vec{\ell}}(\mathbf{\phi})$ in~\eqref{eq:xi} are modular forms of weight~$k+|\vec{\ell}|$ for~$\Gamma_X$ and satisfy~\eqref{eq:ellcoeff}. 
\item\label{it:jf5} For all~$X\in \XX$ and~$\vec{\ell}\in \z^n$ the functions~${g}^{X}_{\vec{\ell}}(\mathbf{\phi})$ are quasimodular forms of weight~$k+|\vec{\ell}|$ for~$\Gamma_X$, satisfying~\eqref{eq:ellcoeff} and 
\begin{align}\label{eq:glX}\dtau^r{g}_{\vec{\ell}}^X(\mathbf{\phi}) \= \sum_{s}\frac{r!}{(r-s)!}\,{g}_{\vec{\ell},s}^{X,r-s}(\mathbf{\phi})\, .\end{align}
\end{enumerate}
\end{thm}
\begin{proof}
(\ref{it:jf1}) implies~(\ref{it:jf5}): Let~$X\in \XX$ and~$\gamma \in \Gamma_X$ be given. By \cref{prop:phi||X} the function~$\phi\|X$ satisfies the conditions of \cref{thm:3.2} for~$\Gamma=\Gamma_X\mspace{1mu}$. Moreover, by the same proposition 
\eqref{eq:ellcoeff} is satisfied.

(\ref{it:jf5}) implies~(\ref{it:jf3}$'$): This follows directly from \cref{thm:3.2} for $\Gamma=\Gamma_X\mspace{1mu}$.

(\ref{it:jf3}$'$) implies~(\ref{it:jf3}):
Observe that~$\Gamma_{X}=\sltwoz$ for~$X$ equal to the zero matrix. Hence, 
we simply forget some of the properties of~$\xi_{\vec{\ell}}^X\mspace{1mu}$. 

(\ref{it:jf3}) implies~(\ref{it:jf2}$'$): As~$\xi_{\vec{\ell}}$ is a modular form for~$\sltwoz$ for all~$\vec{\ell}\in \z^n$, it follows by \cref{thm:3.2} that~$\phi$ satisfies the quasimodular transformation for all~$\gamma\in \sltwoz$. As by \cref{prop:phi||X}
\begin{align}
\phi\|X|_k\gamma &\= \sum_{i,\vec{j}} (\phi_{i,\vec{j}}\|X\gamma)
\,\Bigl(\frac{1}{2\pi\ii}\frac{c}{c\tau+d}\Bigr)^{\! i} 
\,\Bigl(\frac{c\vec{z}}{c\tau+d}\Bigr)^{\! \vec{j}},
\end{align}
the result follows by extracting the coefficients of~$\vec{z}$ on both sides. Finally, note that $\xi_{\vec{0}}^X=g_{\vec{0}}^X\mspace{1mu}$. 

(\ref{it:jf2}$'$) implies~(\ref{it:jf2}): This follows directly by restricting to $\vec{\ell}=\vec{0}$.

(\ref{it:jf2}) implies~(\ref{it:jf1}): Suppose~$\vec{z}=\vec{\lambda}\tau+\vec{\mu}$, with~$X=\lm\in \XX$, is not a pole of~$\phi$. 
Let~$\mathbf{\phi}_{i}$ be the family of functions corresponding to $\phi_{i,\vec{0}}\mspace{1mu}$. Using~(\ref{it:jf2}), for~$\gamma \in \sltwoz$ one has that
\begin{align}
({g}^{X}_{\vec{0}}(\mathbf{\phi})|_k\gamma)(\tau) \= 
&\sum_{i}{g}_{\vec{0}}^{X\gamma}(\mathbf{\phi}_{i})(\tau)\,\Bigl(\frac{1}{2\pi\ii}\frac{c}{c\tau+d}\Bigr)^{\! i} \\
\=&\rho(-X\gamma)\sum_{i,\vec{\ell}}(\phi_{i,\vec{\ell}}|X\gamma)(\tau,0)\,\Bigl(\frac{1}{2\pi\ii}\frac{c}{c\tau+d}\Bigr)^{\! i} (\vec{\lambda}^\gamma)^{\vec{\ell}}  \\	
\=&\rho(-X\gamma)\sum_{i,\vec{j},\vec{\ell}}(\phi_{i,\vec{j}+\vec{\ell}}|X\gamma)(\tau,0)\,\Bigl(\frac{1}{2\pi\ii}\frac{c}{c\tau+d}\Bigr)^{\! i} \binom{\vec{j}+\vec{\ell}}{\vec{j}} \Bigl(\frac{c(\vec{\lambda}^\gamma\tau+\vec{\mu}^\gamma)}{c\tau+d}\Bigr)^{\! \vec{j}} \vec{\lambda}^{\vec{\ell}}.
\end{align}
On the other hand,
\begin{align}
({g}^{X}_{\vec{0}}(\mathbf{\phi}))|_k\gamma &\=\rho(-X)\sum_{\vec{\ell}} (\phi_{0,\vec{\ell}}|X|_k\gamma)(\tau,0)\, \vec{\lambda}^{\vec{\ell}} \\
&\=\rho(-X\gamma) \sum_{\vec{\ell}} (\phi_{0,\vec{\ell}}|_k\gamma|X\gamma)(\tau,0)\, \vec{\lambda}^{\vec{\ell}}. 
\end{align}
Combining the identities yields
\begin{equation}
\sum_{\vec{\ell}} (\phi_{0,\vec{l}}|_k\gamma|X\gamma)(\tau,0)\, \vec{\lambda}^{\vec{\ell}}  \= \sum_{i,\vec{j},\vec{\ell}}(\phi_{i,\vec{j}+\vec{\ell}}|X\gamma)(\tau,0)\,\Bigl(\frac{1}{2\pi\ii}\frac{c}{c\tau+d}\Bigr)^{\! i} \binom{\vec{j}+\vec{\ell}}{\vec{j}} \Bigl(\frac{c(\vec{\lambda}^\gamma\tau+\vec{\mu}^\gamma)}{c\tau+d}\Bigr)^{\! \vec{j}} \vec{\lambda}^{\vec{\ell}} .
\end{equation}
As both sides equal~$\rho(X\gamma)\,{g}^{X}_{\vec{0}}(\mathbf{\phi})$, which is periodic with finite period as a function of~$\vec{\lambda}$, the constant terms with respect to~$\vec{\lambda}$ agree. Hence,
\[(\phi|\gamma|X)(\tau,0)\=\sum_{i,\vec{j}}(\phi_{i,\vec{j}}|X)(\tau,0)\,\Bigl(\frac{1}{2\pi\ii}\frac{c}{c\tau+d}\Bigr)^{\! i}  \Bigl(\frac{c(\vec{\lambda}\tau+\vec{\mu})}{c\tau+d}\Bigr)^{\! \vec{j}} \] 
for all~$X=\lm\in \XX$ with~$X$ not corresponding to a pole. Therefore,~$\phi$ satisfies the quasimodular transformation for all~$\vec{z}$ of the given form. As~$(\q^n\tau\times \q)^n\setminus P_\phi$, with $P_\phi$ the set of poles of~$\phi$, lies dense in~$\c^n$ for all $\tau\in \mathfrak{h}$, the function~$\phi$ satisfies the quasimodular transformation equation. 

For the elliptic transformation, we again assume~$\vec{z}=\vec{\lambda}\tau+\vec{\mu}$, with~$X=\lm\in \XX$, is not a pole of $\phi$. Given~$X'=(\vec{\lambda'},\vec{\mu'})\in \XXZ$, by \cref{defn:Taylor}, \cref{defn:doubleslash} and \cref{prop:app} we have
\begin{align}
\sum_{\vec{\ell}}(\phi_{0,\vec{\ell}}|X)(\tau,0)\,\vec{\lambda}^{\vec{\ell}} &\=\rho(X)\,{g}^{X}_{\vec{0}}(\mathbf{\phi}) \\
&\= \rho(X)\,\rho(X')\,\zeta_{X',X}\,{g}^{X+X'}_{\vec{0}}(\mathbf{\phi}) \\
&\= \rho(X)\,\rho(X')\,\rho(-X-X')\,\zeta_{X',X}\,\sum_{\vec{\ell}}(\phi_{0,\vec{\ell}}|(X+X'))(\tau,0)\,(\vec{\lambda}+\vec{\lambda'})^{\vec{\ell}} \\
&\= \sum_{\vec{\ell}}(\phi_{0,\vec{\ell}}|X'|X)(\tau,0)\,(\vec{\lambda}+\vec{\lambda'})^{\vec{\ell}}.
\end{align}
The coefficients of~$\vec{\lambda}$ agree, so
\[(\phi|X')(\tau,\vec{z}) \= \sum_{\vec{j}}\phi_{0,\vec{j}}(\tau,z)\,(-\vec{\lambda'})^{\vec{j}}\]
for all~$\vec{z}$ of the given form. As before by continuity of~$\phi$ the above equation holds for all~$\vec{z}$. 
\end{proof}
\begin{remark}
The proof of the above result also applies to \emph{weak} Jacobi forms, after replacing `strictly meromorphic Jacobi form' and `(quasi)modular' by `weak Jacobi form' and `weakly holomorphic (quasi)modular', respectively.
\end{remark}

Specializing to \emph{holomorphic Jacobi forms} (instead of \emph{meromorphic} \emph{quasi}-Jacobi forms), we obtain the following result, generalizing the main results on Taylor coefficients of Jacobi forms in \cite{EZ85} to multivariable Jacobi forms. 
\begin{qedcor}\label{cor:mainforJacobi}
Let~$k\in \z, M\in \MM$ and~$\phi\in \Hol^M_n\mspace{1mu}$. 
Then, the following are equivalent:
\begin{enumerate}[\upshape(i)]
\item\label{it:jf1J} The function~$\phi$ is a holomorphic Jacobi form of weight~$k$ and index~$M$.
\item\label{it:jf2J} For all~$X\in \XX$ the function~${g}^X_{\vec{0}}(\mathbf{\phi})$ is a vector-valued modular form satisfying~\eqref{eq:ellcoeff} and transforming as 
\[{g}^X_{\vec{0}}(\mathbf{\phi})|_k\gamma \=  {g}^{X\gamma}_{\vec{0}}(\mathbf{\phi}).\]
\item[\upshape(ii$'$)]\label{it:jf2'J} For all~$X\in \XX$ the function~${g}^X_{\vec{\ell}}(\mathbf{\phi})$ is a vector-valued quasimodular form satisfying~\eqref{eq:ellcoeff} for $\vec{\ell}=\vec{0}$ and transforming as 
\[{g}^X_{\vec{\ell}}(\mathbf{\phi})|_k\gamma \= 
\sum_{r}\frac{1}{r!}\,{g}_{\vec{\ell}}^{X\gamma,r}(\mathbf{\phi})\,\Bigl(\frac{c}{c\tau+d}\Bigr)^{\! r}.\]
\item[\upshape(iii$'$)] For all~$X\in \XX$ and~$\vec{\ell}\in \z^n$ the functions~$\xi^{X}_{\vec{\ell}}(\mathbf{\phi})$ given by
\[\xi_{\vec{\ell}}^X(\mathbf{\phi})\=\sum_{r} (-1)^r\frac{D_\tau^r\,{g}^{X,r}_{\vec{\ell}}(\phi)}{(k+|\vec{\ell}|-1)_{r}\,r!} \]
are modular forms of weight~$k+|\vec{\ell}|$ for~$\Gamma_X$ and satisfy~\eqref{eq:ellcoeff}. 
\item[\upshape(iv)] For all~$X\in \XX$ and~$\vec{\ell}\in \z^n$ the functions~${g}^{X}_{\vec{\ell}}(\mathbf{\phi})$ are quasimodular forms of weight~$k+|\vec{\ell}|$ for~$\Gamma_X$, satisfying~\eqref{eq:ellcoeff} and 
\begin{align}\dtau^r{g}_{\vec{\ell}}^X(\mathbf{\phi}) \= {g}_{\vec{\ell}}^{X,r}(\mathbf{\phi}). &\qedhere\end{align}
\end{enumerate}
\end{qedcor}

\section{Quasimodular algebras for congruence subgroups}\label{sec:cs}
The main result of the previous part, \cref{thm:main}, will almost immediately imply the proof of \cref{thm:higherlevel} on the quasimodularity of the elements of~$\Lambda^*(N)$ for some subgroup. We will first introduce a more general set-up, in the context of which we will present this proof. In the rest of this section 
we provide many examples of algebras of functions on partitions to which this general set-up applies, i.e., we will recall the hook-length moments and the (double) moment functions and explain how the corresponding algebras can be extended to several congruence subgroups.

\subsection{General set-up}\label{sec:set-up}
We answer the question of how to extend an algebra of functions on partitions for which the~$q$-bracket is a quasimodular form for~$\sltwoz$ to one for a congruence subgroup~$\Gamma$. More precisely, we consider \emph{quasimodular algebras} for~$\Gamma$.
\begin{defn} A \emph{quasimodular algebra} for a congruence subgroup~$\Gamma\leq \sltwoz$ is a graded algebra of functions~$f\colon\partitions\to \c$ for which~$\langle f \rangle_q$ is a quasimodular form for~$\Gamma$ of the same weight as~$f$.
\end{defn}

Given a quasimodular algebra for~$\sltwoz$, we now present a construction of a quasimodular algebra for a congruence subgroup. To do so, from now on we assume that~$\Phi\colon\partitions\times \c^r\to \c$ and $k\in \z$ are such that for all $n\geq 1$ the function $\phi_n^\Phi\colon\mathfrak{h}\times M_{n,r}(\c)\to \c$ given by
\[\phi_n^\Phi(\tau,Z)\defis\Bigl\langle \prod_{i=1}^n \Phi(\cdot,Z_i)\Bigr\rangle_{\!q}\mspace{1mu},\]
where~$Z_i$ is the~$i$th row of~$Z$, is a meromorphic quasi-Jacobi form of weight~$k n$ which admits a Laurent expansion around all~$Z\in M_{n,r}(\q)$ (after identifying~$M_{n,r}(\c)$ with~$\c^{nr}$). Here,~$\Phi(\lambda,\vec{z})$ can be thought of a generalisation of the generating series~$W_\lambda(z)$ of the Bloch--Okounkov functions~$Q_k\mspace{1mu}$, defined by~(\ref{eq:BOgen}).

\begin{defn}\label{defn:F} Given such a~$\Phi$, for $\vec{a}\in \q^r$ denote by~$f^\Phi_{\vec{\ell}}(\cdot,\vec{a})=f_{\vec{\ell}}(\cdot,\vec{a})\colon\partitions\to\c$ the~$\vec{\ell}$th Taylor coefficient of~$\Phi(\vec{z})$ around $\vec{z}=\vec{a}$, i.e.,
\[ \Phi(\cdot,\vec{z}) \,=:\, \sum_{\vec{\ell}} f_{\vec{\ell}}(\cdot,\vec{a})\,(\vec{z}-\vec{a})^{\vec{\ell}}.\]
 Define the graded~$\q$-algebra~$\mathcal{F}^\Phi(N)=\mathcal{F}(N)$ as the algebra generated by the weight $k+|\vec{\ell}|$ elements~$f_{\vec{\ell}}(\cdot,\vec{a})$ for $\vec{a} \in \frac{1}{N}\z^r,\vec{\ell}\in \z^r$. 
\end{defn}
\begin{remark}
By \cref{thm:main}, up to a sign, $f_{\vec{m}}(\cdot,\vec{a})$ and~$f_{\vec{m}}(\cdot,\vec{b})$ agree whenever $\vec{a}-\vec{b}\in \z^r$. Hence, in the definition one can assume that $\vec{a}\in [0,1)^r$.  
\end{remark}

For example, letting $\Phi$ be the Bloch--Okounkov generating series $W$, we find
~$Q_{\ell+1}(a)=\e(-\tfrac{1}{2}a)\,f_\ell^W(a)$ (see~\eqref{def:qka}) and $\Lambda^*(N)=\mathcal{F}^W(N)$. 

Now we relate the Taylor coefficients of $\Phi$ to the Taylor coefficients of the corresponding meromorphic quasi-Jacobi forms~$\phi_n^\Phi\mspace{1mu}$. 
Let $L\in M_{n,r}(\z)$ and $A\in M_{n,r}(\q)$. An arbitrary monomial~$f_{L}$ in~$\F^\Phi(N)$ is given by 
\begin{align}\label{eq:monelt} f_L(A)\defis f_{L_1}(A_1)\cdots f_{L_n}(A_n)\end{align}
 with~$L_i$ and~$A_i$ the~$i$th row of~$L$ and~$A$, respectively. 
 Recall that $g_{\vec{\ell}}^X(\phi)$ denotes the $\vec{\ell}$th Taylor coefficient of $\phi$ around $\vec{\lambda}\tau+\vec{\mu}$; see \cref{defn:Taylor}. By construction of these Taylor coefficients, 
 we find
\begin{align}\langle f^\Phi_{L}(A) \rangle_q \= g_{L}^{(0,A)}(\phi_n^\Phi),\end{align}
where on the right-hand side we identified $M_{n,r}(\z)$ and $M_{n,r}(\q)$ with $\z^{nr}$ and $\q^{nr}$, respectively
\footnote{Here we hide a small subtlety. Recall that \cref{defn:Taylor}, in which the Taylor coefficients of a strictly meromorphic Jacobi form are defined, depended not only on this Jacobi form but also on a family of meromorphic functions determined by the transformation of this Jacobi form. We omit this family from the notation, as, for $X=(0,A)$ the double slash operator $\|X$ coincides with the slash operator $|X$, so that the ``Taylor coefficients'' do not involve this family.}.

Recall that for $\widehat{N}\in \z$, we write $m_{\widehat{N}}=\left(\begin{smallmatrix} \widehat{N} & 0 \\ 0 & 1\end{smallmatrix}\right)$. The following result is the general statement of \cref{thm:higherlevel}. 

\begin{thm}\label{thm:higherlevel+}
Given~$\Phi$ as above and $N\geq 1$, let $\widehat{N}=(2,N)N$. The algebra~$\mathcal{F}^\Phi(N)$ is a quasimodular algebra for~$m_{\widehat{N}}^{-1} \Gamma(\widehat{N}) m_{\widehat{N}}\mspace{1mu}$.
\end{thm}
\begin{proof} 
Consider a monomial element $f_L(A)$ of~$\mathcal{F}(N)$ as in~\eqref{eq:monelt}, for some $L\in M_{n,r}(\z)$ and $A\in M_{n,r}(\q)$. 
Write $X=(0,A)$. Then, $\langle f_L(A)\rangle_q = g_{L}^X(\phi_n)$. This Taylor coefficient is quasimodular for~$\Gamma_X$ by \cref{thm:main}. Therefore, it suffices to show that the~$q$-bracket respects the weight grading of~$\mathcal{F}(N)$ and that~$\Gamma_X$ contains~$m_{\widehat{N}}^{-1} \Gamma(\widehat{N}) m_{\widehat{N}} \mspace{1mu}$. 		
												
For the first, observe that the weight of~$f$ is given by $\sum_{i=1}^n (k + |L_i|)$, whereas correspondingly the weight of~$g_{L}^X(\phi)$ equals $kn+|L|$ (here $|L|=\sum_{i,j} L_{ij}$). 

Write $M$ for the index of $\phi_n^\Phi$ and $B=B_M$ for the corresponding bilinear form. Recall
\begin{align}\label{eq:gammaX2}\Gamma_X = \{\gamma\in \sltwoz \mid X\gamma- X \in \XXZ \text{, } \rho(X-X\gamma)=\zeta_{X,X\gamma-X}\}.\end{align}
Writing $\gamma=\abcd$, we have
\begin{align}
X\gamma-X \= (c A, (d-1)A), \quad &\rho(X\gamma-X) \= \e((c^2-c(d-1)+(d-1)^2) B(A,A))	
\end{align}
and
\[
\zeta_{X,X\gamma-X} \= \e(B(0,(d-1)A)-B(cA,A))\=\e(-c\,B(A,A)).
\]
Observe that $2N^2\,B(A,A)$ is integral. Hence, if $\gamma \in \sltwoz$ satisfies
\begin{equation}\label{eq:congcond}c \equiv 0 \bmod N, \quad d\equiv 1 \bmod N\quad \text{and} \quad c^2-c(d-1)+(d-1)^2\equiv c \bmod 2N^2,\end{equation}
then $\gamma \in \Gamma_X\mspace{1mu}$. 

Let $N'\in \z_{>0}\mspace{1mu}$. Then,
\[m_{N'}^{-1} \Gamma(N') m_{N'} = \{ \abcd \in \sltwoz \mid c \equiv 0\,(N'^2),\, a\equiv d\equiv 1\,(N')\}.\]
In case $2\nmid N$, the conditions~(\ref{eq:congcond}) are satisfied for all $\gamma \in m_{N'}^{-1} \Gamma(N') m_{N'}$ when $N'=N$, in case $2\mid N$ for $N'=2N$. Therefore, $m_{\widehat{N}}^{-1} \Gamma(\widehat{N}) m_{\widehat{N}} \leq \Gamma_X\mspace{1mu}$. 
\end{proof}
																													
For a monomial element $f_L(A)$ as in~\eqref{eq:monelt} with $L\in M_{n,r}(\z)$ and $A\in M_{n,r}(\q)$, and~$\gamma 
\in \Gamma_1(N)$ one has that
\begin{align}\label{eq:omega2}\langle f_L(A)\rangle_q \big|\gamma \=  \e((c^2-cd+(d-1)^2) \, B(A,A))\, \langle f_L(A) \rangle_q\, , \end{align}
where $B$ is again the bilinear form corresponding to the index of $\phi_n^\Phi\mspace{1mu}$. 
Hence, restricting to $A\in M_{n,r}(\q)\simeq \q^{nr}$ for which $\e((c^2-cd+(d-1)^2) \, B(A,A))=1$ for all $\gamma\in \Gamma_1(N)$, we find the following result. This result allows us to derive \cref{thm:bohigherlevel}, using some additional properties of the Bloch--Okounkov~$n$-point functions.
\begin{prop}\label{prop:higherlevel+}
Given~$N\geq 1$, for all $L\in M_{n,r}(\z)$ and $A\in \frac{1}{N}M_{n,r}(\z)$ satisfying ${B(A,A)\in \frac{1}{2}\z+\frac{1}{N}\z}\mspace{1mu}$, one has that
$\langle f_L(A)\rangle_q$ is a quasimodular form for~$\Gamma_1(N)$. 
\end{prop}
\begin{proof}
This follows from~(\ref{eq:omega2}) using the following two observations. First of all,  $c^2-cd+(d-1)^2\equiv 0 \bmod N$ when $c\equiv 0, d\equiv 1 \bmod N$. Secondly, for integers~$c,d$ the integer $c^2-cd+(d-1)^2$ is always even whenever not both~$c$ and~$d$ are even. 
\end{proof}

In the introduction, we introduced certain functions~$Q_k^{(p)}$ introduced in \cite{GJT16}. In that work the authors let
\[ Q_k^{(p)}(\lambda) \= \beta_k^{(p)} \, + \sum_{\gcd(2\lambda_i-2i+1,p)=1} \!\bigl((\lambda_i-i+\tfrac{1}{2})^{k-1}-(-i+\tfrac{1}{2})^{k-1}\bigr),\]
where $\beta_k^{(p)} = \beta_k(0)(1-\tfrac{1}{p})$. 
Observe that for primes~$p$ 
 this agrees with the definition in~\eqref{def:Qkp} in the introduction.
Similarly, we define functions~$f_{\vec{\ell}}^{\vec{d}}$ in terms of the Taylor coefficients~$f_{\vec{\ell}}(\vec{a})$.

\begin{defn} Let~$\Phi$ be as above. 
Given $\vec{d}\in \z_{>0}^r\mspace{2mu}$, we let
\[U(\vec{d}) =  \bigl\{0,\tfrac{1}{d_1},\ldots,\tfrac{d_1-1}{d_1}\bigr\} \times \cdots \times  \bigl\{0,\tfrac{1}{d_r},\ldots,\tfrac{d_r-1}{d_r}\bigr\} \]
and for $\vec{\ell}\in \z^r$ define
\[f_{\vec{\ell}}^{\vec{d},\Phi}(\lambda) \= f_{\vec{\ell}}^{\vec{d}}(\lambda) \defis \sum_{\vec{a}\in U(\vec{d})} f_{\vec{\ell}}^\Phi(\lambda,2\vec{a}) \qquad (\lambda \in \partitions).\]
Define the graded algebra~$\mathcal{F}^{(N)}=\mathcal{F}^{(N),\Phi}$ as the $\q$-algebra generated by the functions~$f_{\vec{\ell}}^{\vec{d},\Phi}$ for all $\vec{\ell}\in \z^r$ and $\vec{d}\in \z^r_{>0}$ for which $d_i\mid N$ for all~$i$. 
\end{defn}

Then, \cref{thm:bohigherlevel2} follows directly from the following result.
		   
\begin{thm}\label{thm:bohigherlevel2+} Let $\Phi$ be as above. Given~$N\geq 1$, the algebra~$\mathcal{F}^{(N),\Phi}$ is a quasimodular algebra for the congruence subgroup~$\Gamma_0(N^2)$.  
\end{thm}
\begin{proof} 
Consider the monomial elements $f_{L}^{D}:=f_{L_1}^{D_1}\cdots f_{L_n}^{D_n}$ in~$\mathcal{F}^{(N)}$, where $L,D\in M_{nr}(\z)$. Everywhere in this proof we identify $M_{n,r}(\z)$ with $\z^{nr}$. Then,
\[ f_L^{D} \= \sum_{A\in U(D)} f_{L}(\cdot, 2A).\]
Now, by part~(\ref{it:jf2'}$'$) in \cref{thm:main}, for all $\gamma=\abcd \in \sltwoz$ one has that
\[\langle  f_{L}(\cdot, 2A) \rangle_q|\gamma \= \sum_{r\geq 0} h_r(2cA,2dA) \Bigl(\frac{c}{c\tau+d}\Bigr)^{\! r},\]
where $h_r(\vec{\lambda},\vec{\mu}) = \sum_{s=0}^r\frac{1}{(r-s)!}\,{g}_{\vec{\ell},s}^{X,r-s}(\phi_n)$ and $X=\lm$. Note that $h_r$ is zero for all but finitely many $r$. 

If $\gamma \in \Gamma_0(N^2)$, then $2cA \in 2N M_{n,r}(\z)$. Hence, for $X=(0,2dA)$ and $X'=(2bA,0)$ one has $\rho(X')\zeta_{X',X}=1$. Therefore,
\[h_r(2cA,2dA) \= h_r(0,2dA) \qquad \text{and}\qquad h_r(0,2dA+B)\=h_r(0,2dA)\]
for all $B\in 2M_{n,r}(\z)$. As~$2dA$ ranges over the same values modulo~$2$ as~$2A$ does for $A\in U(D)$, one finds
\[\langle  f_{L}^{D} \rangle_q|\gamma \= \sum_{r\geq 0}\sum_{A\in U(D)} h_r(0,2A) \Bigl(\frac{c}{c\tau+d}\Bigr)^{\! r}\]
for all $\gamma \in \Gamma_0(N^2)$. Hence,~$\langle  f_{L}^{D} \rangle_q$ is a quasimodular form for~$\Gamma_0(N^2)$. 
\end{proof}
\begin{remark}
In fact, given $\vec{d}\in \z_{>0}^r$ with $d_i\mid N^2$ for all $i$, and Dirichlet characters $\chi_1,\ldots,\chi_r$ modulo $d_1,\ldots, d_r\mspace{1mu}$, respectively, the series given by
\[\biggl\langle\sum_{\vec{a}\in U(\vec{d})} \overline{\chi_1}(a_1d_1)\cdots\overline{\chi_r}(a_rd_r)\,f_{\vec{\ell}}^\Phi(\cdot,2\vec{a})\biggr\rangle_{\!q}\]
are quasimodular forms for $\Gamma_0(N^2)$ of character $\chi_1\cdots\chi_r\mspace{1mu}.$
\end{remark}
The rest of this part is devoted to providing examples of quasimodular algebras of higher level, using \cref{thm:bohigherlevel2+}.

\subsection{First application: the Bloch--Okounkov theorem of higher level}\label{sec:BON}
The results on the Bloch--Okounkov algebra, as stated in the introduction, are proven in this section. The proofs follow almost immediately from the results in the previous section using the properties of the Bloch--Okounkov~$n$-point functions. 

Recall the Bloch--Okounkov~$n$-point functions $F_n$ are defined (in \cref{defn:BOnpoint}) as follows. For all $n\geq 0$, let~$\mathfrak{S}_n$ be the symmetric group on~$n$ letters and let
\[F_n(\tau,z_1,\ldots,z_n) \= \sum_{\sigma \in \mathfrak{S}_n} V_n(\tau,z_{\sigma(1)},\ldots,z_{\sigma(n)}),\]
where the functions~$V_n$ are defined recursively by~$V_0(\tau)=1$ and
\begin{equation}\sum_{m=0}^n \frac{(-1)^{n-m}}{(n-m)!}\;\theta^{(n-m)}(\tau,z_1+\ldots+z_m)\;V_m(\tau,z_1,\ldots,z_m)\=0.\end{equation}
These $n$-point functions~$F_n$ are quasi-Jacobi forms of which we determine the weight and index (or rather the bilinear form uniquely determining a symmetric matrix $M\in \MM$ which is the index).
\begin{lem}
The~$n$-point functions~$F_n$ are meromorphic quasi-Jacobi forms of weight~$n$ and index $B(\vec{z},\vec{z}) = -\tfrac{1}{2}(z_1+\ldots+z_n)^2$. 
\end{lem}
\begin{proof}
We start with the observation that for all $n\geq 0$ the function~$\frac{\Theta^{(n)}(z)}{\Theta(z)}$ is a true meromorphic Jacobi form (of weight~$n$ and index $B(z,z)=0$), in contrast to~$\Theta(z)$ itself which is \emph{weakly} holomorphic. Namely, all poles are given by $z\in \z\tau+\z$ and for $z=a\tau+b$ with $a,b\in \q$, one has that
\[\frac{\Theta^{(n)}(a\tau+b)}{\Theta(a\tau+b)} = \frac{\sum_{\nu\in \mathbb{F}} \nu^n\, \e(\nu b)\, q^{\nu^2/2+a\nu}}{\sum_{\nu\in \mathbb{F}} \e(\nu b)\, q^{\nu^2/2+a\nu}} \to -a,\]
whenever $\Im\tau\to \infty$, or equivalently $q\to 0$. 

Next, observe that $\Theta(z_1+\ldots+z_n) \, V_n(z_1,\ldots,z_n)$ can be written as a polynomial of weight~${n-1}$ in $\frac{\Theta^{(i)}(z_1+\ldots+z_j)}{\Theta(z_1+\ldots+z_j)}$ for $i,j=1,\ldots,n$; a fact which can be proven inductively by its recursion~(\ref{eq:B-Ocor}). Hence, $\Theta(z_1+\ldots+z_n) \, V_n(z_1,\ldots,z_n)$ is a meromorphic Jacobi form. As~$\Theta(z)^{-1}$ is a meromorphic Jacobi form of weight~$1$ and index given by the bilinear form $B(z,z)=-\tfrac{1}{2}z^2$, we conclude that~$F_n$ is a meromorphic quasi-Jacobi form of weight~$n$ and index $B(\vec{z},\vec{z}) = -\tfrac{1}{2}|\vec{z}|^2$. 
\end{proof}

Observe that \cref{thm:higherlevel} is a direct corollary of the previous lemma and \cref{thm:higherlevel+}. Also, \cref{thm:bohigherlevel2} follows directly from the previous lemma and \cref{thm:bohigherlevel2+}. So, we are only left with the following proof. 

\begin{proof}[Proof of \cref{thm:bohigherlevel}]
First of all, in case $a\in \z$, one has $B(\vec{a},\vec{a})=-\tfrac{1}{2}|\vec{a}|^2\in \frac{1}{2}\z$. Hence, the result follows directly from \cref{prop:higherlevel+}. 

If $a\not \in \z$, for both $\langle Q_{\vec{k}}(\vec{a})\rangle_q$ and $\langle Q_{1}(a)\rangle_q$ the root of unity in~(\ref{eq:omega2}) is $\e(B(\vec{a},\vec{a}))=\e(B(a,a))$, so that the subgroup of quasimodularity for their ratio is~$\Gamma_1(N)$. 

For the holomorphicity in the second case, observe that~$\langle Q_{1}(a)\rangle_q$ equals~$\Theta(a)^{-1}$ up to a constant. Also, $\langle Q_{\vec{k}}(\vec{a})\rangle_q$ can be written as a product of Taylor coefficients of the function~$\Theta(z_1+\ldots+z_n+a)^{-1}$ and of $\Theta(z_1+\ldots+z_n+a)\, F_n(z_1+a_1,\ldots,z_n+a_n)$, the latter Taylor coefficients being holomorphic quasimodular forms. Observe that the Taylor coefficients around $z_1=\ldots=z_n=0$ of 
\[\frac{\Theta(a)}{\Theta(z_1+\ldots+z_n+a)}\]
 are all polynomials in the holomorphic quasimodular forms~$\frac{\Theta^{(i)}(a)}{\Theta(a)}.$ Therefore, $\frac{\langle Q_{\vec{k}}(\vec{a})\rangle_q}{\langle Q_{1}(a)\rangle_q}$ is a holomorphic quasimodular form. 
\end{proof}

\subsection{Second application: hook-length moments of higher level}
As a second example, consider the hook-length moments
\begin{equation}\label{eq:hlm} H_k(\lambda) \defis -\frac{B_k}{2k} \+ \sum_{\xi \in Y_\lambda} h(\xi)^{k-2} \qquad (k\geq 2), \end{equation}
where~$B_k$ is the $k$th Bernoulli number, $Y_\lambda$ denotes the Young diagram of~$\lambda$,~$\xi$ is a cell in this Young diagram, and~$h(\xi)$ denotes the hook-length of this cell. By \cite[Theorem~13.5]{CMZ16} one has that~$H_k(\lambda)$ is (up to a constant) equal to the~$(k-2)$th order Taylor coefficient of~$W_\lambda(z)\,W_\lambda(-z)$. In particular, any homogeneous polynomial in the~$H_k$ admits a quasimodular~$q$-bracket. 
The results in \cref{sec:set-up} now specialize to the following statements.
\paragraph*{(i).} For~$a\in \q$ and~$k\in \z_{\geq 2}\mspace{1mu}$, let 
\[ H_k(\lambda,a) \defis \tilde{\alpha}_k(a) \+\frac{1}{2} \sum_{\xi \in Y_\lambda} \bigl(\e(a\, h(\xi))+(-1)^k\,\e(-a\, h(\xi))\bigr)\,h(\xi)^{k-2} ,\]
where 
\[\frac{\tilde{\alpha}_{-2}(a)}{z^2}\+\frac{\tilde{\alpha}_{-1}(a)}{z}\+\sum_{k\geq 0} \tilde{\alpha}_k(a)\frac{z^{k-2}}{(k-2)!} \defis \frac{1}{8}\sinh\Bigl(\frac{z+2 \pi\ii a}{2}\Bigr)^{-2}.\] 
Denote by~$\mathcal{H}(N)$ the algebra generated by the~$H_k(\cdot,a)$ with $k\geq 2$ and $a\in \frac{1}{N}\z$. The algebra~$\mathcal{H}(N)$ is graded by assigning to~$H_k(\cdot,a)$ weight~$k$. Let $\widehat{N}=(2,N)N$.
\begin{cor} The algebra~$\mathcal{H}(N)$ is quasimodular of level~$\widehat{N}$, after scaling~$\tau$ by~$\widehat{N}$. \end{cor}
More concretely,  for all homogeneous $f\in \mathcal{H}(N)$ of some weight~$k$, the rescaled~$q$-bracket~$\langle f \rangle_{q_{\widehat{N}}}$, where $q_{\widehat{N}}=q^{1/\widehat{N}}$, is a quasimodular form of weight~$k$ for~$\Gamma(\widehat{N})$.
\paragraph*{(ii).} Given $N\geq 1$, $\vec{k}\in \z^n$ and $\vec{a}\in \frac{1}{N}\z^n$, write $H_{\vec{k}}(\lambda,\vec{a}) = H_{k_1}(\lambda,a_1)\cdots H_{k_n}(\lambda,a_n)$. \begin{cor}
For $\vec{a}\in \frac{1}{N}\z^n$ with $|\vec{a}|\in \z$ and $\vec{k}\in \z_{\geq 2}^n$ the~$q$-bracket~$\langle H_{\vec{k}}(\cdot,\vec{a})\rangle_q$ is a quasimodular form of weight~$|\vec{k}|$ for~$\Gamma_1(N)$.
\end{cor}

\paragraph*{(iii).} Let
\[ H_k^{t}(\lambda) \defis -\frac{B_k}{2k}t^k \+\sum_{\substack{\xi \in Y_\lambda \\ h(\xi)\equiv 0 \bmod t}} h(\xi)^{k-2} \qquad (k,t\in \z_{> 0}),\] 
which (up to a constant) also occurs in \cite{BOW20}.
Denote by~$\mathcal{H}^{(N)}$ the algebra generated by the~$H_k^{t}$ for which~$k$ is even and $t\mid N$. This algebra is graded by assigning weight~$k$ to~$H_k^{t}\mspace{1mu}$. 
\begin{cor}
The algebra~$\mathcal{H}^{(N)}$ is quasimodular for~$\Gamma_0(N^2)$. 
\end{cor}
More concretely, for all homogeneous $f\in \mathcal{H}^{(N)}$ of weight~$k$, the~$q$-bracket~$\langle f \rangle_q$ is a quasimodular form of weight~$k$ for~$\Gamma_0(N^2)$.

\subsection{Third application: moment functions of higher level}
Next, we consider the moment functions in \cite{Zag16}
\[ S_k(\lambda) \defis -\frac{B_k}{2k} \+ \sum_{i=1}^\infty \lambda_i^{k-1} \qquad (k\geq 1). \]
The generating series $\mathscr{S}(z)\colon\partitions\to \q$ given by $\mathscr{S}(z) := \frac{1}{2z^2} + \sum_{k\geq 2} S_k \frac{z^{k-2}}{(k-2)!}$ satisfies \cite[Corollary~3.3.2]{vI20}
\begin{align}\label{eq:gs3}
\langle \mathscr{S}(z_1)\cdots \mathscr{S}(z_n)\rangle_q \= \frac{1}{2^{n+1}}\sum_{\alpha \in \Pi(n)}\prod_{A\in \alpha} \sum_{\vec{s}\in \{-1,1\}^{|A|}}\!D_\tau^{|A|-1}E_2(\vec{s}\cdot \vec{z}_A),\end{align}
where~$\Pi(n)$ denotes the set of all set partitions of the set $\{1,\ldots,n\}$,~$|A|$ the cardinality of the set~$A$, and~$\vec{z}_A=(z_{a_1},\ldots,z_{a_r})$ if $A=\{a_1,\ldots,a_r\}$. Hence, the~$n$-point functions~\eqref{eq:gs3} are quasi-Jacobi forms of weight~$2n$ and index zero. Because the index of the quasi-Jacobi forms~\eqref{eq:gs3}  is zero, the root of unity in~(\ref{eq:omega2}) vanishes for all $\gamma \in \sltwoz$. Therefore, the results of \cref{sec:set-up} specialize to the following results for the groups~$\Gamma_1(N)$ and~$\Gamma_0(N)$ (rather than~$\Gamma(N)$ and~$\Gamma_0(N^2)$, respectively).  

\paragraph*{(i) \& (ii).}  Let 
\[ S_k(\lambda,a) \defis \alpha_k(a) \+\frac{1}{2} \sum_{i=1}^\infty \bigl(\e(a\lambda_i)+(-1)^k\e(-a\lambda_i)\bigr)\lambda_i^{k-1} \qquad (a\in \q, k\geq 1),\]
where 
\begin{align}\label{def:alpha}\frac{\alpha_{-1}(a)}{2\pi\ii z}\+\sum_{k\geq 0} \alpha_k(a)\frac{(2\pi\ii z)^{k-1}}{(k-1)!} \defis \frac{1}{2}(\e(z+a)-1)^{-1}\end{align} (for $k\geq 2$, these values agree with the constants~$\tilde{\alpha}_k$ in the previous example). 
Denote by~$\mathcal{S}(N)$ the algebra generated by the~$S_k(\cdot,a)$ with $a\in \frac{1}{N}\z$. Assign to~$S_k(\cdot,a)$ weight~$k$.
\begin{cor} The algebra~$\mathcal{S}(N)$ is quasimodular for~$\Gamma_1(N)$. \end{cor}
More concretely, for all homogeneous $f\in \mathcal{H}(N)$ of weight~$k$, the~$q$-bracket~$\langle f \rangle_q$  is a quasimodular form of weight~$k$ for~$\Gamma_1(N)$.

\paragraph*{(ii).}  Let
\[ S_k^{t}(\lambda) \defis -\frac{B_k}{2k}t^k \+\sum_{\substack{i\geq 0 \\ \lambda_i \equiv 0 \bmod t}} \lambda_i^{k-1} \qquad (k,t\in \z_{>0}).\] 
Denote by~$\mathcal{S}^{(N)}$ the algebra generated by the~$S_k^{t}$ for which $k$ is even and $t\mid N$. Assign to~$S_k^{t}$ weight~$k$.
\begin{cor} The algebra~$\mathcal{S}^{(N)}$ is quasimodular for~$\Gamma_0(N)$.\end{cor}
More concretely, for all homogeneous $f\in \mathcal{S}^{(N)}$ of weight~$k$, the~$q$-bracket~$\langle f \rangle_q$ is a quasimodular form of weight~$k$ for~$\Gamma_0(N)$.

\subsection{Fourth application: double moment functions of higher level}\label{sec:double}
As a final example, consider the double moment functions introduced in \cite{vI20} given by
\begin{equation}\label{def:Tkl0} T_{k,l}(\lambda) \defis  C_{k,l}\+\sum_{m=1}^\infty m^k \Faulhaber_l(r_m(\lambda)) \qquad (k\geq 0, l\geq 1).\end{equation}
Here,~$C_{k,l}$ is a constant equal to~$-\frac{B_{k+l}}{2(k+l)}$ if~$k=0$ or~$l=1$ and~$0$ else,~$\Faulhaber_l$ is the \emph{Seki--Bernoulli polynomial} of positive integer degree~$l$, defined by~$\Faulhaber_l(n):=\sum_{i=1}^n i^{l-1}$ for all~$n\in\z_{>0}\mspace{1mu}$, and  the \emph{multiplicity}~$r_m(\lambda)$ of parts of size~$m$ in a partition~$\lambda$ is defined as the number of parts of~$\lambda$ of size~$m$. The generating series $\mathscr{T}(z,w) := -\frac{1}{2z}-\frac{1}{2w}+\sum_{k+l\equiv 0 (2)} T_{k,l} \,\frac{z^{k}w^{\ell-1}}{(k)!(\ell-1)!}$ satisfies \cite[Theorem~4.4.1.]{vI20}				  
\[ G_1(z,w)\defis\langle \mathscr{T}(z,w) \rangle_q \= -\frac{1}{2}\frac{\Theta(z+w)}{\Theta(z)\Theta(w)}.\]
Hence, the~$1$-point function~$G_1$ is a Jacobi form of weight~$1$ and index~$B((z,w),(z,w))=zw$. This example now deviates from the previous ones because $\langle \mathscr{T}(z_1,w_1) \cdots \mathscr{T}(z_n,w_n) \rangle_q$ 
is not a Jacobi form of some \emph{fixed} weight (but rather a linear combination of functions of different weights). It is, therefore, that we have to consider a different product~$\blok$ on the space of functions of partitions, for which
\begin{align}\label{eq:gn} G_n(\vec{z},\vec{w})\defis \langle \mathscr{T}(z_1,w_1) \blok \cdots \blok \mathscr{T}(z_n,w_n) \rangle_q \= (-1)^n \prod_{i=1}^n \frac{1}{2}\frac{\Theta(z_i+w_i)}{\Theta(z_i)\Theta(w_i)}.\end{align}
To define this product consider the isomorphism $\c^\partitions \to \c[\![u_1,u_2,\ldots]\!]$, $f\mapsto \langle f \rangle_{\vec{u}}\mspace{1mu}$, given by 
\begin{align}
\quad \langle f\rangle_{\vec{u}} \defis \frac{\sum_{\lambda\in \partitions} f(\lambda)\, u_\lambda }{\sum_{\lambda \in \partitions} u_\lambda} \qquad (u_\lambda=u_{\lambda_1} u_{\lambda_2}\cdots).\end{align}
Then, the \emph{induced product} is defined by
\begin{align}\label{eq:blok}\langle f\blok g \rangle_{\vec{u}}\defis \langle f\rangle_{\vec{u}}\,\langle g\rangle_{\vec{u}}\,.
\end{align}
The algebra~$\mathcal{T}$ generated by the~$T_{k,l}$ contains exactly the same elements as the algebra consisting of polynomials in the moment functions~$T_{k,l}$ with multiplication given by the induced product~$\blok$ \cite{vI20}.  
 
In \cref{sec:set-up} one can replace the pointwise product on functions of partitions by the product~$\blok$. Therefore, we have the following generalisations of the algebra~$\mathcal{T}$.

\paragraph*{(i).} For $a,b\in \q$, and $k\geq 0, \ell\geq 1$, let
\[ T_{k,\ell}(\lambda,a,b) \defis C_{k,\ell}(a,b) \+ \sum_{m=1}^\infty m^k\bigl(\e(am)\, \Faulhaber_\ell^b(r_m(\lambda))+(-1)^{k+l}\e(-am)\, \Faulhaber_\ell^{-b}(r_m(\lambda))\bigr) ,\]
where 
\[C_{k,\ell}(a,b) \defis \begin{cases} \alpha_k(a) & \ell=1 \\ \alpha_{\ell-1}(b) & k=0 \\ 0 & \text{else,}\end{cases}\]
and the constants~$\alpha_k$  are defined by~\eqref{def:alpha}. Also,~$\Faulhaber_\ell^b$ 
is defined by $\Faulhaber_\ell^b(n) := \sum_{i=1}^n \e(bi) \, i^{l-1}$ for all~$n\in\z_{>0}\mspace{1mu}$. 
Let the~$\q$-algebra~$\mathcal{T}(N)$ be generated by the functions~$T_{k,\ell}(\cdot,a,b)$, where $a,b\in \frac{1}{N}\z$, \emph{under the induced product}. Assign to~$T_{k,\ell}(\cdot,a,b)$ weight~$k+\ell$ and extend to a weight grading under the induced product. Let $\widehat{N}=(2,N)N$.
 
\begin{cor} The algebra~$\mathcal{T}(N)$ is quasimodular of level~$\widehat{N}$, after scaling~$\tau$ by~$\widehat{N}$. \end{cor}

More concretely,  for all homogeneous $f\in \mathcal{T}(N)$ of some weight~$k$, the rescaled~$q$-bracket~$\langle f \rangle_{q_{\widehat{N}}}$, where $q_{\widehat{N}}=q^{1/\widehat{N}}$, is a quasimodular form of weight~$k$ for~$\Gamma(\widehat{N})$.

\paragraph*{(ii).}
\begin{cor}\label{cor:Tklii} Let $\vec{k},\vec{\ell}\in \z^n$ and $\vec{a},\vec{b}\in \frac{1}{N}\z^n$. Whenever $\vec{a}\cdot \vec{b}\in \frac{1}{2}\z + \frac{1}{N}\z$, one has that
\[\langle T_{k_1,\ell_1}(\cdot,a_1,b_1)\blok\cdots\blok T_{k_n,\ell_n}(\cdot,a_n,b_n)\rangle_q\]
 is a quasimodular form of weight~$|\vec{k}|+|\vec{\ell}|$ for~$\Gamma_1(N)$.
 \end{cor}
 
\paragraph*{(iii).} Let
\[ T_{k,\ell}^{s,t}(\lambda) \defis C_{k,\ell} \+\sum_{m=1}^\infty m^k \,\Faulhaber_{\ell}\Bigl(\Bigl\lfloor \frac{r_{ms}(\lambda)}{t}\Bigr\rfloor\Bigr) \qquad (k,\ell,s,t\in \z_{>0}).\] 
Denote by~$\mathcal{T}^{(N)}$ the algebra generated by the~$T_{k,\ell}^{s,t}$, where $k,\ell$ are even and $s,t\mid N$, \emph{under the induced product}. Assign to~$T_{k,\ell}^{s,t}$ weight~$k+\ell$ and extend to a weight grading under the induced product. 
\begin{cor} 
The algebra~$\mathcal{T}^{(N)}$ is quasimodular for~$\Gamma_0(N^2)$. 
\end{cor} 
More concretely, for all homogeneous $f\in \mathcal{H}^{(N)}$ of weight~$k$, the~$q$-bracket~$\langle f \rangle_q$ is a quasimodular form of weight~$k$ for~$\Gamma_0(N^2)$. 
\begin{remark}
Combining this result with~\cref{cor:Tklii}, by restricting to $t=1$, we find that any polynomial in~$T_{k,\ell}^{s,1}$ with respect to the induced product is quasimodular for~$\Gamma_0(N)$. A similar statement holds after restricting to $s=1$. 
\end{remark}

\section{When is the~\texorpdfstring{$q$}{q}-bracket modular?}\label{sec:when}
We first state and prove our answer to the question in the title of this section in full generality, using the main result on the Taylor coefficients of quasi-Jacobi forms (\cref{thm:main}). Afterwards, we provide many examples, i.e., we prove the results on the `modular subspace' of the Bloch--Okounkov algebra as stated in the introduction, and state similar results for the Bloch--Okounkov algebra for congruence subgroups as well as the algebra of double moment functions.

\subsection{Construction of functions with modular~\texorpdfstring{$q$}{q}-bracket}\label{sec:whenconstruction}
Recall $D_\tau=\frac{1}{2\pi\ii}\pdv{}{\tau},$ $\mathbbm{e}_2=\frac{1}{12}-2\sum_{m,r\geq 1}m \,q^{mr}$ is the quasimodular Eisenstein series of weight $2$ and ${g}_{\vec{\ell},s}^{r-s}(\phi)$ is a ``Taylor coefficient'' of~$\phi$ defined by \cref{defn:Taylor}.
Given a quasi-Jacobi form~$\phi$ of weight~$k$ satisfying the conditions of~\cref{thm:main}, the functions
\[\xi_{\vec{\ell}}({\phi}) \= \displaystyle \sum_{r}(-1)^r\sum_{s\leq r}\frac{(D_\tau+\mathbbm{e}_2)^r\,{g}_{\vec{\ell},s}^{r-s}(\phi)}{(k+|\vec{\ell}|-r-\frac{3}{2})_{r}\,(r-s)!},\]
are modular forms (see also \cref{prop:equiv}). Therefore, as in the previous section, assume that~$\Phi\colon\partitions\times \c^r\to \c$ and $k\in \z$
are such that for all $n\geq 1$ the function $\phi_n^\Phi\colon\mathfrak{h}\times M_{n,r}(\c)\to \c$ given by
\[\phi_n^\Phi(\tau,Z)\defis\Bigl\langle \prod_{i=1}^n \Phi(\cdot,Z_i)\Bigr\rangle_{\!q}\,,\]
where~$Z_i$ is the~$i$th row of~$Z$, is a meromorphic quasi-Jacobi form of weight~$k n$ which admits a Laurent expansion around all~$Z\in M_{n,r}(\q)$ (after identifying~$M_{n,r}(\c)$ with~$\c^{nr}$). Write $\F=\F^\Phi(1)$ for the graded algebra of Taylor coefficients of~$\Phi$ (see \cref{defn:F}). Given $\vec{\ell}\in \z^n$, our aim is to find $h_{\vec{\ell}}\in \F$ such that $\langle h_{\vec{\ell}}\rangle_q = \xi_{\vec{\ell}}(\phi_n)$, i.e., we want to determine a pullback of~$\xi_{\vec{\ell}}(\phi_n)$ under the~$q$-bracket. Observe that these pullbacks $h_{\vec{\ell}}$ for all $\vec{\ell}$ define a unique linear map $\mathcal{F}\to\mathcal{F}$ so that $\pi(f_{\vec{\ell}})=h_{\vec{\ell}}\mspace{1mu}$. Note that in this case $\langle \pi(f)\rangle_q$ is modular for all $f\in \mathcal{F}$.

To define~$\pi$, assume the algebra~$\F$ satisfies the following two properties:
\begin{enumerate}[{\upshape (i)}]\itemsep2pt
\item $Q_2\in \F$,
\item There is a linear operator~$\mathscr{D}$ acting on~$\F$ such that 
\begin{align}\label{eq:D}\Bigl\langle \mathscr{D}\prod_{i=1}^n \Phi(\vec{z}^i)\Bigr\rangle_q \= (\dtau+z_1\dd_{z_1}+\ldots+z_n\dd_{z_n})\,\phi,\end{align}
where $\mathscr{D}$ is extended to a linear operator on $\F[\![z_1,\ldots,z_n]\!]$ by $\mathscr{D}(f\,\vec{z}^{\vec{\ell}})=\mathscr{D}(f)\,\vec{z}^{\vec{\ell}}$ for all $f\in \F$ and $\vec{\ell}\in \z^n$. 
\end{enumerate}
\begin{remark} As observed in \cite{Zag16}, by an easy computation one finds that the function~$Q_2$  makes the~$q$-bracket equivariant with respect to the operator~$D_\tau+\mathbbm{e}_2\mspace{1mu}$, i.e.\
\[\langle Q_2 f\rangle_q = (D+\mathbbm{e}_2)\langle f \rangle_q\]
for all $f\colon\partitions\to \c$. This motivates the first condition. The second condition is motivated by noting that
\begin{equation}\label{eq:mathscrD} g_{\vec{\ell},s}(\phi) \= {g}_{\vec{\ell}}\Bigl(\sum\nolimits_{i+|\vec{j}|=s}\phi_{i,\vec{j}} \,\frac{\vec{z}^{\vec{j}}}{(2\pi\ii)^i}\Bigr) 
\= \frac{1}{s!} g_{\vec{\ell}}\bigl((\dtau+z_1\dd_{z_1}+\ldots+z_n\dd_{z_n})^s \phi \bigr).\qedhere\end{equation}
\end{remark}

Recall that the algebra~$\F=\F^\Phi(1)$ is generated by the Taylor coefficients $f_{\vec{\ell}}=f_{\vec{\ell}}(\vec{0})$ of~$\Phi$ for $\vec{\ell}\in \z^r$. An arbitrary monomial~$f_{L}$ in~$\F$ is given by $f_{L_1}\cdots f_{L_n}$ with $L\in M_{n,r}(\z)$.  We define an operator on~$\F$ corresponding to the index of~$\phi$.
\begin{defn}\label{defn:Q} Let $M=(m_{i,j})\in \MM$ be the index of $\phi_n^\Phi$. Define~$\mathscr{M}$ to be the linear operator on~$\F$ which is given on monomials by
 \[\mathscr{M} f_{L} \defis \sum_{i,j} m_{i,j} f_{L-e_i-e_j} \qquad (L\in M_{n,r}(\z)\simeq \z^{nr}),\]
 where, on the right-hand side, $e_i$ is a unit vector in $\z^{nr}$. 
\end{defn}
 Now, \cref{thm:when} can more explicitly be stated as follows, where the three properties below should be compared with the three properties satisfied by the functions~$h_k$ in the introduction (\cref{sec:introwhen}).
 
By $\modular$ denote the algebra of modular forms for $\sltwoz$ \emph{with rational Fourier coefficients}. 
\begin{thm}\label{thm:when+} Let $\F$ be a quasimodular algebra satisfying the above conditions, and $\mathscr{M},\mathscr{D}:\F\to\F$ the operators defined by \cref{defn:Q} and Equation~\ref{eq:D}, respectively. Then, the linear mapping~$\pi\colon\F\to \F$ given by										   
\[ \pi(f)\= \sum_{r\geq 0}\sum_{s=0}^r(-1)^r\frac{Q_2^r\,\mathscr{M}^{r-s}\mathscr{D}^sf}{(m-r-\frac{3}{2})_{r}\,(r-s)!\,s!}\]
on~$f\in \F$ of homogeneous weight~$m$ satisfies
\[\langle \pi(\F)\rangle_q \subseteq M.\]

Furthermore, one can choose a vector subspace $\mathcal{M}\subseteq \pi\F$ such that \vspace{-3pt}
\begin{enumerate}[{\upshape (i)}]\itemsep2pt
\item \label{it:splitting} $\F = \mathcal{M} \oplus Q_2\F;$
\item $\langle \mathcal{M}\rangle_q \subseteq \modular;$
\item $\langle Q_2\F\rangle_q \cap \modular = \{0\}.$
\end{enumerate}
\end{thm}
\begin{remark}
For the Bloch--Okounkov algebra~$\Lambda^*$ the mapping~$\pi$ turns out to be the canonical projection of $\Lambda^*$ on $\mathcal{H}$ in \cite{vI18} and $\mathcal{M}=\mathcal{H}$. We do not expect that the conditions in this section ensure that~$\pi$ is a projection in general, nor that $\pi(Q_2\F)=\{0\}$ (in which case one could choose $\mathcal{M}=\pi(\F)$), nor that the splitting in~(\ref{it:splitting}) is canonical. However,
once one has chosen $\mathcal{M}$ it follows immediately from~(\ref{it:splitting}) that every element of $\F$ has a canonical expansion
\[f\=\sum_{i\geq 0} f_i\, Q_2^i\]
with $f_i\in \mathcal{M}$, and that $\langle f\rangle_q\in M$ precisely if $\langle f_i\rangle_q=0$ for all $i>0$.  
\end{remark}
\begin{proof}
By~(\ref{eq:mathscrD}) and by construction of~$\mathscr{M}$ one has that
\[ \frac{1}{s!}\bigl\langle \mathscr{M}^{r-s}\mathscr{D}^sf_{\vec{\ell}} \bigr\rangle_{\! q} \=  {g}_{\vec{\ell},s}^{r-s}(\phi)\, .\]
Hence,
\[\langle \pi(f_{\vec{\ell}}) \rangle_q \= \xi_{\vec{\ell}}(\phi),\]
where~$\phi$ is the meromorphic Jacobi form $\phi(\vec{z})=\phi_n^\Phi(\vec{z})=\langle \prod_{i=1}^n \Phi(\vec{z}^i)\rangle_q\mspace{1mu}$. Hence, $\pi(f)$ is modular under the~$q$-bracket for all $f\in \F$. 

Choose $\mathcal{M}^0\subseteq \mathcal{F}$ such that $\F=\mathcal{M}^0\oplus Q_2\F$. Then, we take $\mathcal{M}=\pi\mathcal{M}^0$. As, by definition, $\pi(f)-f$ is a multiple of $Q_2$, the first property follows. The second property is immediate as $\mathcal{M}\subseteq \pi(\F)$. For the last property, let $f\in Q_2\F$ with $\langle f \rangle_q\in M$ be given. 
As~$f$ is divisible by~$Q_2$, the $q$-bracket~$\langle f\rangle_q$ is in the image of~$D+\mathbbm{e}_2$ acting on quasimodular forms. Now, the zero function is the only function in the image of~$D+\mathbbm{e}_2$ which is modular, so that the last property follows. 
\end{proof}

We now provide several examples of quasimodular algebras to which \cref{thm:when+} applies. 																											

\subsection{First example: the Bloch--Okounkov algebra}
To apply the results of the previous section to the Bloch--Okounkov algebra~$\Lambda^*$, we have to understand how the operators~$\dtau$ and~$\dd_{z_i}$ act on the~$n$-points functions~$F_n$ (defined by \cref{defn:BOnpoint}), or equivalently, we have to understand the transformation behaviour of~$F_n \mspace{1mu}.$ This behaviour is uniquely determined by the following two properties.
\begin{prop}\label{prop:FnQJF} For all~$n\geq 1$ one has
\begin{align}
\dtau F_n(z_1,\ldots,z_n) &\= 0, \\
\dd_{z_1} F_n(z_1,\ldots,z_n) &\= \sum_{i=2}^n F_{n-1}(z_1+z_i,z_2,z_3,\ldots,z_{i-1},z_{i+1},z_{i+2},\ldots,z_n).
\end{align}
\end{prop}
\begin{remark}
As $F_n$ is symmetric in its arguments, above proposition provides an expression for $\dd_{z_j} F_n(z_1,\ldots,z_n)$ for all $j$. 
\end{remark}
\begin{proof}
The second equality is equivalent to \cite[Theorem~0.6]{BO00}, whereas the first statement seems not to be in the literature. As both statements follow by more or less the same argument, we give both proofs. That is, we prove
\begin{align}
\dtau V_n(z_1,\ldots,z_n) &\= 0, \\
\label{eq:ind2} \dd_{z_i} V_n(z_1,\ldots,z_n) &\= \begin{cases} V_{n-1}(z_1,\ldots,z_i+z_{i+1},\ldots,z_{n})  & i<n \\ 0 & i=n \end{cases}
\end{align}
inductively using the recursion~(\ref{eq:B-Ocor}), from which the proposition follows directly. 

For $n=1$ both statements are clearly true. Hence, by the identity 
\[[\dtau,D_z^m] \= -2m D_z^{m-1}\dd_z \+ m(m-1)D_z^{m-2}\]
and after assuming that $\dtau V_{n}=0$, we find that $\theta(z_1+\ldots+z_{n+1})\,\dtau V_{n+1}(\vec{z})$ equals
\begin{align}\label{eq:dtau}
 -\sum_{m=0}^{n-1} \frac{(-1)^{n-1-m}}{(n-1-m)!}\;\theta^{(n-1-m)}(z_{1}+\ldots+z_{m})\;V_m(z_{1},\ldots,z_{m}).	
\end{align}
By the recursion~(\ref{eq:B-Ocor}) one finds that this expression vanishes. Hence, $\dtau V_{n+1}=0$ and $\dtau F_n=0$ as desired. 

Denote 
\[V_{m-1}^i(z_1,\ldots,z_{m} ) \= V_{m-1}(z_1,\ldots, z_i+z_{i+1},\ldots,z_{m}).\]
By applying $\dd_{z_i}$ to the recursion~\eqref{eq:B-Ocor}, using the identity  $[\dd_z,D_z^m] = -2m D_z^{m-1} I$ (with $I$ the index operator; see \cref{sec:j-alg}) and assuming that~(\ref{eq:ind2}) holds, we find
\begin{align}
\theta(z_1+\ldots &+z_{n+1})\,\dd_{z_i}V_{n+1}(z_1,\ldots,z_{n+1}) \\
\=&  \sum_{m=i}^{n} \frac{(-1)^{n-m}}{(n-m)!}\;\theta^{(n-m)}(z_1+\ldots+z_m)\;V_m(z_1,\ldots,z_m) \+ \\
&\, - \sum_{m=i+1}^{n} \frac{(-1)^{n+1-m}}{(n+1-m)!}\;\theta^{(n+1-m)}(z_1+\ldots+z_m)\;V_{m-1}^i(z_1,\ldots,z_{m} ) \\
\=& -\sum_{m=0}^{i-1} \frac{(-1)^{n-m}}{(n-m)!}\;\theta^{(n-m)}(z_1+\ldots+z_m)\;V_m(z_1,\ldots,z_m) \+ \\
&\,- \sum_{m=i+1}^{n} \frac{(-1)^{n+1-m}}{(n+1-m)!}\;\theta^{(n+1-m)}(z_1+\ldots+z_m)\;V_{m-1}^i(z_1,\ldots,z_{m} ) \\
\=&\theta(z_1+\ldots+z_{n+1})\;V_n(z_1,\ldots,z_i+z_{i+1},\ldots,z_{n+1}). \qedhere
\end{align}
\end{proof}

Next, recall the~$j$th order differential operators~$\mathscr{D}_j$ in \cite{vI18}.
\begin{defn}\label{def:Dop}
Define the~$j$th order differential operators~$\mathscr{D}_j$ by
\[
\mathscr{D}_j \defis \sum_{\vec{i}\in \z_{\geq 0}^{j}} \binom{|\vec{i}|}{i_1,i_2,\ldots,i_{j}}\, Q_{|\vec{i}|}\,\partial_{\vec{i}}\,,\quad \text{with} \qquad
\partial_{\vec{i}} \defis \frac{\partial^j}{\partial Q_{i_1+1}\, \partial Q_{i_j+1}\cdots \partial Q_{i_j+1}},
\]
where the coefficient is a multinomial coefficient (in this section we pretend these operators act on~$\Lambda^*$, although formally there are only defined on the formal algebra $\mathcal{R}=\q[Q_1,Q_2,\ldots]$ freely generated by the variables $Q_1\mspace{1mu},Q_2\mspace{1mu},\ldots$, which admits a canonical mapping to $\Lambda^*$). 
\end{defn}

These operators turn out to correspond to certain symmetric powers of the derivative operators~$\dd_{z_i}$. In particular, observe that the coefficient of~$\vec{z}^{\vec{\ell}}$ in the case $j=1$ below, is given by~${g}_{\vec{\ell},1}(F_n) \mspace{1mu}.$ Hence, the operator $\mathscr{D}$, requested in the previous section, is given by $\mathscr{D}=\mathscr{D}_2/2$.
\begin{prop}\label{prop:operator} For all $j\geq 1$ one has
\[ \langle \mathscr{D}_j W(\vec{z}) \rangle_q  \=j!\left(z_1\dd_{z_1}^{j-1} + \ldots + z_n\dd_{z_n}^{j-1}\right) F_n(\vec{z}).\]
\end{prop}
\begin{proof}
Observe that
\[ \mathscr{D}_{j} Q_{\vec{\ell}} \= j!\! \sum_{i_1<\ldots<i_j} \binom{\ell_{i_1}+\ldots+\ell_{i_j}-j}{\ell_{i_1}-1,\ldots,\ell_{i_j}-1}\,Q_{\ell_{i_1}+\ldots+\ell_{i_j}-j} \cdots \widehat{Q}_{l_{i_1}}\cdots \widehat{Q}_{l_{i_j}} \cdots, \]
where the binomial coefficient vanishes whenever $l_i=0$ for some~$i$, and we wrote a hat above an element to indicate that this element is omitted. Hence, 
\[ \mathscr{D}_j W_\lambda(\vec{z}) \= j!\! \sum_{i_1<\ldots<i_j}\, (z_{i_1}+\ldots+z_{i_j})\, W_\lambda(z_{i_1}+\ldots+z_{i_j},\ldots,\hat{z}_{i_1},\ldots,\hat{z}_{i_j},\ldots),\]
where $\ldots,\hat{z}_{i_1},\ldots,\hat{z}_{i_j},\ldots$ stands for the elements $z_1,\ldots,z_n$ omitting $z_{i_1},\ldots,z_{i_j}\mspace{1mu}$. 
By symmetry of~$F_n$ and \cref{prop:FnQJF} the statement follows. 
\end{proof}

Now, by invoking \cref{thm:when+}, the following refinement of \cref{thm:when} follows. In particular, we find a different expression for the basis elements~$h_{\lambda}$ in \cite{vI18}, which equal $\pi(Q_{\lambda_1}\ldots Q_{\lambda_n})$. Moreover, by construction, we deduce that the~$q$-bracket of~$h_\lambda$ is given by $\xi_{\lambda^+}(F_n)$, where $\lambda^+=(\lambda_1+1,\lambda_2+1,\ldots,\lambda_n+1)$. 

Define the mapping~$\vee$ as the algebra homomorphism uniquely determined by $Q_n\mapsto \Delta_n\mspace{1mu}$, where the commuting family of operators~$\Delta_{k}$ of {\cite{vI18}} are given by
\[ \Delta_n \defis \sum_{i=0}^n (-1)^i \binom{n}{i} \partial^i\mathscr{D}_{n-i}\,  .\]
We write $(x)_n^{-}=x(x-1)\cdots(x-n+1)$ for the falling factorial. 
\begin{thm}\label{thm:when-BO}
Let $\partial=\mathscr{D}_1$ and $\mathscr{D}_2$ be given by \cref{def:Dop}. The linear mapping $\pi\colon\Lambda^*\to \Lambda^*$ given by
\[\pi(f)\=\sum_{r\geq 0}\sum_{s=0}^r (-1)^s\frac{Q_2^r \,\partial^{2r-2s}\mathscr{D}_2^s(f) }{2^{r}(\ell-r-\frac{3}{2})_r\,(r-s)!\,s!}\]
whenever $f$ is of weight~$\ell$ is a projection satisfying $\pi(Q_2\Lambda^*)=0$ and
\begin{align}\label{eq:pi} \pi(f) = \frac{Q_2^{-3/2+\ell} f^\vee \, Q_2^{3/2}}{(\frac{3}{2})_{\ell}^{-}\, \ell!}.\end{align}
 Furthermore, $\mathcal{M}:=\pi(\Lambda^*)$ satisfies
\begin{enumerate}[{\upshape (i)}]\itemsep2pt
\item $\Lambda^* = \mathcal{M}\oplus (Q_2)$, where $(Q_2)=Q_2\Lambda^*;$
\item $\langle \mathcal{M}\rangle_q \subseteq \modular;$
\item $\langle (Q_2)\rangle_q \cap \modular = \{0\};$
\item $\mathcal{M}=\mathcal{H}$, where $\mathcal{H}$ is the harmonic subspace of~$\Lambda^*$ as in \upshape\cite{vI18}.
\end{enumerate}
\end{thm}
\begin{proof}
Equation~\ref{eq:pi}, as well as the fact that~$\pi$ is a projection, follows directly from \cite[
Corollary~2]{vI18}. In particular, as $\Delta_2(Q_2^{3/2})=0$, it follows that $\pi(Q_2f)=0$. 

For the properties of $\mathcal{M}$, observe that~$\mathscr{D}$ and~$\mathscr{M}$, defined in the previous section, can also be expressed as~$\mathscr{D}_2/2$ and~$-\partial^2/2$ respectively. Hence, the properties follow from \cref{thm:when+}, and the fourth follows as the set $\{\pi(Q_\lambda)\}$, where $\lambda$ goes over all partitions with all parts at least~$3$, is a basis for both spaces. 
\end{proof}

\subsection{Second example: the Bloch--Okounkov algebra of higher level}
To extend the result in the previous section to the Bloch--Okounkov algebras~$\Lambda^*(N)$ of level~$N$ (see \eqref{eq:LambdaN} and \cref{sec:BON}), we should generalise the operators~$\partial$ and~$\mathscr{D}_2\mspace{1mu}$. 
\begin{defn}\label{eq:DlevelN} Let $\widehat{\mathcal{R}}=\q[Q_k(a) \mid k\in \z_{\geq 0}\mspace{1mu}, a\in \q]$ be the algebra in the formal variables $Q_k(a)$ with canonical projection to $\bigcup_{N\in \z}\Lambda^*(N)$. Given $\vec{a}\in \q^j$, define the~$j$th order differential operators~$\mathscr{D}_j$ on $\hat{\mathcal{R}}$ by									 
\begin{align}
\mathscr{D}_j &\defis \sum_{\vec{i}\in \z_{\geq 0}^{j}}\sum_{\vec{a}\in \q^j} \binom{|\vec{i}|}{i_1,i_2,\ldots,i_{j}}\, Q_{|\vec{i}|}(|\vec{a}|)\,\partial_{\vec{i}}(\vec{a}),
\end{align}
where
\begin{align}
\partial_{\vec{i}}(\vec{a}) &\defis \frac{\partial^j}{\partial Q_{i_1+1}(a_1) \cdots \partial Q_{i_n+1}(a_n)}.
\end{align}
\end{defn}
From now on we pretend that these operators act on~$\Lambda^*(N)$, by identifying $\Lambda^*(N)$ with a quotient of $\hat{\mathcal{R}}$ via the obvious inclusion map. Note that restricted to~$\Lambda^*$ the operators~$\mathscr{D}_j$ are the same as defined in \cref{def:Dop}. Similarly, the operators~$\mathscr{D}_j$ satisfy the following property. 
\begin{prop}\label{prop:operatorN} For all $j\geq 1$ and $\vec{a}\in \q^n$ one has
\[ \langle \mathscr{D}_j W(\vec{z}+\vec{a}) \rangle_q  \=j!\left((z_1+a_1)\,\dd_{z_1}^{j-1} + \ldots + (z_n+a_n)\,\dd_{z_n}^{j-1}\right) F_n(\vec{z}+\vec{a}).\]
\end{prop}

Denote by $\modular(N)$ the algebra of modular forms of level~$N$. Then, \cref{thm:when+} specializes to the following result. 
\begin{thm} Let $N\geq 1$ and $\partial=\mathscr{D}_1$ and $\mathscr{D}_2\colon\Lambda^*(N)\to\Lambda^*(N)$ be given by Equation~\ref{eq:DlevelN}. Let the projection~$\pi\colon\Lambda^*(N)\to \Lambda^*(N)$ be given by													
\[ \pi(f)\=\sum_{r\geq 0}\sum_{s=0}^r (-1)^s \frac{Q_2^r \,\partial^{2r-2s}\mathscr{D}_2^s(f) }{2^{r}(\ell-r-\frac{3}{2})_r\,(r-s)!\,s!}\, ,  \]
whenever~$f\in \Lambda^*(N)$ is homogeneous of weight~$\ell$. Then, the subspace $\mathcal{M}(N):=\pi(\Lambda^*(N))$ of~$\Lambda^*(N)$ satisfies
\begin{enumerate}[{\upshape (i)}]\itemsep2pt
\item $\displaystyle\Lambda^*(N) = \mathcal{M}(N) \oplus Q_2\,\Lambda^*(N);$
\item $\langle \mathcal{M}(N)\rangle_q \subseteq \modular(N);$
\item $\langle Q_2\,\Lambda^*(N)\rangle_q \cap \modular(N) = \{0\}.$
\end{enumerate}
\end{thm}

\subsection{Third example: double moment functions}
For the algebra of double moments functions (see \cref{sec:double}) the~$n$-point functions with respect to the induced product~$\blok$ (see~(\ref{eq:blok})) are given by \eqref{eq:gn}, i.e.,
\[G_n(\vec{z},\vec{w})\=\prod_{i=1}^n\frac{\Theta(z_i+w_i)}{\Theta(z_i)\Theta(w_i)},\]
which is a Jacobi form. Hence, the operators~$\dtau, \dd_{z_i}$ and~$\dd_{w_i}$ vanish acting on~$G_n\mspace{1mu}$, so that~$\mathscr{D}$ can taken to be the zero operator. We write~$\dd$ for the derivation on~$\mathcal{T}$ with respect to the induced product (i.e., $\dd(f\blok g) = \dd(f)\blok g + f \blok \dd(g)$ for all $f,g\in \mathcal{T}$) given by
\[\dd(T_{k,\ell})\defis \begin{cases}
k(\ell-1)T_{k-1,\ell-1} & k\geq 1, \ell\geq 2,\\
-\frac{1}{2} & k+\ell=2,\\
0 & \text{else.}
\end{cases} \]
The notation~$\dd$ is suggested by the fact that $\langle \dd f\rangle_q = \dtau \langle f \rangle_q$ for all $f\in \mathcal{T}$. 
\begin{thm}
Let the projection~$\pi\colon\mathcal{T}\to\mathcal{T}$ be given by
\[ \pi(f) \= \sum_{r\geq 0}
\frac{2^r}{r!}\overbrace{T_{1,1}\blok\cdots\blok T_{1,1}}^{r} \blok \, \dd^{r}(f).\]
Then $\mathcal{M}=\pi(\mathcal{T})$ satisfies the following three properties:
\begin{enumerate}[{\upshape (i)}]\itemsep2pt
																				  
\item $\mathcal{T} = \mathcal{M} \oplus (T_{1,1})$, where $(T_{1,1})=T_{1,1}\blok\mathcal{T};$
\item \label{it:tprop} $\langle \mathcal{M}\rangle_q = \modular;$
\item $\langle T_{1,1}\blok\mathcal{T}\rangle_q \cap \modular = \{0\}.$
\end{enumerate}
\end{thm}
\begin{proof}
The statement follows along the same lines as \cref{thm:when+}, by making the following observations:
\begin{itemize}\itemsep0pt
\item Analogous to~$\xi_{\vec{\ell}}(\phi)$, the functions
\[\sum_{r\leq |\vec{\ell}|}(-1)^r\sum_{s\leq r}\frac{\mathbbm{e}_2^r\,{g}_{\vec{\ell},s}^{r-s}(\phi)}{(r-s)!}\]
are modular forms exactly if~$\phi$ is a quasi-Jacobi form;
\item $\langle T_{1,1}\blok f\rangle_q = -2\mathbbm{e}_2 \langle f \rangle_q$ for all $f\in \c^\partitions$;
\item The operator~$\dd$ coincides with the operator~$\mathscr{M}$;
\item By \cite[Theorem~3.4.1]{vI20} we have that $\langle \mathcal{T}\rangle_q = \widetilde{M}$, from which it follows that equality holds in (\ref{it:tprop}). 
\qedhere
\end{itemize}
\end{proof}
\begin{remark}
In fact, for all $f,g\in\mathcal{T}$ one has 
\[\pi(f\blok g)=\pi(f)\blok\pi(g).\]
 Hence,~$\langle \pi\mathcal{T}\rangle_q$ is uniquely determined by $\langle \pi(T_{2,0})\rangle_q = \langle \pi(T_{1,1})\rangle_q=0$ and 
\[\langle \pi(T_{k,l}) \rangle_q \= \begin{cases} \vartheta^{k-1}G_{l-k+2} & l\geq k \\ \vartheta^{l} G_{k-l} & k\geq l+2 \end{cases}\]
for $T_{k,l}\in \mathcal{T}$ with $k+l\geq 4$, where~$\vartheta:=D_\tau  - \mathbbm{e}_2 W$ denotes the Serre derivative on the space of (quasi)modular forms (recall $W$ is the operator multiplying a quasimodular form by its weight) and $G_k=\frac{(k-1)!}{2(2\pi\ii)^k} e_k\mspace{1mu}$. For any modular form~$f$, the Serre derivative $\vartheta f$ is modular as well. 
In particular, the Serre derivatives of Eisenstein series appearing on the right of this equation are indeed modular forms. Moreover, the case distinction according to the sign of $l+1-k$, should be compared to the Taylor coefficients of the Jacobi form~$G_1(z,w)$ in \cite{Zag91}. In fact, they are very similar (but here in level~$1$ and there in level~$2$, and here with Serre derivatives and there with usual derivatives) to the ones that appeared in \cite{KZ95} to prove the original assertion of Dijkgraaf from which the whole Bloch--Okounkov story arose. 
\end{remark}

\small%

\end{document}